\documentclass[
	a4paper,
	10pt,
	oneside
]{article}

\usepackage{amssymb,amsfonts,amsmath,amsthm}
\numberwithin{equation}{section}
\usepackage{mathtools,dsfont}
\usepackage{paralist,enumitem,bm,placeins,url}
\usepackage{array,color}
\usepackage{cite}

\usepackage[%
left=1in,
right=1in,
bottom=1in,
top=1in,
footskip=0.5in,
]{geometry}

\usepackage{tocbasic}
\DeclareTOCStyleEntry[
beforeskip=.2em plus 1pt,
]{tocline}{section}

\definecolor{FrameColor}{rgb}{0.85,0.85,0.85}

\newtheorem{theorem}{Theorem}[section]

\newtheorem{lemma}[theorem]{Lemma}
\newtheorem{proposition}[theorem]{Proposition}
\newtheorem{corollary}[theorem]{Corollary}
\newtheorem{remark}[theorem]{Remark}
\newtheorem{definition}[theorem]{Definition}

\definecolor{rosso}{rgb}{0.8,0,0}
\definecolor{LinkColor}{rgb}{0,0,1}
\definecolor{LinkColor2}{rgb}{0,0.5,0}
\definecolor{lg}{rgb}{.25,.25,.25}
\usepackage{hyperref} 
\makeatletter
\hypersetup{%
	colorlinks	=true,%
	linkcolor	=LinkColor,%
	citecolor	=LinkColor2,%
	urlcolor	=LinkColor,%
}
\makeatother

\newcommand{\Om}{\Omega}
\newcommand{\Ga}{\Gamma}

\newcommand{\Si}{{\Sigma}}

\newcommand{\LL}{\mathcal{L}}
\renewcommand{\SS}{\mathfrak{S}}

\newcommand{\N}{\mathbb{N}}
\newcommand{\R}{\mathbb{R}}

\newcommand{\dG}{\, \mathrm d\Gamma}

\newcommand{\dd}{\mathrm d}
\newcommand{\ds}{\, \mathrm ds}
\newcommand{\dx}{\, \mathrm dx}

\newcommand{\pd}{\partial}

\newcommand{\pdnu}{\pd_{\bm{n}}}

\newcommand{\abs}[1]{\left| #1 \right|}
\newcommand{\sm}[1]{[ \hspace{1pt} #1 \hspace{1pt} ]}
\newcommand{\norm}[1]{\| #1 \|}
\newcommand{\bignorm}[1]{\big\| #1 \big\|}
\newcommand{\inn}[2]{ \langle #1 , #2  \rangle}
\newcommand{\biginn}[2]{ \big< #1 , #2  \big>}
\newcommand{\scp}[2]{ \left( #1 , #2  \right)}
\newcommand{\bigscp}[2]{\big( #1 , #2 \big)}
\newcommand{\mean}[1]{\langle #1 \rangle}

\newcommand{\eps}{\varepsilon}
\newcommand{\Lap}{\Delta}
\newcommand{\Lapg}{\Delta_{\Gamma}}

\newcommand{\n}{\mathbf{n}}
\newcommand{\grad}{\nabla}
\newcommand{\gradg}{\nabla_\Gamma}
\newcommand{\mo}{m_\Omega}
\newcommand{\mg}{m_\Gamma}
\newcommand{\dtau}{\, \mathrm d\tau}

\newcommand{\del}{\partial}
\newcommand{\delt}{\partial_t}

\newcommand{\deln}{\partial_\n}

\newcommand{\VV}{{\mathcal{V}}}
\newcommand{\Vk}{{\mathcal{V}^k}}

\newcommand{\WW}{\mathcal{W}}
\newcommand{\Wm}{{\mathcal{W}_{\beta,m}^1}}
\newcommand{\Wmt}{{\mathcal{W}_{\beta,m}^2}}
\newcommand{\Wmk}{{\mathcal{W}_{\beta,m}^k}}
\newcommand{\Wo}{{\mathcal{W}_{\beta,0}^1}}
\newcommand{\Wot}{{\mathcal{W}_{\beta,0}^2}}

\newcommand{\DD}{\mathcal{D}}
\newcommand{\Db}{{\mathcal{D}_{\beta}}}

\renewcommand{\LL}{\mathcal{L}}

\newcommand{\HH}{\mathcal{H}}
\newcommand{\HLB}{\HH^1_{L,\beta}}

\newcommand{\intO}{\int_\Omega}
\newcommand{\intG}{\int_\Gamma}

\newcommand{\wto}{\rightharpoonup}

\newcommand{\emb}{\hookrightarrow}

\newcommand{\RP}{\mathbb{R}_0^+}

\newcommand{\mom}{m_\Om}
\newcommand{\mga}{m_\Ga}

\newcommand{\suchthat}{\;\ifnum\currentgrouptype=16 \middle\fi|\;}

\newcommand{\dist}{\mathrm{dist}}

\newcommand{\Index}[2][0pt]{%
	\raisebox{#1}{\scriptsize\ensuremath{#2}}
}

\makeatletter
\renewenvironment{proof}[1][\proofname]{%
	\par\pushQED{\qed}\normalfont%
	\topsep6\p@\@plus6\p@\relax
	\trivlist\item[\hskip\labelsep\bfseries#1\@addpunct{.}]%
	\ignorespaces
}{%
	\popQED\endtrivlist\@endpefalse
}
\makeatother

\makeatletter
\renewcommand\paragraph{\@startsection{paragraph}{4}{\z@}%
	{1ex \@plus1ex \@minus.2ex}%
	{-1em}%
	{\normalfont\normalsize\bfseries}}
\renewcommand\subparagraph{\@startsection{paragraph}{4}{\z@}%
	{1ex \@plus1ex \@minus.2ex}%
	{-1em}%
	{\normalfont\normalsize\itshape}}
\makeatother

\begin{document}
	
	\title{\sc Long-time dynamics of the Cahn--Hilliard equation with kinetic rate dependent\\ dynamic boundary conditions}
	
	\author{Harald Garcke \footnotemark[1] \and Patrik Knopf \footnotemark[1] \and Sema Yayla \footnotemark[2]}
	
	\renewcommand{\thefootnote}{\fnsymbol{footnote}}
	
	\footnotetext[1]{Fakult\"at f\"ur Mathematik, Universit\"at Regensburg, 93053 Regensburg, Germany,  \tt(\href{mailto:Harald Garcke@ur.de}{harald.garcke@ur.de}, \href{mailto:Patrik.Knopf@ur.de}{patrik.knopf@ur.de}) }
	\footnotetext[2]{Department of Mathematics, Faculty of Science, Hacettepe University, Beytepe 06800, Ankara, Turkey 
		\tt(\href{mailto:semasimsek@hacettepe.edu.tr}{semasimsek@hacettepe.edu.tr}) }	
	
	\date{}
	\maketitle
	
	\begin{center}
		\small
		{
			\textit{This is a preprint version of the paper. Please cite as:} \\  
			H. Garcke, P. Knopf, S. Yayla. Nonlinear Analysis, 215:112619, 2022  \\ 
			\url{https://doi.org/10.1016/j.na.2021.112619}
		}
	\end{center}

\begin{abstract}
	We consider a Cahn--Hilliard model with kinetic rate dependent dynamic boundary conditions that was introduced by Knopf, Lam, Liu and Metzger (ESAIM Math.~Model.~Numer.~Anal., 2021) and will thus be called the KLLM model.
	In the aforementioned paper, it was shown that solutions of the KLLM model converge to solutions of the GMS model proposed by Goldstein, Miranville and Schimperna (Physica D, 2011) as the kinetic rate tends to infinity.  
	We first collect the weak well-posedness results for both models and we establish some further essential properties of the weak solutions. Afterwards, we investigate the long-time behavior of the KLLM model. We first prove the existence of a global attractor as well as convergence to a single stationary point. Then, we show that the global attractor of the GMS model is stable with respect to perturbations of the kinetic rate. Eventually, we construct exponential attractors for both models, and we show that the exponential attractor associated with the GMS model is robust against kinetic rate perturbations.
\end{abstract}

\begin{small}
\noindent{\bf Keywords:}  Cahn-Hilliard equation, dynamic boundary conditions, long-time dynamics, stability of global attractors, robustness of exponential attractors. \\[1ex]
\noindent{{\bf 2010 MSC classification:}  
	35B40, 
	35B41, 
	35K35  
	35K61  
	35Q92, 
	37L30, 
}
\end{small}

%
%

\bigskip
\setlength\parindent{0ex}
\setlength\parskip{1ex}
\allowdisplaybreaks

\section{Introduction}

The Cahn--Hilliard equation was originally introduced in \cite{cahn-hilliard} to describe spinodal decomposition in binary alloys.  
Meanwhile, it has become one of the most popular models to describe various kinds of phase separation phenomena arising, for instance, in materials science, life sciences and image processing. The standard Cahn--Hilliard equation as proposed in \cite{cahn-hilliard} reads as follows:
\begin{subequations}
	\label{CH}
	\begin{alignat}{2}
	\label{CH:1}
	&\delt u = \mom \Lap \mu &&\quad\text{in } Q := \Omega\times (0,\infty),\\
	\label{CH:2}
	&\mu = -\eps \Lap u + \eps^{-1} F'(u) &&\quad\text{in } Q,\\
	\label{CH:3}
	&u\vert_{t=0}=u_0 &&\quad\text{in } \Omega.
	\end{alignat}
\end{subequations}
Here, $\Omega\subset\R^d$ stands for a bounded domain (usually $d\in\{2,3\}$) with boundary $\Gamma$, $\Lap$ stands for the Laplace operator acting on $\Omega$, and $\mo$ denotes a mobility parameter. For simplicity, it is assumed to be a non-negative constant, although non-constant mobilities find a use in some situations (see, e.g., \cite{elliotgarcke}). 

In order to describe a binary mixture, the phase-field variable $u$ represents the difference in volume fractions of both materials. 
After a short period of time, the solution $u$ will attain values
close to $\pm 1$ in most parts of the domain $\Omega$. These regions, which correspond to the pure phases of the materials, are 
separated by a diffuse interface whose thickness is proportional to the parameter $\eps>0$ appearing in \eqref{CH:2}. In most applications this interface will be very thin and thus, the parameter $\eps$ is usually chosen to be very small.
The time evolution of the mixture described by $u$ is driven by the chemical potential $\mu$ in the bulk (i.e., in $\Omega$).
It is given as the derivative of the following Ginzburg--Landau type free energy:
\begin{align}
E_\text{bulk}(u) = \int_\Omega \frac \eps 2|\grad u|^2 + \frac 1 \eps F(u) \dx.
\end{align}
The bulk potential $F$ is usually double-well shaped. Typically, it attains its minimum at $-1$ and $1$ and has a local maximum in between at $0$. From a mathematical point of view, $F$ is responsible for the phase separation since it is energetically favorable for $u$ to attain values close to $\pm 1$.
A physically relevant choice is the \emph{logarithmic potential}
\begin{align}
\label{POT:LOG}
F_\text{log}(s)= \frac\vartheta 2 \big((1+s)\ln(1+s) + (1-s)\ln(1-s)\big) 
- \frac{\vartheta_c}{2} s^2, \quad s\in (-1,1)
\end{align}
with constants $0<\vartheta<\vartheta_c$.
It is often approximated by the \emph{smooth double-well potential}
\begin{align}
\label{POT:SMOOTH}
F_\text{sdw}(s) = \frac{1}{4}(s^2-1)^2, \quad s\in\R
\end{align}
which is easier to handle in terms of mathematical analysis. In this paper, we will deal with general smooth potentials satisfying certain polynomial growth conditions such that \eqref{POT:SMOOTH} fits into our setting. However, singular potentials like \eqref{POT:LOG} cannot be taken into account in our approach.

To ensure well-posedness of the system \eqref{CH}, certain boundary conditions on $u$ and $\mu$ need to be imposed. The classical choices are the homogeneous Neumann conditions
\begin{alignat}{2}
\label{HNC:1}
\del_\n u &= 0 &&\quad\text{on}\;\; \Sigma:=\Gamma\times(0,\infty),\\
\label{HNC:2}
\del_\n \mu &= 0 &&\quad\text{on}\;\; \Sigma.
\end{alignat}
The \emph{no-flux condition} $\eqref{HNC:2}$ implies mass conservation in the bulk, meaning that (sufficiently regular) solutions satisfy
\begin{align}
\int_\Omega u(t) \dx = \int_\Omega u(0) \dx, \quad t\in [0,\infty).
\end{align}
Moreover, due to \eqref{HNC:1} and \eqref{HNC:2}, we obtain the following energy dissipation law:
\begin{align}
\frac{d}{dt} E_\text{bulk}\big(u(t)\big) + \mom \intO |\grad\mu(t)|^2 \dx = 0, 
\quad t\in  [0,\infty).
\end{align}
Furthermore, $\eqref{HNC:1}$ can be regarded as a \emph{contact angle condition} as it enforces the diffuse interface separating the regions of pure materials to intersect the boundary at a perfect right angle. 
The Cahn--Hilliard equation \eqref{CH} with the homogeneous Neumann conditions \eqref{HNC:1} and \eqref{HNC:2} is already very well understood and has been studied from many different viewpoints (see, e.g., 
\cite{Abels-Wilke,Bates-Fife,Cherfils,elliotgarcke,elliotzheng}). In particular, we refer to \cite{zheng,rybka,miranville-lt,GY,EMZ} for the investigation of long-time behavior. 

However, in many situations the contact angle condition \eqref{HNC:1} turned out to be very restrictive as in many applications the contact angle of the interface will not only deviate from ninety degrees but also change dynamically over the course of time.
In certain situations (e.g., in hydrodynamic applications), it also turned out to be essential to model short-range interactions between the mixture of materials and the solid wall of the container more precisely. To this end physicists (see \cite{Fis1,Fis2,Kenzler}) proposed that the total free energy should contain an additional contribution on the surface being also of Ginzburg--Landau type:
\begin{equation}\label{DEF:ENS}
E_\text{surf}(u)  = \intG \frac{\kappa \delta }{2} \abs{\gradg u}^2 + \frac{1}{\delta} G(u) 
\dG.
\end{equation} 
Here, $\gradg$ denotes the surface gradient operator, the constant $\kappa\ge 0$ acts as a weight for surface diffusion effects and the parameter $\delta>0$ is related to the thickness of the diffuse interface on the boundary.
Moreover, the function $G$ is an additional surface potential. If phase separation processes are expected to also occur on the boundary it makes sense to assume that $G$ exhibits a double-well structure similar to $F$. 
In this paper, we will thus impose similar conditions on $G$ as on $F$ such that the choice $G=F_\mathrm{sdw}$ is admissible.
The total free energy $E=E_\text{bulk}+E_\text{surf}$ then reads as
\begin{equation}\label{DEF:EN}
E(u) 
= \int_\Omega \frac \eps 2|\grad u|^2 + \frac 1 \eps F(u) \dx
+ \intG \frac{\kappa \delta }{2} \abs{\gradg u}^2 + \frac{1}{\delta} G(u) \dG.
\end{equation} 

Associated with this total free energy, various Cahn-Hilliard type systems with dynamic boundary conditions have been proposed and investigated in the literature (see,e.g.,
\cite{colli-fukao-ch,colli-gilardi,Gal1,GalWu,colli-gilardi-sprekels,liero,mininni,miranville-zelik,motoda,racke-zheng,WZ, FukaoWu, Wu,GalGra, CGG, CGW, CFW, MW}). 
In recent times, dynamic boundary conditions which also exhibit a Cahn--Hilliard type structure have become very popular.
Therefore, we want to highlight the Cahn--Hilliard equation subject to a class of dynamic boundary conditions of Cahn--Hilliard type depending on a parameter $L\in[0,\infty]$:
\begin{subequations}\label{CH:INT}
	\begin{alignat}{3}
	& \delt u = \mo \Delta\mu, 
	&&\mu = - \eps \Lap  u + \eps^{-1} F'( u) 
	&&\quad \text{in } Q, \label{CH:INT:1}\\
	& \delt v =  \mg \Lapg  \theta - \beta \mo \pdnu\mu,\quad 
	&&\theta = - \delta \kappa\Lapg v + \delta^{-1} G'(v) + \eps \pdnu  u 
	&&\quad \text{on } \Sigma, \label{CH:INT:2}\\
	& u\vert_\Sigma = v 
	&&&&\quad \text{on } \Sigma, \label{CH:INT:3}\\[1ex]
	&\left\{
	\begin{aligned}
	\mu\vert_\Si &= \beta  \theta \\
	L \pdnu\mu\vert_\Si &= \beta  \theta -  \mu\vert_\Si \\
	\pdnu\mu\vert_\Si &= 0  \\
	\end{aligned}
	\right.
	&&
	\begin{aligned}
	&\text{if}\; L=0,\\
	&\text{if}\; L\in(0,\infty),\\
	&\text{if}\; L=\infty,\\
	\end{aligned}
	&&\quad \text{on } \Sigma, \label{CH:INT:4} \\[1ex]
	& (u,v)\vert_{t=0} = (u_0,v_0) 
	&&\text{with}\; u_0\vert_\Si = v_0
	&&\quad \text{in } \Omega \times \Gamma. \label{CH:INT:5}
	\end{alignat}
\end{subequations}
The chemical potentials $\mu$ and $\theta$ which are coupled by the boundary condition $\eqref{CH:INT:4}$ describe the chemical interaction between the materials in the bulk and the materials on the surface. The constant $1/L$ is related to a \emph{kinetic rate}, and the term $L\deln\mu$ describes adsorption and desorption processes. Here, the mass flux $-\mo \deln \mu$ (which describes the motion of the materials towards and away from the boundary) is directly influenced by differences in the chemical potentials through the condition \eqref{CH:INT:4}. As the models corresponding to the cases $L=0$, $L=\infty$ and $0<L<\infty$ were introduced separately in different articles in the literature, we are now going to briefly highlight their most important features.

\paragraph{The case $L=0$ (GMS model).} The system \eqref{CH:INT} was first introduced with $L=0$ in \eqref{CH:INT:4} by G.~Goldstein, A.~Miranville and G.~Schimperna \cite{GMS}. In view of the authors' initials we will refer to this system as the \textit{GMS model}.

It can be regarded as an extension of a model that was previously introduced by Gal \cite{Gal1} where the equation $ \delt v = -\beta \pdnu \mu + \gamma \mu $ on $\Sigma$ (for some additional constant $\gamma$) was proposed instead of $\eqref{CH:INT:2}$.
Here, the chemical potentials in the bulk and on the boundary can differ only by the factor $\beta$, i.e., they are directly proportional.
In other words, this means that the chemical potentials $\mu$ and $\theta$ are always in chemical equilibrium. 
It is worth mentioning that in \cite{GMS}, $\beta$ is even allowed to be a uniformly positive function in $L^\infty(\Ga)$. However, in this paper we restrict ourselves to the case where $\beta$ is a positive constant.

We observe that any (sufficiently regular) solution to the GMS system satisfies the mass conservation law
\begin{align}
\label{GMS:MASS}
\beta \intO u(t) \dx + \intG v(t) \dG = \beta \intO u(0) \dx + \intG v(0) \dG, \quad t\in [0,\infty),
\end{align}
meaning that the parameter $\beta$ can be interpreted as a weight for the bulk mass compared to the surface mass. Moreover, the energy dissipation law 
\begin{align}
\label{GMS:NRG}
\frac{d}{dt} E\big(u(t),v(t)\big) + \mom \intO |\grad\mu(t)|^2 \dx +  \mga \intG|\gradg \theta(t)|^2 \dG = 0
\end{align}
is satisfied for all $t\in[0,\infty)$. In particular, we observe that the dissipation rate is strongly influenced by the values of the mobilities $\mom$ and $\mga$. 

The weak well-posedness and some results on long-time behavior were established in \cite{GMS}. We will present the well-posedness result as well as some additional important properties of weak solutions in Section~\ref{sec:wellposedGMS}. A summary of the results on long-time behavior is given in Section~\ref{LT:GMS}.

We further point out that numerical analysis for the GMS model as well as some numerical simulations can be found in \cite{Harder2020}. A nonlocal variant of the GMS model (including a nonlocal dynamic boundary condition) was proposed and analyzed in \cite{KS}.

\paragraph{The case $L=\infty$ (LW model).} Some time after the introduction of the GMS model, the system \eqref{CH:INT} with $L=\infty$ in \eqref{CH:INT:4} was derived by C.~Liu and H.~Wu \cite{LW} via an energetic variational approach. We will thus call this system the \textit{LW model}. 

The crucial difference to the GMS model is that the Dirichlet type boundary condition $\beta\theta = \mu$  is replaced by the no mass flux condition $\deln\mu = 0$. This means that the chemical potentials $\mu$ and $\theta$ are not directly coupled. However, mechanical interactions between the bulk and the surface materials are still taken into account through the trace condition \eqref{CH:INT:3} for the phase-field variable. Mathematically speaking, the elliptic subproblems $\big(\eqref{CH:INT:1}_1,\eqref{CH:INT:3}\big)$ and $\eqref{CH:INT:2}_1$ are coupled only through the trace relation \eqref{CH:INT:3}.

Compared to \eqref{GMS:MASS}, we obtain the very different mass conservation law
\begin{align}
\label{LW:MASS}
\intO u(t) \dx = \intO u(0) \dx 
\quad\text{and}\quad 
\intG v(t) \dG = \intG v(0) \dG, \quad t\in [0,\infty),
\end{align}
meaning that the bulk mass and the surface mass are conserved \textit{separately}. However, the energy dissipation law \eqref{GMS:NRG} is still satisfied by solutions of this system.

The well-posedness of the LW model was addressed in~\cite{LW,GK}, and its long-time behavior was investigated in~\cite{LW,MW}.
For numerical analysis and some simulations we refer to~\cite{Metzger2019,BaoZhangLW}. 

Moreover, a variant of the LW model \eqref{CH:INT} was proposed and investigated in \cite{KL} where the relation between $u$ and $v$ is given by the Robin type transmission condition
\begin{align*} 
K\deln u = H(v) - u \quad \text{ on } \Sigma
\end{align*}
with $K>0$ and a function $H\in C^2(\R)$ satisfying suitable growth conditions. In particular, it was rigorously established in \cite{KL} that in the case $H(s) = s$, solutions of this model converge to solutions of the LW model in the limit $K\to 0$ in some suitable sense. 

\paragraph{The case $0<L<\infty$ (KLLM model).} 
For $0<L<\infty$, the system \eqref{CH:INT} was recently proposed and analyzed by the second author in collaboration with K.F.~Lam, C.~Liu and S.~Metzger \cite{KLLM}. It will thus be referred to as the \textit{KLLM model}. 

The Robin type condition \eqref{CH:INT:3} establishes a connection between the GMS model (\eqref{CH:INT} with $L=0$) and the LW model (\eqref{CH:INT} with $L=\infty$) despite their very different chemical and physical properties. Suppose that $\beta>0$ and that $(u^L, v^L, \mu^L, \theta^L)$ is a solution of the system \eqref{CH:INT} corresponding to the parameter $L>0$.
Let $(u^0, v^0, \mu^0, \theta^0)$ denote its formal limit as $L\to 0$ and let $(u^\infty, v^\infty, \mu^\infty, \theta^\infty)$ denote its formal limit as $L\to\infty$. 
Passing to the limit in the Robin boundary condition, we deduce that 
\begin{align*}
\beta\theta^0 = \mu^0 \quad\text{on } \Sigma
\quad\text{and}\quad
\deln \mu^\infty = 0\quad\text{on } \Sigma.
\end{align*}
This corresponds to the limit cases of instantaneous relaxation to chemical equilibrium ($1/L\to\infty$), and the absence of adsorption and desorption ($1/L\to0$).
We infer that $(u^0, v^0, \mu^0, \theta^0)$ is a solution to the GMS model while $(u^\infty, v^\infty, \mu^\infty, \theta^\infty)$ is a solution to the LW model. These formal considerations are rigorously verified in \cite{KLLM}. In this regard, the Cahn--Hilliard system \eqref{CH:INT} with $L\in(0,\infty)$ can be interpreted as an interpolation between the GMS model and the LW model where the interpolation parameter $L$ corresponds to positive but finite kinetic rates.  

We observe that solutions of the KLLM model satisfy the same mass conservation law \eqref{GMS:MASS} as solutions of the GMS model. However, we obtain an additional term in the dissipation rate depending on the relaxation parameter $L$. To be precise, it holds that 
\begin{align}
\label{INT:NRG}
\frac{d}{dt} E\big(u(t)\big) + \mom \intO |\grad\mu(t)|^2 \dx + \mga \intG |\gradg \theta(t)|^2 \dG + \frac{\mom}{L} \intG \big(\beta\theta(t)-\mu(t)\big)^2 \dG = 0
\end{align}
for all $t\in [0,\infty)$.  In particular, this implies that the total free energy $E$ is decreasing along solutions, and as $E$ is bounded from below (at least for reasonable choices of $F$ and $G$), we infer that $\frac{d}{dt} E(u(t))$ tends to zero as $t\to\infty$. As a consequence, the chemical potentials will converge to the chemical equilibrium $\mu = \beta \theta$ as $t\to\infty$. 

We further point out that weak and strong well-posedness of the KLLM model was established in \cite{KLLM}. As already mentioned above, the asymptotic limits $L\to 0$ and $L\to \infty$ were also discussed in \cite{KLLM}. We will exploit some of these results for the limit $L\to 0$ in Section~5, Lemma~\ref{LEM:CONV:GMS}.

For simulations and numerical analysis concerning the KLLM model, we refer to \cite{KLLM,BaoZhangKLLM}.
A nonlocal variant of the KLLM model (including a nonlocal dynamic boundary condition) was proposed and investigated in \cite{KS}.

\paragraph{Contents of this paper.}
This paper is structured as follows. In Section~\ref{SEC:PIT}, we first introduce some notation and assumptions which will be used throughout the paper. We further define some suitable spaces and operators and we introduce several important interpolation inequalities. In Section~\ref{SEC:WPP}, we present the well-posedness results for both the KLLM model and the GMS model for the reader's convenience. Furthermore, we establish some properties of the weak solutions to both models which are essential to investigate their long-time behavior. Subsection~\ref{sec:long-time} is devoted to the long-time analysis of the KLLM model. We first characterize the set of stationary points containing the time-independent solutions to the KLLM model. Next, we establish the existence of a (unique) global attractor and we prove convergence of weak solutions to a single stationary point as $t\to\infty$ by means of a \L ojasiewicz--Simon inequality. Eventually, we show that the global attractor can be represented as the union of the unstable manifolds of all stationary points. In Subsection~\ref{LT:GMS}, supported by the results on long-time behavior that have already been established in \cite{GMS}, we sketch how all results we proved in Subsection~\ref{sec:long-time} for the KLLM model can be obtained also for the GMS model in a similar fashion. Next, in Section~\ref{sec:contglbatt}, we show that the global attractor of the GMS model (i.e., $L=0$) is stable with respect to perturbations of the kinetic rate (i.e., $L>0$ being small). Ultimately, in Section~\ref{SEC:EXP}, we show that there exists a family $\{\mathfrak M^L\}_{L\ge 0}$ of exponential attractors for the system \eqref{CH:INT} such that the attractor $\mathfrak M^0$ associated with the GMS model is robust in some certain sense against perturbations of the kinetic rate (i.e., $L>0$ being small).

\paragraph{Comparison of the GMS/KLLM dynamics with the LW dynamics.}
We point out that we do not see any possibility of transferring the results of Section~\ref{sec:contglbatt} and Section~\ref{SEC:EXP} to the scenario $L\to\infty$. This is mainly due to the different mass conservation law of the LW model. Roughly speaking, the mass conservation law \eqref{GMS:MASS} of the GMS/KLLM model fixes only one degree of freedom, whereas the mass conservation law \eqref{LW:MASS} of the LW model already fixes two degrees of freedom. As a consequence, the GMS/KLLM model and the LW model exhibit a different behavior regarding energy dissipation. Namely, as only one degree of freedom is fixed by the mass conservation law, the free energy can usually be decreased much further by the GMS/KLLM model than by the LW model. This phenomenon can also be observed in numerical simulations, see, e.g., \cite[Fig.~4]{KLLM} or \cite[Fig.~16]{BaoZhangKLLM}. This already indicates that the GMS/KLLM model and the LW model differ in their long-time behavior.
For more details why our analysis in Section~\ref{sec:contglbatt} and Section~\ref{SEC:EXP} fails in the situation $L\to\infty$, we refer to Remark~\ref{REM:FAIL}.

\paragraph{Further related results in the literature.}
For further results on long-time behavior for Cahn--Hilliard models with dynamic boundary conditions, we refer to \cite{pruss-wilke,GMS-lt,Gal-lt,MW,GMS,CFP}.
We also want to mention \cite{israel-lt} where the long-time dynamics of an Allen--Cahn model with dynamic boundary conditions were studied,
and \cite{israel-lt,CFP} where the long-time behavior of a Caginalp phase field model with dynamic boundary conditions was analyzed.

\section{Preliminaries and important tools}\label{SEC:PIT}

We will now fix some notation and assumptions that are supposed to hold throughout this paper.

\paragraph{Notation.}
\begin{enumerate}[label=$(\mathrm{N \arabic*})$, ref = $\mathrm{N \arabic*}$]
\item Sometimes, we will use the notation $\RP:=[0,\infty)$. The symbol $\N_0$ denotes the set of natural numbers including zero and $\N = \N_0\setminus\{0\}$.
\item For any real numbers $k \geq 0$ and $1 \leq p \leq \infty$, the standard Lebesgue and Sobolev spaces over a domain $\Omega$ are denoted as $L^p(\Omega)$ and $W^{k,p}(\Omega)$. We write $\norm{\cdot}_{L^p(\Omega)}$ and $\norm{\cdot}_{W^{k,p}(\Omega)}$ to denote the standard norms on these spaces. If $p = 2$, these spaces are Hilbert spaces and we use the notation $H^k(\Omega) = W^{k,2}(\Omega)$. We point out that $H^0(\Omega)$ can be identified with $L^2(\Omega)$. An analogous notation is used for Lebesgue and Sobolev spaces on $\Gamma:=\partial\Omega$, provided that the boundary is sufficiently regular.  
\item For any Banach space $X$, we write $X'$ to denote its dual space. The associated duality pairing of elements $\phi\in X'$ and $\zeta\in X$ is denoted as $\inn{\phi}{\zeta}_X$. If $X$ is a Hilbert space, we write $(\cdot, \cdot)_X$ to denote its inner product. 
\item We define
\begin{align*}
\mean{u}_\Omega := \begin{cases}
\frac{1}{\abs{\Omega}} \inn{u}{1}_{H^1(\Omega)} & \text{ if } u \in H^1(\Omega)', \\
\frac{1}{\abs{\Omega}} \int_\Omega u \dx & \text{ if } u \in L^1(\Omega)
\end{cases}
\end{align*}
to denote the (generalized) spatial mean of $u$. Here, $\abs{\Omega}$ denotes the $d$-dimensional Lebesgue measure of $\Omega$. The spatial mean of a function $v \in H^1(\Gamma)'$ (or $v \in L^1(\Gamma)$, respectively) is defined analogously.
\item Let $(X,d)$ be a metric space. Then for any sets $A,B\subset X$, the \emph{Hausdorff semidistance} is defined as
\begin{align}
	\label{DEF:SEMIDIST}
	\dist_{X}(A,B) :=\underset{a\in A}{\sup } \; \underset{b\in B}{\inf } \; d \left( a,b\right), 
\end{align} 
and the \emph{symmetric Hausdorff distance} is given by 
\begin{align}
	\label{DEF:SYMDIST}
	\dist_{\mathrm{sym},X}(A,B) := \max\big( \dist_{X}(A,B) \,{,}\, \dist_{X}(B,A) \big).
\end{align}
For any compact set $K\subset X$, the \emph{fractal dimension of $K$} is defined as
\begin{align}
	\label{DEF:DIMFRAC}
	\dim_{\mathrm{frac},X}(K) := \underset{r\to 0}{\lim\sup} \frac{\ln\big(\mathcal N_r(K;X)\big)}{-\ln(r)},
\end{align}
where $\mathcal N_r(K;X)$ denotes the minimal number of balls in $X$ with radius $r$ that are necessary to cover the set $K$.
\end{enumerate}

\paragraph{Assumptions.}%
We make the following general assumptions.
\begin{enumerate}[label=$(\mathrm{A \arabic*})$, ref = $\mathrm{A \arabic*}$]
	\item \label{ass:dom} We assume that $\Omega\subset \R^d$ with $d\in\{2,3\}$ is a bounded domain whose boundary $\Gamma:=\del\Omega$ is of class $C^3$. We further use the notation
	\begin{align*}
		Q:=\Omega\times(0,\infty),\quad \Si:=\Ga\times(0,\infty).
	\end{align*}
	\item \label{ass:const} In general, we assume that the constants occurring in the system \eqref{CH:INT} satisfy $T$, $\beta$, $\kappa$, $\mo$, $\mg>0$ and $L\in[0,\infty)$. 
	Since the choice of $\delta$, $\eps$, $\kappa$, $\mo$ and $\mg$ has no impact on the mathematical analysis, we will simply set $\delta=\eps=\kappa=\mo=\mg=1$ for convenience in the mathematical analysis.
	\item  \label{ass:pot}
	We assume that the potentials $F$ and $G$ are non-negative functions which can be written as $F=F_1+F_2$ and $G=G_1+G_2$ with  $F_1,F_2,G_1,G_2 \in C^2(\R)$ such that the following properties hold:
	\begin{enumerate}[label=$(\mathrm{A 3.\arabic*})$, ref = $\mathrm{A 3.\arabic*}$]
		\item \label{ass:pot:1} There exist exponents $2<p\le 4$ and $q>2$,
		as well as constants $a_{F},a_{F'},c_{F},c_{F'}>0$ and $b_{F},b_{F'}\ge 0$ such that for all $s\in\R$,
		\begin{subequations}
		\begin{alignat}{3}
		\label{GR:F}
		a_{F}\abs{s}^p - b_{F} &\le F(s) &&\le c_{F}(1+\abs{s}^p), \\
		\label{GR:G}
		a_{G}\abs{s}^q - b_{G} &\le G(s) &&\le c_{G}(1+\abs{s}^q), \\
		a_{F'}\abs{s}^{p-1} - b_{F'} &\le \abs{F'(s)} &&\le c_{F'}(1+\abs{s}^{p-1}), \\
		a_{G'}\abs{s}^{q-1} - b_{G'} &\le \abs{G'(s)} &&\le c_{G'}(1+\abs{s}^{q-1}).
		\end{alignat}
		\end{subequations}
		\item \label{ass:pot:2} The functions $F_1$ and $G_1$ are convex and non-negative. There further exist positive constants $c_{F''}$ and $c_{G''}$ such that the  second order derivatives satisfy the growth conditions
		\begin{align}
		0\le F_1''(s) \leq c_{F''}(1 + |s|^{p-2}), \quad 
		0\le G_1''(s) \leq c_{G''} (1 + |s|^{q-2})
		\end{align}
		for all $s\in\R$. 
		\item \label{ass:pot:3} The derivatives $F_2'$ and $G_2'$ are Lipschitz continuous. Consequently, there exist positive constants $d_F$, $d_G$, $d_{F'}$ and $d_{G'}$ such that for all $s\in\R$,
		\begin{alignat}{3}
		&\abs{F_2'(s)} \le d_{F'}(1+\abs{s}), \quad &&\abs{G_2'(s)} \le d_{G'}(1+\abs{s}), \\ 
		&\abs{F_2(s)} \le d_F(1+\abs{s}^2), \quad &&\abs{G_2(s)} \le d_G(1+\abs{s}^2).
		\end{alignat}
	\end{enumerate} 
\end{enumerate}
 
In the analysis of the long-time dynamics, we will frequently use the following additional assumptions:

\begin{enumerate}[label=$(\mathrm{A\arabic*})$, ref = $\mathrm{A\arabic*}$, start=4] 
	\item \label{ass:comp} The potentials $F$ and $G$ satisfy the compatibility condition $F(s) = \beta G(s)$ for all $s\in\R$.
	\item \label{ass:ana} The potentials $F,G\in C^\infty(\R)$ are analytic functions. 
\end{enumerate}

\bigskip

\begin{remark}\label{REM:ASS}\normalfont 
		We point out that the polynomial double-well potential
		\begin{align*}
		W_\text{dw}(s)=\tfrac 1 4 (s^2-1)^2,\quad s\in\R,
		\end{align*} 
		is a suitable choice for $F$ and $G$ as it satisfies \eqref{ass:pot} with $p = 4$ and $q = 4$, and obviously also \eqref{ass:ana}. However, singular potentials like the logarithmic potential or the obstacle potential are not admissible as they do not even satisfy \eqref{ass:pot}.
\end{remark}

\bigskip

\paragraph{Preliminaries.} We next introduce several function spaces, products, norms and operators that will be used throughout this paper.
\begin{enumerate}[label=$(\mathrm{P \arabic*})$, ref = $\mathrm{P \arabic*}$]
	\item For any $k\in\N_0$ and any real number $p\in[1,\infty]$, we set
	\begin{align*}
	\LL^p := L^p(\Omega)\times L^p(\Gamma), 
	\quad\text{and}\quad
	\HH^k := H^k(\Omega) \times H^k(\Gamma),
	\end{align*}
	and we identify $\LL^2$ with $\HH^0$. Note that $\HH^k$ is a Hilbert space with respect to the inner product
	\begin{align*}
	\bigscp{(\phi,\psi)}{(\zeta,\xi)}_{\HH^k} := \bigscp{\phi}{\zeta}_{H^k(\Omega)} + \bigscp{\psi}{\xi}_{H^k(\Gamma)}
	\quad\text{for all $(\phi,\psi),(\zeta,\xi)\in\HH^k$,}
	\end{align*}
	and its induced norm $\norm{\cdot}_{\HH^k}:= \scp{\cdot}{\cdot}_{\HH^k}^{1/2}$.
	
	\item \label{pre:V} 
	For any $k\in\N$, we introduce the Hilbert space
	\begin{align*}
	\Vk  := 
	\big\{ (\phi,\psi) \in \HH^k \suchthat \phi \vert_\Gamma = \psi \;\text{a.e.~on}\; \Gamma\big\}
	\end{align*}
	endowed with the inner product $\scp{\cdot}{\cdot}_{\Vk}:=\scp{\cdot}{\cdot}_{\HH^k}$ and the norm $\norm{\cdot}_{\Vk}:=\norm{\cdot}_{\HH^k}$.
	
	\item \label{pre:W}
	For any 
	$m\in\R$, $\beta>0$ and $k\in\N$, we set
	\begin{align*}
		\Wmk &:= \big\{ (\phi,\psi) \in \Vk  \suchthat \beta\abs{\Omega}\mean{\phi}_\Om + \abs{\Gamma}\mean{\psi}_\Ga = m\big\}.
	\end{align*}
	Endowed with the inner product
	\begin{align*}
	\scp{\cdot}{\cdot}_{\WW_{\beta,0}^k} := \scp{\cdot}{\cdot}_{\HH^k}
	\end{align*}
	and the induced norm, the space $\WW_{\beta,0}^k$ is a Hilbert space.
	
	\item \label{pre:D}
	For $\beta>0$, we introduce the subspace
	\begin{align*}
		\Db := \left\{ (\phi,\psi) \in \HH^1 \;\big\vert\; \phi\vert_\Ga = \beta\psi \;\; \text{a.e. on}\; \Ga \right\} \subset \HH^1.
	\end{align*}
	Endowed with the inner product
	\begin{align*}
		\scp{\cdot}{\cdot}_\Db := \scp{\cdot}{\cdot}_{\HH^1}
	\end{align*}
	and its induced norm,
	the space $\Db$ is a Hilbert space. Moreover, we define the product
	\begin{align*}
		\biginn{(\phi,\psi)}{(\zeta,\xi)}_{\Db} := \scp{\phi}{\zeta}_{L^2(\Omega)} + \scp{\psi}{\xi}_{L^2(\Gamma)}
	\end{align*}
	for all $(\phi,\psi), (\zeta,\xi)\in \LL^2$. By means of the Riesz representation theorem, this product can be extended to a duality pairing on $(\Db)'\times \Db$, which will also be denoted as $\inn{\cdot}{\cdot}_{\Db}$. 
	
	In particular, the spaces $\big(\Db,\LL^2,(\Db)'\big)$ form a Gelfand triplet, and the operator norm on $(\Db)'$ is given by
	\begin{align*}
		\norm{(\phi,\psi)}_{(\Db)'} := \sup\Big\{\, \big|\inn{(\phi,\psi)}{(\zeta,\xi)}_{\Db}\big| \;\Big\vert\; (\zeta,\xi)\in\Db \text{ with } \norm{(\zeta,\xi)}_{\Db} = 1 \Big\},
	\end{align*}
	for all $(\phi,\psi) \in(\Db)'$. Note that the mapping
	\begin{align*}
		\mathfrak{I}:\VV^1\to\Db,\quad (\phi,\psi) \mapsto (\beta\phi,\psi)
	\end{align*}
	is an isomorphism.
	
	\item \label{pre:H} Let $L\ge 0$, $\beta> 0$ and $m\in\R$ be arbitrary. We define the space 
	\begin{align*}
		\HLB &:= 
		\begin{cases}
			\left\{ (\phi,\psi) \in \HH^1 \suchthat \beta\abs{\Omega}\mean{\phi}_\Om + \abs{\Gamma}\mean{\psi}_\Ga = 0\right\},
			&\text{if}\; L>0,\\[2ex]
			\left\{ (\phi,\psi) \in \Db \suchthat \beta\abs{\Omega}\mean{\phi}_\Om + \abs{\Gamma}\mean{\psi}_\Ga = 0\right\},
			&\text{if}\; L=0,
		\end{cases}
	\end{align*}
	along with the inner product 
	\begin{align*}
		\bigscp{(\phi,\psi)}{(\zeta,\xi)}_{L,\beta} 
		&:= \intO \nabla \phi \cdot \nabla \zeta \dx + \intG \gradg \psi \cdot \gradg \xi\dG \\
		&\qquad + \sigma(L) \intG(\beta \psi-\phi)(\beta\xi-\zeta) \dG,
	\end{align*}
	for all $(\phi,\psi),(\zeta,\xi) \in \HLB$, where
	\begin{align}
	\label{DEF:SIG}
		\sigma(L) = 
		\begin{cases}
			L^{-1} &\text{if}\; L>0,\\
			0 &\text{if}\; L=0.
		\end{cases}
	\end{align}
	This is indeed an inner product since $\scp{(\phi,\psi)}{(\phi,\psi)}_{L,\beta}=0$ already entails that $\phi$ and $\psi$ are constant almost everywhere and satisfy $\phi\vert_\Gamma = \beta \psi$ almost everywhere on $\Gamma$. Then, the mean value constraint of $\HLB$ directly yields $(\phi,\psi)=(0,0)$ almost everywhere.
	The induced norm is given as $\norm{\,\cdot\,}_{L,\beta}:= \scp{\cdot}{\cdot}_{L,\beta}^{1/2}$. 
	
	\item \label{pre:lin:L} 
	Let $L,\beta> 0$ be arbitrary. We define the space
	\begin{align*}
	\HH_{\beta}^{-1} &:= \big\{ (\phi,\psi) \in (\HH^1)' \suchthat \beta\abs{\Omega}\mean{\phi}_\Om + \abs{\Gamma}\mean{\psi}_\Ga = 0 \big\}
	\subset (\HH^1)'.
	\end{align*}
	We conclude from \cite[Thm.~3.3]{knopf-liu} that for any $(\phi,\psi)\in \HH_{\beta}^{-1}$, there exists a unique weak solution $\SS^L(\phi,\psi) = (\SS^L_\Omega(\phi,\psi),\SS^L_\Gamma(\phi,\psi))\in \HLB $ to the elliptic problem%
	\begin{subequations}
		\label{EQ:LIN}
		\begin{align}
		- \Lap \SS^L_\Om &= - \phi && \text{ in } \Omega, \\
		- \Lapg \SS^L_\Ga + \beta \pdnu \SS^L_\Om &= - \psi && \text{ on } \Gamma, \\
		\label{EQ:LIN:3}
		L\pdnu \SS^L_\Om &= (\beta \SS^L_\Ga - \SS^L_\Om) && \text{ on } \Gamma.
		\end{align}
	\end{subequations}	
	This means that $\SS^L(\phi,\psi)$ satisfies the weak formulation
	\begin{align}
	\label{WF:LIN}
		\bigscp{\SS^L(\phi,\psi)}{(\zeta,\xi)}_{L,\beta} 
			= - \biginn{(\phi,\psi)}{(\zeta,\xi)}_{\HH^1}
	\end{align}
	for all test functions $(\zeta,\xi)\in \HH^1$.
	As in \cite[Thm.~3.3 and Cor.~3.5]{knopf-liu}, we can thus define the solution operator
	\begin{align*}
	\SS^L: \HH_{\beta}^{-1} \to \HLB,\quad (\phi,\psi) \mapsto \SS^L(\phi,\psi) = (\SS^L_\Omega(\phi),\SS^L_\Gamma(\phi,\psi))  
	\end{align*}
	as well as an inner product and its induced norm on the space $\HH_{\beta}^{-1}$ by
	\begin{align*}
	\bigscp{(\phi,\psi)}{(\zeta,\xi)}_{L,\beta,*} &:= \bigscp{\SS^L(\phi,\psi)}{\SS^L(\zeta,\xi)}_{L,\beta}, \\
	\norm{(\phi,\psi)}_{L,\beta,*} &:=\scp{(\phi,\psi)}{(\phi,\psi)}_{L,\beta,*}^{1/2} 
	\end{align*}
	for all $(\phi,\psi),(\zeta,\xi)\in \HH_{\beta}^{-1}$. This norm is equivalent to the standard operator norm on $\HH_{\beta}^{-1}$.	
	Since $\HH_{L,\beta}^1 \subset \HH_{\beta}^{-1}$, the product $\scp{\cdot}{\cdot}_{L,\beta,*}$ can also be used as an inner product on $\HH_{L,\beta}^1$.
	Moreover, $\norm{\cdot}_{L,\beta,*}$ is also a norm on $\HH_{L,\beta}^1$ but $\HH_{L,\beta}^1$ is not complete with respect to this norm.
	
	\item \label{pre:lin:0}
	Let $\beta>0$ be arbitrary. We define the space
	\begin{align*}
		\DD_{\beta}^{-1} &:= \Big\{ (\phi,\psi) \in (\Db)' \;\big\vert\; \biginn{(\phi,\psi)}{(\beta,1)}_\Db = 0\, \Big\}
		\subset (\Db)'.
	\end{align*}
	Proceeding exactly as in \cite[Proof of Thm.~3.3]{knopf-liu}, we use the Lax--Milgram theorem to show that for any $(\phi,\psi)\in\DD_{\beta}^{-1}$, there exists a unique weak solution $\SS^0(\phi,\psi) = (\SS^0_\Omega(\phi),\SS^0_\Gamma(\phi,\psi))\in \HH^1_{0,\beta} \subset \Db$ to the elliptic problem
	\begin{subequations}
		\label{EQ:LIN}
		\begin{align}
			- \Lap \SS^0_\Om &= - \phi && \text{ in } \Omega, \\
			- \Lapg \SS^0_\Ga + \beta \pdnu \SS^0_\Om &= - \psi && \text{ on } \Gamma, \\
			\label{EQ:LIN:3}
			 \SS^0_\Om \vert_\Ga &= \beta \SS^0_\Ga && \text{ on } \Gamma.
		\end{align}
	\end{subequations}	
	This means that $\SS^0(\phi,\psi)$ satisfies the weak formulation
	\begin{align}
		\label{WF:LIN}
		\biginn{ \SS^0(\phi,\psi) }{ (\zeta,\xi) }_{0,\beta} 
			= - \biginn{(\phi,\psi)}{(\zeta,\xi)}_{\Db}
	\end{align}
	for all test functions $(\zeta,\xi)\in \Db$.	
	Similar to \cite[Thm.~3.3 and Cor.~3.5]{knopf-liu}, we can thus define the solution operator
	\begin{align*}
		\SS^0: \DD_{\beta}^{-1} \to \HH^1_{0,\beta},\quad (\phi,\psi) \mapsto \SS^0(\phi,\psi) = (\SS^0_\Omega(\phi),\SS^0_\Gamma(\phi,\psi))  
	\end{align*}
	as well as an inner product and its induced norm on the space $\DD_{\beta}^{-1}$ by
	\begin{align*}
		\bigscp{(\phi,\psi)}{(\zeta,\xi)}_{0,\beta,*} &:= \bigscp{\SS^0(\phi,\psi)}{\SS^0(\zeta,\xi)}_{0,\beta}, \\
		\norm{(\phi,\psi)}_{0,\beta,*} &:=\scp{(\phi,\psi)}{(\phi,\psi)}_{0,\beta,*}^{1/2} 
	\end{align*}
	for all $(\phi,\psi),(\zeta,\xi)\in \DD_{\beta}^{-1}$.	
	Since $\HH_{0,\beta}^1 \subset \DD_\beta^{-1}$, the product $\scp{\cdot}{\cdot}_{L,\beta,*}$ can also be used as an inner product on $\HH_{0,\beta}^1$.
	Moreover, $\norm{\cdot}_{L,\beta,*}$ is also a norm on $\HH_{0,\beta}^1$ but $\HH_{0,\beta}^1$ is not complete with respect to this norm.
\end{enumerate}

We next show that for functions in $\Wo$, the norm $\norm{\cdot}_{L,\beta,*}$ can be bounded by the norm $\norm{\cdot}_{(\HH^1)'}$ uniformly in $L\in[0,1]$.

\begin{lemma}\label{LEM:LBS}
	Let $\beta>0$ be arbitrary.
	Then, there exists a constant $C>0$ depending only on $\beta$ and $\Omega$ such that for all $L\in[0,1]$ and all $(\phi,\psi)\in \Wo$, we have 
	\begin{align*}
		\norm{(\phi,\psi)}_{L,\beta,*} \le C \norm{(\phi,\psi)}_{(\HH^1)'}.
	\end{align*} 
\end{lemma}

\begin{proof}
	Let $\beta>0$, $L\in[0,1]$ and $(\phi,\psi)\in \Wo$ be arbitrary. In the following, the letter $C$ will denote generic positive constants depending only on $\beta$ and $\Omega$. Recalling \eqref{pre:lin:L} and \eqref{pre:lin:0}, we obtain
	\begin{align}
		\label{LBS:1}
		\norm{(\phi,\psi)}_{L,\beta,*}^2 
		&= \norm{\SS^L(\phi,\psi)}_{L,\beta}^2
		= - \bigscp{\SS^L(\phi,\psi)}{(\phi,\psi)}_{\LL^2} \notag\\
		&\le \norm{\SS^L(\phi,\psi)}_{\HH^1}\, \norm{(\phi,\psi)}_{(\HH^1)'}.
	\end{align}
	We can now apply \cite[Cor.~7.2]{knopf-liu} (with $\alpha=\beta$ and $K=1$) to deduce that
	\begin{align}
		\label{LBS:2}
		\norm{\SS^L(\phi,\psi)}_{\HH^1} \le C\, \norm{\SS^L(\phi,\psi)}_{1,\beta} \; .
	\end{align}
	Plugging this estimate into \eqref{LBS:1}, and recalling that $L\in[0,1]$, we conclude that
	\begin{align}
		\label{LBS:3}
		\norm{(\phi,\psi)}_{L,\beta,*}^2 
		&\le C\, \norm{\SS^L(\phi,\psi)}_{1,\beta}\, \norm{(\phi,\psi)}_{(\HH^1)'}
		\le C\, \norm{\SS^L(\phi,\psi)}_{L,\beta}\, \norm{(\phi,\psi)}_{(\HH^1)'} \notag\\
		&= C\, \norm{(\phi,\psi)}_{L,\beta,*}\, \norm{(\phi,\psi)}_{(\HH^1)'}.
	\end{align}
	Since $(\phi,\psi)\in \Wo$ was arbitrary, the assertion directly follows and the proof is complete.
\end{proof}

\paragraph{Interpolation inequalities.}

We will further need the following interpolation estimates. It will turn out to be essential that the constants in these estimates are independent of the parameter $L$.

\begin{lemma} \label{LEM:INT}
	Suppose that \eqref{ass:dom} holds and let $L\ge 0$ and $\beta>0$ be arbitrary. Then there exists a constant $c>0$ depending only on $\beta$, such that for all $(u,v)\in \Wo\,$,
	\begin{align}
		\label{IEQ:INT:1}
		\norm{(u,v)}_{\LL^2}^2 \le 
		c\, \norm{(\grad u,\gradg v)}_{\LL^2}
			\left\|(u,v)\right\|_{L,\beta,*}.
	\end{align}
	Moreover, for any $\alpha>0$ and all $(u,v)\in \Wo\,$, it holds that
	\begin{align}
		\label{IEQ:INT:2}
		\norm{(u,v)}_{\LL^2}^2 
		\le \alpha \norm{(\grad u,\gradg v)}_{\LL^2}^2 
			+ \frac{c^2}{4\alpha} \norm{(u,v)}_{L,\beta,*}^2 .
	\end{align}
\end{lemma}

\begin{proof}
	Let $L\in [0,\infty)$ and $(u,v)\in \Wo$ be arbitrary. Recalling the definition of the operator $\SS^L$ (see \eqref{pre:lin:L} and \eqref{pre:lin:0}) and that $u\vert_\Gamma = v$ almost everywhere on $\Gamma$, we obtain
	\begin{align*}
		&\min\{\beta,1\}\, \norm{(u,v)}_{\LL^2} 
		\le \beta \left\| u\right\|^{2}_{L^{2}(\Omega)}
			+\left\| v\right\|^{2}_{L^{2}(\Gamma)} \\
		&\quad =\scp{u}{\beta u}_{L^2(\Omega)}
			+\scp{v}{v}_{L^2(\Omega)} \\[1ex]
		&\quad= - \int _{\Omega}\grad \big( \SS^L_{\Omega}(u,v) \big) \cdot \grad(\beta u)  \dx
			- \int _{\Gamma}\gradg\big(\SS^L_{\Gamma}(u,v)\big)\cdot \gradg v \dG \\
		&\qquad\quad - \sigma(L) \int _{\Gamma}(\beta\SS^L_{\Gamma}(u,v)-\SS^L_{\Omega}(u,v))(\beta v-\beta v) \dG\\[1ex]
		&\quad\leq \beta \bignorm{ \grad \big( \SS^L_{\Omega}(u,v) \big) }_{L^{2}(\Omega)}
				\bignorm{\grad u}_{L^{2}(\Omega)}
			+ \bignorm{ \gradg \big( \SS^L_{\Gamma}(u,v) \big) }_{L^{2}(\Gamma)}
				\bignorm{ \gradg v }_{L^{2}(\Gamma)}\\
		&\quad\leq \max\{\beta,1\}\, \norm{(\grad u,\gradg v)}_{\LL^2} 
			\,\norm{\SS^L(u,v)}_{L,\beta}\\
		&\quad= \max\{\beta,1\}\, \norm{(\grad u,\gradg v)}_{\LL^2} 
			\,\norm{(u,v)}_{L,\beta,*}\;.
	\end{align*}
	This proves \eqref{IEQ:INT:1} and thus, \eqref{IEQ:INT:2} direcly follows by means of Young's inequality.
\end{proof}

\medskip

\begin{lemma}\label{LEM:INT:2}
	Suppose that \eqref{ass:dom} holds. Then there exist constants $C_1,C_2,C_3,C_4>0$ such that 
	\begin{alignat}{2}
	\label{INT:H1}
		\norm{(u,v)}_{\HH^1} &\le C_1 \norm{(u,v)}_{\HH^2}^{\frac 12} \norm{(u,v)}_{\LL^2}^{\frac 12} 
			&&\quad\text{for all $(u,v) \in \HH^2$}, \\
	\label{INT:H2}
		\norm{(u,v)}_{\HH^2} &\le C_2 \norm{(u,v)}_{\HH^3}^{\frac 12} \norm{(u,v)}_{\HH^1}^{\frac 12} 
			&&\quad\text{for all $(u,v) \in \HH^3$}, \\
	\label{INT:H1:ALT}
	\norm{(u,v)}_{\HH^1} &\le C_3 \norm{(u,v)}_{\HH^3}^{\frac 13} 
	\norm{(u,v)}_{\LL^2}^{\frac 23} 
			&&\quad\text{for all $(u,v) \in \HH^3$},\\
	\label{INT:H2:ALT}
		\norm{(u,v)}_{\HH^2} &\le C_4 \norm{(u,v)}_{\HH^3}^{\frac 23} 
			\norm{(u,v)}_{\LL^2}^{\frac 13} 
			&&\quad\text{for all $(u,v) \in \HH^3$}.
	\end{alignat}	
\end{lemma}

\begin{proof}
	To prove \eqref{INT:H1}, let $(u,v)\in\HH^2$ be arbitrary. It obviously holds that
	\begin{align}
		\label{EST:L2}
		\norm{(u,v)}_{\LL^2}^2 \le \norm{(u,v)}_{\HH^2} \norm{(u,v)}_{\LL^2}.	
	\end{align}	
	Using the Gagliardo--Nierenberg inequality, we obtain 
	\begin{align}
		\label{EST:H1:BULK}
		\norm{\grad u}_{L^2(\Omega)}^2 
		\le \norm{u}_{H^2(\Omega)} \norm{u}_{L^2(\Omega)}
		\le \norm{(u,v)}_{\HH^2} \norm{(u,v)}_{\LL^2}.
	\end{align}
	Moreover, the Gagliardo--Nierenberg inequality for compact Riemannian manifolds (see, e.g., \cite[Thm.~3.70]{Aubin}) implies that
	\begin{align}
		\label{EST:H1:SURF}
		\norm{\gradg v}_{L^2(\Gamma)}^2 
		\le \norm{v}_{H^2(\Gamma)} \norm{v}_{L^2(\Gamma)}
		\le \norm{(u,v)}_{\HH^2} \norm{(u,v)}_{\LL^2}.
	\end{align}
	Now, adding \eqref{EST:L2}, \eqref{EST:H1:BULK} and \eqref{EST:H1:SURF}  proves \eqref{INT:H1}. 
	
	The inequality \eqref{INT:H2} follows by applying \eqref{INT:H1} on the components $\big([\grad u]_i,[\gradg v]_i\big)$, $i=1,2,3$. Eventually, the estimates \eqref{INT:H1:ALT} and \eqref{INT:H2:ALT} are direct consequences of \eqref{INT:H1} and \eqref{INT:H2}. This completes the proof.
\end{proof}

The following result is a direct consequence of Lemma~\ref{LEM:INT} and Lemma~\ref{LEM:INT:2}.

\begin{corollary} \label{COR:INT}
	Suppose that \eqref{ass:dom} holds and let $L\ge 0$ and $\beta>0$ be arbitrary. Then there exist positive constants $c_1$ and $c_2$ depending only on $\beta$ such that
	\begin{alignat}{2}
		\norm{(u,v)}_{\HH^1} &\le c_1\, \norm{(u,v)}_{\HH^2}^{\frac 23} \norm{(u,v)}_{L,\beta,*}^{\frac 13} 
		&&\quad \text{for all $(u,v)\in \WW_{\beta,0}^2$,}\\
		\norm{(u,v)}_{\HH^1} &\le c_2\, \norm{(u,v)}_{\HH^3}^{\frac 12} \norm{(u,v)}_{L,\beta,*}^{\frac 12} 
		&&\quad \text{for all $(u,v)\in \WW_{\beta,0}^3$.}
	\end{alignat}
\end{corollary}

%

\section{Well-posedness and further essential properties}\label{SEC:WPP}
In this section, we present the weak well-posedness results for the problem \eqref{CH:INT} with $L\in [0,\infty)$, and we show that these weak solutions satisfy a certain parabolic smoothing property. For simplicity, and as their choice has no impact on the analysis we will carry out, the positive constants $m_\Omega$, $m_\Gamma$, $\eps$, $\delta$ and $\kappa$ are set to one. 
By this choice, the system \eqref{CH:INT} can be restated as follows.
\begin{subequations}\label{CH:INT.}
	\begin{alignat}{3}
	& \delt u = \Delta
	\mu, 
	&&\mu = -  \Lap  u +  F'( u) 
	&&\quad \text{in } Q, \label{CH:INT:1.}\\
	& \delt v =  \Lapg  \theta - \beta  \pdnu\mu, 
	&&\theta = - \Lapg v +  G'(v) + \pdnu  u 
	&&\quad \text{on } \Sigma, \label{CH:INT:2.}\\
	& u\vert_\Gamma = v 
	&&&&\quad \text{on } \Sigma, \label{CH:INT:3.}\\[1ex]
	&\left\{
	\begin{aligned}
	\mu\vert_\Si &= \beta  \theta \\
	L \pdnu\mu\vert_\Si &= \beta  \theta -  \mu\vert_\Si
	\end{aligned}
	\right.
	&&\quad
	\begin{aligned}
	&\text{if}\; L=0,\\
	&\text{if}\; L\in(0,\infty),
	\end{aligned}
	&&\quad \text{on } \Sigma, \label{CH:INT:4.} \\[1ex]
	& (u,v)\vert_{t=0} = (u_0,v_0) 
	&&\text{with}\; u_0\vert_\Si = v_0 
	&&\quad \text{in } \Omega \times \Gamma. \label{CH:INT:5.}
	\end{alignat}
\end{subequations}
The total free energy associated with this system is given as
\begin{align}
	\label{DEF:EN}
	E(u,v) = \intO \frac 12 \abs{\grad u}^2 + F(u) \dx + \intG \frac 12 \abs{\gradg v}^2 + G(v) \dG.
\end{align}
In the following, we will discuss the cases $0<L<\infty$ (KLLM~model) and $L=0$ (GMS~model) separately.

\subsection{The KLLM model ($0<L<\infty$)}

\begin{theorem}[Weak well-posedness and smoothing property for the KLLM model] \label{THM:WP:KLLM}
	Suppose that \eqref{ass:dom}--\eqref{ass:pot} hold and that $L\in(0,\infty)$. Let $m\in\R$ be arbitrary and let 
	$(u_0,v_0) \in \Wm$ be any initial datum.
	
	Then there exists a unique global weak solution $(u^L,v^L,\mu^L,\theta^L)$ of the system \eqref{CH:INT.} existing on the whole time interval $\RP$ and having the following properties:
	\begin{enumerate}[label=$(\mathrm{\roman*})$, ref = $\mathrm{\roman*}$]
	\item For every $T>0$, the solution has the regularity
	\begin{subequations}
	\label{REG:KLLM}
	\begin{align}
	\label{REG:KLLM:1}
	&\begin{aligned}
	\big(u^L,v^L\big) 
		& \in C^{0,\frac{1}{4}}([0,T];\LL^2) \cap L^\infty(0,T;\VV^1) \\
		& \qquad \cap L^2(0,T;\HH^3) \cap H^1\big(0,T;(\HH^1)'\big), 
	\end{aligned}
	\\
	\label{REG:KLLM:2}
	&\big(\mu^L,\theta^L\big) \in L^2(0,T;\HH^1).
	\end{align}
	\end{subequations}
	\item
	The solution satisfies
	\begin{subequations}
		\label{WF:KLLM}
		\begin{align}
		\label{WF:KLLM:1}
		\biginn{\delt u^L}{ w}_{H^1(\Omega)} &= - \intO \grad\mu^L \cdot \grad w \dx 
		+  \intG \tfrac 1 L(\beta\theta^L-\mu^L)  w \dG, \\
		\label{WF:KLLM:2}
		\biginn{\delt v^L}{ z}_{H^1(\Gamma)} &= - \intG \gradg\theta^L \cdot \gradg z \dG 
		- \intG \tfrac 1 L (\beta\theta^L-\mu^L) \beta z \dG, 
		\end{align}
	a.e.~on $\RP$ for all test functions $(w,z)\in\HH^1$, and
	\begin{alignat}{2}
		\label{WF:KLLM:3}
		&\mu^L = -\Lap u^L + F'(u^L) 
		&&\quad\text{a.e.~in Q},\\
		\label{WF:KLLM:4}
		&\theta^L = -\Lapg v^L + G'(v^L) + \deln u^L ,
		&&\quad\text{a.e.~in $\Sigma$},\\
		\label{WF:KLLM:TR}
		&u^L\vert_\Sigma = v^L ,
		&&\quad\text{a.e.~in $\Sigma$},\\
		\label{WF:KLLM:5}
		&(u^L,v^L)\vert_{t=0} = (u_0,v_0) 
		&&\quad \text{a.e.~in } \Omega \times \Gamma.
		\end{alignat}
	\end{subequations}
	\item 
	The solution satisfies the global mass conservation law
	\begin{align}
	\label{MASS:KLLM}
		\beta\intO u^L(t) \dx + \intG v^L(t) \dG = \beta\intO u_0 \dx + \intG v_0 \dG = m \quad 
	\end{align}
	for all $t\in\RP$, i.e., $\big(u^L(t),v^L(t)\big)\in\Wm$ for almost all $t\in\RP$.
	
	\item The solution satisfies the energy inequality
	\begin{align}
	\label{ENERGY:KLLM}
		E\big(u^L(t),v^L(t)\big) + \frac 1 2 \int_0^t \bignorm{\big(\mu^L(s),\theta^L(s)\big)}_{L,\beta}^2 \ds 
		\le E(u_0,v_0)
	\end{align}	
	for almost all $t\ge 0$. 
	
	\item The solution satisfies the smoothing property, i.e., there exists a constant $C_*>0$ depending only on $\Omega$, $\beta$, $F$, $G$ and  $\norm{(u_0,v_0)}_{\HH^1}$, such that
	\begin{align}
	\label{SMOOTH:KLLM}
		\bignorm{ \big(u^L(t),v^L(t)\big) }_{\HH^3} 
		\leq C_* \left( \dfrac{1+t}{t}\right)^{\frac{1}{2}}
	\end{align}
	for almost all $t\ge 0$.
	\end{enumerate}
\end{theorem}

\medskip

\begin{remark}\normalfont
	We point out that the above results of Theorem~\ref{THM:WP:KLLM} hold true, if the bounds from below in the conditions \eqref{GR:F} and \eqref{GR:G} are omitted. This means, it would be sufficient to demand that
	\begin{align*}
		F(s) \le c_F(1+\abs{s}^p)
		\quad\text{and}\quad
		G(s) \le c_G(1+\abs{s}^q)
	\end{align*}
	for all $s\in\R$, even though the stronger assumptions were used in \cite{KLLM} to which we will refer in the subsequent proof. This is because instead of \cite[Lem.~2.1]{KLLM}, the stronger Poincaré type inequality
	\begin{align}
	\label{POINC:ALT}
		\norm{(u,v)}_{\LL^2} \le C_P \norm{\grad u}_{L^2(\Omega)}
		\quad\text{for all $(u,v)\in \Wo$}
	\end{align}
	could be used to prove the existence of a weak solution. The proof of \eqref{POINC:ALT} can be done similarly to the proof of \cite[Lem.~2.1]{KLLM} via a contradiction argument. It relies on the fact that if $\grad u = 0$ a.e. in $\Omega$, then the functions $u$ and $v=u\vert_\Gamma$ must be constant a.e. in $\Omega$ and a.e. on $\Gamma$, respectively. Then, the mean value condition $\beta\abs{\Omega}\mean{u}_\Omega + \abs{\Gamma}\mean{v}_\Gamma = 0$ ensures that $u = 0$ a.e. in $\Omega$ and $v=0$ a.e. on $\Gamma$. 

	Nevertheless, in Assumption~\ref{ass:pot}, we stick to the stronger conditions as demanded in \cite{KLLM} since such kind of estimates are very useful in the
	proof of Lemma~\ref{LM.boundedstat} (see \eqref{EST:F1P} and \eqref{EST:G1P}).
	
	We further point out that the above comments also apply for the GMS model ($L=0$) whose well-posedness will be discussed in Subsection~\ref{sec:wellposedGMS}.
\end{remark}

\begin{proof}[Proof of Theorem~\ref{THM:WP:KLLM}]
	The existence and uniqueness of a global unique weak solution which satisfies the assertions (i)--(iv) has already been established in \cite[Thm.~3.1]{KLLM}. 
	
	Hence, only the smoothing property (v) remains to be proved. 
	In the following, we omit the superscript $L$ to provide a cleaner presentation.
	Let us fix arbitrary test functions $(w,z)\in\HH^1$ be arbitrary.
	Adding \eqref{WF:KLLM:1} and \eqref{WF:KLLM:2}, we obtain
	\begin{align}
	\label{sumweak1} 
	\big\langle \delt u,w \big\rangle_{H^1(\Omega)} 
		+\big\langle \delt v,z \big\rangle_{H^1(\Gamma)}
		= -\big((\mu,\theta),(w,z) \big)_{L,\beta} 
	\quad\text{a.e.~on $\RP$}
	\end{align}
	for all test functions $(w,z)\in\HH^1$. 
	
	In the following, let $C$ denote a generic positive constant depending only on $\Omega$, $\beta$, $F$, $G$ and  $\norm{(u_0,v_0)}_{\HH^1}$ that may change its value from line to line.
	For any $\tau\in\RP$ and $h>0$, we will write
	\begin{align}
	\label{NOT:FDQ}
		\partial_{t}^{h}f(\tau)= \frac 1h \big(f(\tau+h)-f(\tau)\big)
	\end{align}
	to denote the forward-in-time difference quotient of a suitable function $f$.
	
	From \eqref{sumweak1}, we deduce that
	\begin{align}
	\big\langle\partial_{t}^{h}\delt u(\tau),w \big\rangle_{H^1(\Omega)} 
		+\big\langle \partial_{t}^{h}\delt v(\tau),z \big\rangle_{H^1(\Gamma)}
		= -\Big(\big(\partial_{t }^{h}\mu(\tau),\partial_{t }^{h}\theta(\tau)\big),(w,z) \Big)_{L,\beta} \label{sumweak2} 
	\end{align} 
	for almost all $\tau\in\RP$ where 
	\begin{alignat}{3}
	\partial_{t }^{h}\mu(\tau) 
	&=-\Delta\partial_{t }^{h}u(\tau)+\frac{1}{h}\Big(F^{\prime}\big(u(\tau+h)\big)-F^{\prime}\big(u(\tau)\big)\Big)\label{mu} 
	&&\quad\text{a.e.~in $\Om$,}\\
	\partial_{t }^{h}\theta(\tau)
	&=-\Delta_{\Gamma}\partial_{t}^{h}v(\tau)+\frac{1}{h}\Big(G^{\prime}\big(v(\tau+h)\big)-G^{\prime}\big(v(\tau)\big)\Big)
		+\partial_{n }\partial_{t }^{h}u(\tau) \label{teta}
	&&\quad\text{a.e.~on $\Ga$}
	\end{alignat}
	due to \eqref{WF:KLLM:3} and \eqref{WF:KLLM:4}. 
	Plugging $(w,z)=\SS^L\big(\partial_{t}^{h}u(\tau),\partial_{t}^{h}v(\tau)\big)$ into \eqref{sumweak2}, recalling \eqref{mu} and \eqref{teta}, and using integration by parts, we obtain
	\begin{align*}
	&-\dfrac{1}{2}\dfrac{\dd}{\dd \tau}\big\| \big(\partial_{t }^{h}u(\tau),\partial_{t }^{h}v(\tau) \big)\big\|_{L,\beta,*}^{2}
	= -\Big(\big(\partial_{t }^{h}\mu(\tau),\partial_{t }^{h}\theta(\tau)\big),
		\SS^L\big( \partial_{t }^{h}u(\tau), \partial_{t }^{h}v(\tau)\big) \Big)_{L,\beta}\\
	&\quad=\big(\partial_{t }^{h}\mu(\tau),\partial_{t }^{h}u(\tau) \big)_{L^2(\Omega)}
		+\big(\partial_{t }^{h}\theta(\tau),\partial_{t }^{h}v(\tau) \big)_{L^2(\Gamma)}\\
	&\quad =\big\|\nabla\partial_{t }^{h}u(\tau) \big\|^{2} _{L^{2}(\Omega)}
		+\dfrac{1}{h}\int_{\Omega }
		\Big(F^{\prime}\big(u(\tau+h)\big)-F^{\prime}\big(u(\tau)\big)\Big) \partial_{t }^{h}u(\tau) \dx\\
	&\qquad	+\big\|\gradg\partial_{t }^{h}v(\tau) \big\|^{2}_{L^{2}(\Gamma)}
		+\dfrac{1}{h}\int_{\Gamma }
		\Big(G^{\prime}\big(v(\tau+h)\big)-G^{\prime}\big(v(\tau)\big)\Big) \partial_{t }^{h}v(\tau) \dG
	\end{align*}
	for almost all $\tau\in\RP$. Hence, we have
	\begin{align*}
	&\dfrac{1}{2}\dfrac{\dd}{\dd \tau}\big\|\big(\partial_{t }^{h}u(\tau),\partial_{t }^{h}v(\tau) \big) \big\|_{L,\beta,*}^{2} 
		+\big\|\nabla\partial_{t }^{h}u(\tau) \big\|^{2} _{L^{2}(\Omega)}
		+\big\|\gradg\partial_{t }^{h}v(\tau) \big\|^{2} _{L^{2}(\Gamma)}\nonumber\\
	&\qquad +\dfrac{1}{h}\intO \Big(F_1^{\prime}\big(u(\tau+h)\big)-F_1^{\prime}\big(u(\tau)\big)\Big) \partial_{t }^{h}u(\tau)\dx \\
	&\qquad +\dfrac{1}{h}\intG \Big(G_1^{\prime}\big(v(\tau+h)\big)-G_1^{\prime}\big(v(\tau)\big)\Big) \partial_{t }^{h}v(\tau)\dG
	\\
	&=-\dfrac{1}{h}\intO\Big(F_2^{\prime}\big(u(\tau+h)\big)-F_2^{\prime}\big(u(\tau)\big)\Big) \partial_{t }^{h}u(\tau) \dx \\
	&\qquad-\dfrac{1}{h}\intG\Big(G_2^{\prime}\big(v(\tau+h)\big)-G_2^{\prime}\big(v(\tau)\big)\Big) \partial_{t }^{h}v(\tau) \dG
	\end{align*}
	for all $\tau\in\RP$.
	Recall that $ F_{2} $, $ G_{2} $ are Lipschitz continuous, and that $ F_{1} $ and $ G_{1}$ are convex which directly implies that $ F_{1}^\prime $ and $ G_{1}^\prime$ are monotone. Hence, we obtain 
	\begin{align}
	\label{est0}
	\begin{aligned}
	&\dfrac{1}{2}\dfrac{\dd}{\dd \tau}\big\|  \big(\partial_{t }^{h}u(\tau),\partial_{t }^{h}v(\tau) \big) \big\|_{L,\beta,*}^{2}
		+\big\|\nabla\partial_{t }^{h}u(\tau) \big\|^{2} _{L^{2}(\Omega)}
		+\big\|\gradg\partial_{t }^{h}v(\tau) \big\|^{2} _{L^{2}(\Gamma)}\\
	&\quad \leq C\big(\big\|\partial_{t }^{h}u(\tau) \big\|^{2} _{L^{2}(\Omega)}
		+\big\|\partial_{t }^{h}v(\tau) \big\|^{2} _{L^{2}(\Gamma)} \big) 
	\end{aligned}
	\end{align}
	for all $\tau\in\RP$.
	Using Lemma~\ref{LEM:INT} we conclude that
	\begin{align}
	\label{est1}
	\begin{aligned}
	&\dfrac{\dd}{\dd \tau}\big\| \big(\partial_{t }^{h}u(\tau),\partial_{t }^{h}v(\tau) \big) \big\|_{L,\beta,*}^{2} 
	+\big\|\nabla\partial_{t }^{h}u(\tau) \big\|^{2} _{L^{2}(\Omega)}
	+\big\|\gradg\partial_{t }^{h}v(\tau) \big\|^{2} _{L^{2}(\Gamma)}
	\\
	&\quad\leq C\big\| \big(\partial_{t }^{h}u(\tau),\partial_{t }^{h}v(\tau) \big) \big\|_{L,\beta,*}^{2} 
	\end{aligned} 
	\end{align}
	for all $\tau\in\RP$.
	Multiplying both sides by $\tau$ and integrating the resulting inequality with respect to $\tau$ from $0$ to $t$, we obtain
	\begin{align}
	t\big\| \big(\partial_{t }^{h}u(t),\partial_{t }^{h}v(t) \big) \big\|_{L,\beta,*}^{2}  
	\leq\int_{0}^{t}
	\big(C\tau+1 \big)\big\| \big(\partial_{t }^{h}u(\tau),\partial_{t }^{h}v(\tau) \big) \big\|_{L,\beta,*}^{2} \dtau 
	\label{est2}
	\end{align}
	for all $t\in\RP$. Moreover, we have
	\begin{align}
	\label{est2.1}
	&\big\| \big(\partial_{t }^{h}u(\tau),\partial_{t }^{h}v(\tau) \big) \big\|_{L,\beta,*}^{2}
	\leq\dfrac{1}{h}\int_{\tau}^{\tau+h}\big\| \big(\delt u(s),\delt v(s)\big) \big\|_{L,\beta,*}^{2} \ds, \\
	\label{est2.2}
	&\lim_{h\to 0}\dfrac{1}{h}\int_{\tau}^{\tau+h}\big\| \big(\delt u(s),\delt v(s)\big) \big\|_{L,\beta,*}^{2} \ds
	=\big\| \big(\delt u(\tau),\delt v(\tau)\big) \big\|_{L,\beta,*}^{2}.
	\end{align}
	Recalling the energy inequality \eqref{ENERGY:KLLM} and the fact that the energy $E$ is non-negative (since the potentials $F$ and $G$ are non-negative), we derive the estimate
	\begin{align}
		\int_{0}^{t}	\big\Vert \nabla \mu(\tau) \big\Vert _{L^{2}(\Omega) }^{2}
		+\big\Vert \gradg \theta(\tau) \big\Vert _{L^{2}(\Ga) }^{2}
		+\frac{1}{L}\big\|\beta\theta(\tau)-\mu(\tau) \big\|^2_{L^{2}(\Ga) } \dtau 
		\leq C. \label{est3}
	\end{align}
	By the definition of the solution operator $\SS^L$, we infer from \eqref{WF:KLLM:1} and \eqref{WF:KLLM:2} that
	\begin{align}
	\label{REP:MUTHETA}
	\mu(\tau) = \SS^L_\Om\big(\delt u(\tau),\delt v(\tau)\big) + c(\tau)\beta, \quad 
	\theta(\tau) = \SS^L_\Ga\big(\delt u(\tau),\delt v(\tau)\big) + c(\tau)
	\end{align}
	for almost all $\tau\in\RP$ and some function $c=c(\tau)$.
	Testing \eqref{WF:KLLM:1} and \eqref{WF:KLLM:2} with $\SS^L\big(\delt u(\tau),\delt v(\tau)\big)$, and adding the resulting equations, we infer that
	\begin{align}
	\label{est3.5}
	\begin{aligned}
		&\bignorm{\big(\delt u(\tau),\delt v(\tau)\big)}_{L,\beta,*}^2 
		= \big\|\big(\mu(\tau),\theta(\tau)\big)\big\|_{L,\beta}^2 \\
		&\quad = \big\Vert \nabla \mu(\tau) \big\Vert _{L^{2}(\Omega) }^{2}
			+\big\Vert \gradg \theta(\tau) \big\Vert _{L^{2}(\Ga) }^{2}
			+\frac{1}{L}\big\|\beta\theta(\tau)-\mu(\tau) \big\|^2_{L^{2}(\Ga) }
	\end{aligned}
	\end{align}
	for almost all $\tau\in\R$.
	Consequently, after integrating with respect to $\tau$ from $0$ to $t$, \eqref{est3} and \eqref{est3.5} imply
	\begin{align}
	\int_{0}^{t}\big\| \big(\delt u(\tau),\delt v(\tau)\big) \big\|_{L,\beta,*}^{2} \dtau \leq C \quad\text{for all $t\in\RP$.}\label{est4}
	\end{align}
	Therefore, recalling \eqref{est2.1}, \eqref{est2.2} and \eqref{est4}, we obtain from \eqref{est2} that
	\begin{align*}
	t\big\| \big(\partial_{t }^{h}u(t),\partial_{t }^{h}v(t)\big) \big\|_{L,\beta,*}^{2}\leq C\big(t+1 \big) \quad\text{for all $t\in\RP$.}
	\end{align*}
	Invoking Lebesgue's convergence theorem and passing to the limit $ h\to 0 $, we get
	\begin{align}
	\label{est:u*}
	\big\| \big(\delt u(t),\delt v(t)\big) \big\|_{L,\beta,*}^{2}\leq C\left(\dfrac{t+1}{t} \right) \quad\text{for all $t\in\RP$},
	\end{align}
	which in turn yields
	\begin{align}
	\big\| \nabla\mu(t)\big\|_{L^{2}(\Omega)}^{2}
		+\big\| \gradg\theta(t)\big\|_{L^{2}(\Gamma)}^{2}\leq C\left(\dfrac{t+1}{t} \right)\label{est5}
	\quad\text{for almost all $t\in\RP$}
	\end{align}
	because of \eqref{est3.5}. Proceeding as in \cite[Proof of Thm.~3.1, Step 3]{KLLM} we obtain the estimates
	\begin{align}
	\label{est:mu}
	\big\|\mu(t) \big\|_{L^{2}(\Omega)}
	&\leq C\big( 1+\big\|\nabla\mu(t) \big\|_{L^{2}(\Omega)}\big),\\
	\label{est:theta}
	\big\|\theta(t) \big\|_{L^{2}(\Gamma)}
	&\leq C\big( 1+\big\|\nabla\mu(t) \big\|_{L^{2}(\Omega)}
	+\big\|\gradg\theta(t) \big\|_{L^{2}(\Gamma)}\big)
	\end{align}
	for almost all $t\in\RP$. This directly implies that
	\begin{align}
		\label{est:mutheta}
		\bignorm{\big(\mu(t),\theta(t)\big)}_{\HH^1} \le C\left(\dfrac{t+1}{t} \right)^{\frac 12} 
	\end{align}
	for almost all $t\in\RP$. Furthermore, recalling the growth assumptions on $F$ and $G$ imposed in \eqref{ass:pot}, we conclude from the energy inequality \eqref{ENERGY:KLLM} that
	\begin{align}
	\label{est:h1}
		\bignorm{\big(u(t),v(t)\big)}_{\HH^1} 
		\le C + CE\big(u(t),v(t)\big)
		\le C + CE\big(u_0,v_0\big)
		\le C
	\end{align}
	for almost all $t\in\RP$.
	Since $(u,v) \in L^2(0,T;\HH^3)$ for all $T>0$, we have
	\begin{alignat}{3}
	\label{sf1}
	-\Delta u(t)&=\mu(t)-F^{\prime}\big(u(t)\big) 
	&&\quad\text{a.e.~in $\Om$},\\
	\label{sf2}
	- \Delta_{\Gamma} v(t) + \partial_{n }u(t) &=\theta(t)-G^{\prime}\big(v(t)\big)
	&&\quad\text{a.e.~on $\Ga$}
	\end{alignat}
	for almost all $t\in\RP$. 
	Invoking the growth assumptions on $F$ and $G$ in \eqref{ass:pot}, recalling that $p\le 4$, 
	and using the continuous embeddings $H^1(\Omega)\emb L^6(\Omega)$ and $H^1(\Gamma)\emb L^r(\Gamma)$ for all $r\in[1,\infty)$, we further deduce that for almost all $t\in\RP$, 
	\begin{align*}
		\bignorm{F'\big(u(t)\big)}_{L^2(\Omega)} 
		&\le C + C \norm{u(t)}_{H^1(\Omega)}^{p-1} \le C,\\
		\bignorm{G'\big(v(t)\big)}_{L^2(\Gamma)} 
		&\le C + C \norm{v(t)}_{H^1(\Gamma)}^{q-1} \le C,\\
		\bignorm{F''\big(u(t)\big)}_{L^3(\Omega)} 
		&\le C + C \norm{u(t)}_{H^1(\Omega)}^{p-2} \le C,\\
		\bignorm{G''\big(v(t)\big)}_{L^3(\Gamma)} 
		&\le C + C \norm{v(t)}_{H^1(\Gamma)}^{q-2} \le C.
	\end{align*}
	Using Hölder's inequality, the continuous embedding $H^2(\Omega)\emb W^{1,6}(\Omega)$, interpolation inequality \eqref{INT:H1:ALT}, estimate \eqref{est:h1}, and  Young's inequality, we thus infer that for any $\alpha\in (0,1)$ and almost all $t\in\RP$,
	\begin{align}
	\label{est:f}
		&\bignorm{F'\big(u(t)\big)}_{H^1(\Omega)}^2
		= \bignorm{F'\big(u(t)\big)}_{L^2(\Omega)}^2 
			+ \bignorm{F''\big(u(t)\big) \grad u(t)}_{L^2(\Omega)}^2 
		\notag\\
		&\le C + C \bignorm{F''\big(u(t)\big)}_{L^3(\Omega)}^2 
			\bignorm{\grad u(t)}_{L^6(\Omega)}^2 
		\le C + C \bignorm{u(t)}_{H^1(\Omega)}\, \bignorm{u(t)}_{H^3(\Omega)}
		\notag\\
		&\le \frac{C}{\alpha} + \alpha \bignorm{u(t)}_{H^3(\Omega)}^2.
	\end{align}
	Proceeding similarly, we deduce that for any $\alpha\in (0,1)$ and for any $\alpha>0$ and almost all $t\in\RP$,
	\begin{align}
	\label{est:g}
		\bignorm{G'\big(v(t)\big)}_{H^1(\Gamma)}^2
		\le \frac{C}{\alpha} + \alpha \bignorm{v(t)}_{H^3(\Gamma)}^2.
	\end{align}
	Furthermore, a regularity estimate for elliptic problems with bulk-surface coupling (see \cite[Thm.~3.3]{knopf-liu}) implies that for almost all $t\in\RP$,
	\begin{align*}
		\bignorm{\big(u(t),v(t)\big)}_{\HH^3}^2
		&\le C \bignorm{\mu(t) - F'\big(u(t)\big)}_{H^1(\Omega)}^2 + C \bignorm{\theta(t) - G'\big(u(t)\big)}_{H^1(\Gamma)}^2 \\
		&\le C \bignorm{\big(\mu(t),\theta(t)\big)}_{\HH^1}^2 
		+ C\bignorm{F'\big(u(t)\big)}_{H^1(\Omega)}^2
		+ C\bignorm{G'\big(v(t)\big)}_{H^1(\Gamma)}^2 .
	\end{align*}
	Eventually, invoking the estimates \eqref{est:mutheta}, \eqref{est:f} and \eqref{est:g}, and choosing $\alpha\in(0,1)$ sufficiently small, we conclude that for almost all $t\in\RP$,
	\begin{align*}
	\bignorm{\big(u(t),v(t)\big)}_{\HH^3}^2
	\le C + C \bignorm{\big(\mu(t),\theta(t)\big)}_{\HH^1}^2  
	\le C\left(\frac{t+1}{t}\right),
	\end{align*}
	which verifies the smoothing property \eqref{SMOOTH:KLLM}. 
	Thus, the proof of Theorem~\ref{THM:WP:KLLM} is complete.
\end{proof}

\medskip

Using the interpolation inequalities from Section~\ref{SEC:PIT} and the energy inequality \eqref{ENERGY:KLLM}, we further establish improved continuity results for the weak solution.

\begin{proposition}[Improved continuity results] \label{PROP:IMP:KLLM}
	Suppose that \eqref{ass:dom}--\eqref{ass:pot} hold and that $L\in(0,\infty)$. Let $m\in\R$ be arbitrary,
	let $(u_{0},v_{0}) \in \Wm $ be any initial datum, and let $(u^L,v^L,\mu^L,\theta^L)$ denote the corresponding weak solution of the system \eqref{CH:INT.}.
	Then it holds that%
	\begin{align}
		\label{CONT:UV}
		(u^L,v^L) \in C([0,\infty);\HH^1) \cap C((0,\infty);\HH^2).
	\end{align}
\end{proposition}

\begin{proof}
	\textit{Step 1.}
	We first prove that $(u^L,v^L) \in C((0,\infty);\HH^2)$. Therefore, we fix an arbitrary time $t_0\in (0,\infty)$. Let further $t\in (0,\infty)$ with $\frac 12 t_0 < t < t_0+1$ be arbitrary. We already know from Theorem~\ref{THM:WP:KLLM} that $$(u^L,v^L) \in L^\infty\big(\tfrac 12 t_0, t_0+1 ;\HH^2\big) \cap C\big([\tfrac 12 t_0, t_0+1];\LL^2\big).$$ This entails that $(u^L,v^L)(t) \in \HH^2$ for all $t\in [\tfrac 12 t_0, t_0+1]$ (see, e.g., \cite[Chap.~III,\,§1,\,Lem.~1.4.]{temam}). 
	Using Lemma~\ref{LEM:INT:2} and the smoothing property \eqref{SMOOTH:KLLM} we deduce that
	\begin{align*}
	&\bignorm{ \big(u^L(t),v^L(t)\big) - \big(u^L(t_0),v^L(t_0)\big) }_{\HH^2}^2 \\
	&\quad \le C
	\bignorm{ \big(u^L(t),v^L(t)\big) 
		- \big(u^L(t_0),v^L(t_0)\big) }_{\HH^3}^{\frac 43}
	\bignorm{ \big(u^L(t),v^L(t)\big) 
		- \big(u^L(t_0),v^L(t_0)\big) }_{\LL^2}^{\frac 13} \\
	&\quad \le C\left[\left(\frac{2(2+t_0)}{t_0}\right)^{\frac 12} + \left(\frac{1+t_0}{t_0}\right)^{\frac 12} \right]^{\frac 43}
	\bignorm{ \big(u^L(t),v^L(t)\big) 
		- \big(u^L(t_0),v^L(t_0)\big) }_{\LL^2}^{\frac 13}.
	\end{align*}
	We now choose $T:= 1 + t_0$ and recall that $(u^L,v^L) \in C([0,T];\LL^2)$.
	Hence, the last line of the above estimate tends to zero as $t\to t_0$ which proves $(u^L,v^L) \in C((0,\infty);\HH^2)$.
	
	\textit{Step 2.}
	We now want to show that $(u^L,v^L) \in C([0,\infty);\HH^1)$. As $(u^L,v^L) \in C((0,\infty);\HH^1)$ is already established, only the continuity at $t=0$ needs to be investigated. 
	To this end, let $\{t_k\}_{k\in\N}\subset (0,\infty)$ be an arbitrary sequence with $t_k\to 0$ as $k\to \infty$. 
	
	From the growth conditions on $F$ and $G$ in \eqref{ass:pot}, we deduce that
	\begin{align}
	\bignorm{\big(u(t_k),v(t_k)\big)}_{\HH^1} 
	\le C + C E\big(u(t_k),v(t_k)\big)
	\le C + C E(u_0,v_0) 
	\le C
	\end{align}
	for all $k\in\N$. Hence, there exists a subsequence $\{t_k^*\}_{k\in\N}$ of $(t_k)_{k\in\N}$ and a pair $(u_*,v_*)\in \HH^1$ such that 
	\begin{align}
	\label{CONV:UV:W*}
		\big(u^L(t_k^*),v^L(t_k^*)\big) \wto (u_*,v_*) \quad\text{in $\HH^1$.}
	\end{align}
	Since $(u^L,v^L) \in C([0,\infty);\LL^2)$, we also know that 
	\begin{align}
	\label{CONV:UV:S}
		\big(u^L(t_{k}),v^L(t_{k})\big) \to (u_0,v_0) \quad\text{in $\LL^2$}
	\end{align}
	and thus $(u_*,v_*)=(u_0,v_0)$ a.e.~in $\Omega\times\Gamma$ due to uniqueness of the limit. This means that the weak limit in \eqref{CONV:UV:W*} does not depend on the subsequence extraction and thus, \eqref{CONV:UV:W*} holds true for the \emph{whole} sequence, that is 
	\begin{align}
	\label{CONV:UV:W}
		\big(u^L(t_k),v^L(t_k)\big) \wto (u_0,v_0) \quad\text{in $\HH^1$.}
	\end{align}
	With some abuse of notation, we will write
	$E(t) := E\big(u(t),v(t)\big)$ for all $t\ge 0$.
	We further recall that $F_1$ and $G_1$ are convex, and that $F_2$ and $G_2$ have at most quadratic growth. Using Lebesgue's general convergence theorem (see \cite[Sect.~3.25]{Alt}) for the non-convex contributions, and lower-semicontinuity of the convex contributions, we obtain
	\begin{align}
	\label{LIMINF:E}
		E(t_0) \le \underset{k\to\infty}{\lim\inf}\; E(t_k).
	\end{align}
	Moreover, we infer from the energy inequality \eqref{ENERGY:KLLM} that the sequence $(E(t_k))_{k\in\N}$ is non-decreasing and bounded from above by $E(t_0)$. Hence, the limit $k\to \infty$ exists, and in combination with \eqref{LIMINF:E} we conclude that
	\begin{align}
	\label{LIM:E}
		E(t_k) \to E(t_0)\quad\text{as $k\to\infty$.}
	\end{align}
	By the definition of the energy functional $E$ (see \eqref{DEF:EN}), this implies that
	\begin{align}
	\label{ID:NRG}
	\begin{aligned}
		&\frac 12 \bignorm{ \big(\grad u^L(t_k), \gradg v^L(t_k)\big) }_{\LL^2}^2 
			- \frac 12 \bignorm{\big(\grad u_0, \gradg v_0\big)}_{\LL^2}^2 \\
		&= E(t_k) - E(t_0) 
			+ \intO F\big(u_0\big) - F\big(u^L(t_k)\big) \dx  
			+ \intG G\big(v_0\big) - G\big(v^L(t_k)\big) \dG .
	\end{aligned}
	\end{align}
	Since $(u(t_k),v(t_k))\to (u_0,v_0)$ strongly in $\LL^2$, and $(u(t_k),v(t_k))_{k\in\N}$ is bounded in $\HH^1$, we also have $(u(t_k),v(t_k))\to (u_0,v_0)$ strongly in $L^p(\Omega)\times L^q(\Gamma)$ via interpolation.
	We can thus use Lebesgue's general convergence theorem (see \cite[Sect.~3.25]{Alt}) to infer that the right-hand side of \eqref{ID:NRG} tends to zero as $k\to\infty$. Along with \eqref{CONV:UV:S}, this yields
	\begin{align}
	\label{CONV:UV:NORM}
		\bignorm{ \big(u^L(t_k), v^L(t_k)\big) }_{\HH^1}^2 
			- \bignorm{\big(u_0, v_0\big)}_{\HH^1}^2 \to 0
		\quad\text{$k\to\infty$.}
	\end{align}
	Combining \eqref{CONV:UV:W} and \eqref{CONV:UV:NORM}, we eventually conclude that
	$(u^L(t_k),v^L(t_k))$ converges to $(u_0,v_0)$ strongly in $\HH^1$ as $k\to\infty$. As the null sequence $(t_k)_{k\in\N}$ was arbitrary, this proves $(u^L,v^L)\in C([0,\infty);\HH^1)$. 
	
	Thus, the proof is complete.
\end{proof}

\medskip
	
Moreover, we deduce continuous dependence of the weak solutions on the initial data.

\begin{proposition}[Continuous dependence on the initial data]\label{PROP:CD:KLLM}
	Suppose that \eqref{ass:dom}--\eqref{ass:pot} hold and that $L\in(0,\infty)$. Let $m\in\R$ be arbitrary.
	For $i\in\{1,2\}$, let $(u_{0,i},v_{0,i}) \in \Wm $ be any initial datum and let $(u_i^L,v_i^L,\mu_i^L,\theta_i^L)$ denote the corresponding weak solution of the system \eqref{CH:INT.}. 
	
	Then there exist positive, non-decreasing functions $\Lambda_k^* \in C(\RP)$, $k\in\{-1,1,2\}$ depending only on $\Omega$, $\beta$, $F$ and $G$ such that the following statements hold:
	\begin{alignat}{2}
		\label{CD:KLLM:1}
		\bignorm{\big(u^L_{2},v_2^L\big)(t) - \big(u^L_{1},v_1^L\big)(t)}_{L,\beta,*}
		&\le \Lambda_{-1}^*(t) \bignorm{(u_{0,2},v_{0,2})-(u_{0,1},v_{0,1})}_{L,\beta,*},
		&&\;\; t\ge 0, 
		\\[0.5ex]
		\label{CD:KLLM:2}
		\bignorm{\big(u^L_{2},v_2^L\big)(t) - \big(u^L_{1},v_1^L\big)(t)}_{\HH^1} 
		&\le \Lambda_1^*(t) \bignorm{(u_{0,2},v_{0,2})-(u_{0,1},v_{0,1})}_{L,\beta,*},
		&&\;\; t\ge 1. 
		\\
		\label{CD:KLLM:3}
		\bignorm{\big(u^L_{2},v_2^L\big)(t) - \big(u^L_{1},v_1^L\big)(t)}_{\HH^2} 
		&\le \Lambda_2^*(t) \bignorm{(u_{0,2},v_{0,2})-(u_{0,1},v_{0,1})}_{L,\beta,*}^{\frac 12},
		&&\;\; t\ge 1.
	\end{alignat}
\end{proposition}

\begin{proof}
	Let $C$ denote generic positive constants depending only on $\Omega$, $\beta$, $F$ and $G$.
	For brevity, we write
	\begin{align*}
		(\bar u,\bar v,\bar \mu,\bar \theta) := (u^L_2,v^L_2,\mu^L_2,\theta^L_2) - (u^L_1,v^L_1,\mu^L_1,\theta^L_1), \quad
		(\bar u_0, \bar v_0) = ( u_{0,2} , v_{0,2}) - ( u_{0,1}, v_{0,1} ).
	\end{align*}
	From the weak formulation \eqref{WF:KLLM:1}--\eqref{WF:KLLM:4}, we directly obtain that
	\begin{align}
		\label{CD:EQ:1}
		\inn{\delt \bar u(t)}{w}_{H^1(\Omega)} + \inn{\delt \bar v(t)}{z}_{H^1(\Gamma)} 
		= \Big(\big(\bar\mu(t),\bar\theta(t)\big),(w,z)\Big)_{L,\beta}
	\end{align}
	for almost all $t\in\RP$ and all $(w,z)\in\HH^1$, and
	\begin{alignat}{2}
		\label{CD:EQ:2}
		&\bar\mu(t) = -\Lap\bar u(t) + F'\big(u_2^L(t)\big) - F'\big(u_1^L(t)\big) &&\quad\text{a.e.~in $\Omega$} ,\\
		\label{CD:EQ:3}
		&\bar\theta(t) = -\Lapg\bar v(t) + G'\big(v_2^L(t)\big) - G'\big(v_1^L(t)\big) + \deln\bar u(t) &&\quad\text{a.e.~in $\Gamma$}
	\end{alignat}
	for almost all $t\in\RP$.  
	Arguing similarly as in the proof of Theorem~\ref{THM:WP:KLLM} (cf. the derivation of \eqref{est1}), we establish the estimate
	\begin{align}
	\label{EST:GW1}
		\dfrac{\dd}{\dd \tau} \bignorm{ \big(\bar u(\tau),\bar v(\tau)\big) }_{L,\beta,*}^{2} 
		+ \bignorm{ \big(\grad \bar u(\tau),\gradg \bar v(\tau)\big) }_{\LL^2}^{2}
		\leq C\, \bignorm{ \big(\bar u(\tau),\bar v(\tau)\big) }_{L,\beta,*}^{2} 
	\end{align}
	for all $\tau\ge 0$. Hence, Gronwall's lemma implies that 
	\begin{align}
	\label{EST:GW2}
	\bignorm{ \big(\bar u(t),\bar v(t)\big) }_{L,\beta,*}^{2} 
	+ \int_0^t \bignorm{ \big(\grad \bar u(\tau),\gradg \bar v(\tau)\big) }_{\LL^2}^{2} \ds
	\leq C e^{Ct}\, \bignorm{ \big(\bar u(0),\bar v(0)\big) }_{L,\beta,*}^{2} 
	\end{align}
	for all $t\ge  0$. In particular, this proves \eqref{CD:KLLM:1}.
	
	Let now $T>1$ be arbitrary. Moreover, let
	$\varphi \in C^\infty_c(\R;\RP)$ with $\mathrm{supp}\, \varphi \subset (-1,0)$ and $\norm{\varphi}_{L^1(\R)}=1$ be an arbitrary function. 
	For any $k\in\N$, we set
	\begin{align*}
		\varphi_k(s) := k \varphi(ks) 
		\quad\text{for all $s\in\R$.}
	\end{align*}
	This defines functions $\varphi_k \in C^\infty_c(\R;\RP)$ with $\mathrm{supp}\, \varphi_k \subset (-\frac 1k,0)$ and $\norm{\varphi_k}_{L^1(\R)}=1$ for all $k\in\mathbb N$. For any Banach space $X$, any function $f\in L^1([0,T];X)$ and all $k\in\N$, we now define
	\begin{align*}
		\sm{f}_k:\RP \to X, \quad \sm{f}_k(t):= \int_t^{t+1} \varphi_k(t-s)\, f(s) \ds .
	\end{align*} 
	It is well-known (cf.~\cite[Sect.~4.13]{Alt}) that if $f\in L^r([0,T+1];X)$ for any real number $r\in[1,\infty)$, then $\sm{f}_k \in C^\infty([0,T];X)\cap L^r([0,T];X)$ for all $k\in\N$, and $\sm{f}_k\to f$ in $L^r([0,T];X)$ as $k\to\infty$.
	
	Using this smooth approximation, we obtain 
	\begin{align}
	\label{CD:EQ:1k}
		\biginn{ \sm{ \delt \bar u }_k(t) }{w}_{H^1(\Omega)} 
		+ \biginn{ \sm{ \delt \bar v }_k(t) }{z}_{H^1(\Gamma)} 
		= \Big(\big( \sm{\bar\mu }_k(t), \sm{\bar\theta}_k (t)\big),(w,z)\Big)_{L,\beta}, 
	\end{align}
	for almost all $t\in(0,T)$, where
	\begin{alignat}{2}
	\label{CD:EQ:2k}
	\sm{\bar\mu}_k &= -\Lap \sm{\bar u}_k
	+ \big[F'(u_2^L) - F'(u_1^L)\big]_k 
	&&\quad\text{a.e.~in $Q$} ,\\
	\label{CD:EQ:3k}
	\sm{\bar\theta}_k  &= -\Lapg \sm{\bar v}_k 
	+ \big[G'(v_2^L) - G'(v_1^L)\big]_k + \deln \sm{\bar u}_k
	&&\quad\text{a.e.~in $\Sigma$.}
	\end{alignat}
	Note that $\delt \sm{\bar u}_k = \sm{\delt \bar u}_k$ and $\delt \sm{\bar v}_k = \sm{\delt \bar v}_k$, which can be verified by using the change of variables $s\mapsto t-s$ and Leibniz's rule for differentiation under the integral sign. For more details, see \cite[Sect.~4.13]{Alt}.
	
	Let us now choose an arbitrary number $t_0\in(0,1)$, and let $C_0$ denote generic positive constants depending only on $\Omega$, $\beta$, $F$, $G$ and $t_0$. 
	We thus obtain
	\begin{align*}
		&\frac 12 \frac{\mathrm d}{\mathrm d\tau}
		\Big( (\tau-t_0) \bignorm{\grad \sm{\bar u}_k (\tau)}_{L^2(\Omega)}^2
		+ (\tau-t_0) \bignorm{\gradg \sm{\bar v}_k (\tau)}_{L^2(\Gamma)}^2 \Big) 
		\\[1ex]
		&= \frac 12 \bignorm{\grad \sm{\bar u}_k (\tau)}_{L^2(\Omega)}^2
		+ \frac 12 \bignorm{\gradg \sm{\bar v}_k (\tau)}_{L^2(\Gamma)}^2 \\
		&\quad + (\tau-t_0) \bigscp{\grad \sm{\bar u}_k (\tau)}{\delt \grad \sm{\bar u}_k (\tau)}_{L^2(\Omega)}
		+ (\tau-t_0) \bigscp{\gradg \sm{\bar v}_k (\tau)}{\delt \gradg \sm{\bar v}_k (\tau)}_{L^2(\Gamma)} 
		\\[1ex]
		&= \frac 12 \bignorm{\grad \sm{\bar u}_k (\tau)}_{L^2(\Omega)}^2
		+ \frac 12 \bignorm{\gradg \sm{\bar v}_k (\tau)}_{L^2(\Gamma)}^2 \\
		&\quad + (\tau-t_0) \bigscp{\sm{\bar \mu}_k (\tau)}{\delt \sm{\bar u}_k (\tau)}_{L^2(\Omega)}
			+ (\tau-t_0) \bigscp{\sm{\bar \theta}_k (\tau)}{\delt \sm{\bar v}_k (\tau)}_{L^2(\Gamma)} \\
		&\quad - (\tau-t_0) \bigscp{\big[F'(u_2^L) - F'(u_1^L)\big]_k(\tau)}{\delt \sm{\bar u}_k (\tau)}_{L^2(\Omega)}\\ 
		&\quad - (\tau-t_0) \bigscp{\big[G'(v_2^L) - G'(v_1^L)\big]_k(\tau)}{\delt \sm{\bar v}_k (\tau)}_{L^2(\Gamma)} \\[1ex]
		&= \frac 12 \bignorm{\grad \sm{\bar u}_k (\tau)}_{L^2(\Omega)}^2
		+ \frac 12 \bignorm{\gradg \sm{\bar v}_k (\tau)}_{L^2(\Gamma)}^2 
		- (\tau-t_0) \bignorm{\big(\delt \sm{\bar u}_k (\tau),\delt \sm{\bar v}_k (\tau)\big)}_{L,\beta,*}^2\\
		&\quad - (\tau-t_0) \bigscp{\big[F'(u_2^L) - F'(u_1^L)\big]_k(\tau)}{\delt \sm{\bar u}_k (\tau)}_{L^2(\Omega)} \\
		&\quad
		- (\tau-t_0) \bigscp{\big[G'(v_2^L) - G'(v_1^L)\big]_k(\tau)}{\delt \sm{\bar v}_k (\tau)}_{L^2(\Gamma)} 
	\end{align*}
	for all $t_0<\tau<T$.
	Integrating with respect to $\tau$ from $t_0$ to an arbitrary $t\in [t_0,T]$, we infer that
	\begin{align*}
		&\begin{aligned}
		&\frac 12 \Big( (t-t_0) \bignorm{\grad \sm{\bar u}_k (t)}_{L^2(\Omega)}^2
			+ (t-t_0)  \bignorm{\gradg \sm{\bar v}_k (t)}_{L^2(\Gamma)}^2 \Big)
		\\
		&\quad + \int_{t_0}^t (\tau-t_0) \bignorm{\big(\delt \sm{\bar u}_k (\tau), 
			\delt \sm{\bar v}_k (\tau)\big)}_{L,\beta,*}^2 \dtau
		\end{aligned}
		\\[1ex]
		&\begin{aligned}
		&=  \frac 12 \int_{t_0}^t \bignorm{\grad \sm{\bar u}_k (\tau)}_{L^2(\Omega)}^2 \dtau
		+ \frac 12 \int_{t_0}^t \bignorm{\gradg \sm{\bar v}_k (\tau)}_{L^2(\Gamma)}^2 \dtau\\
		&\quad - \int_{t_0}^t (\tau-t_0) \biginn{\delt \sm{\bar u}_k (\tau)}{\big[F'(u_2^L) - F'(u_1^L)\big]_k(\tau)}_{H^1(\Omega)} \dtau \\
		&\quad - \int_{t_0}^t (\tau-t_0) \biginn{\delt \sm{\bar v}_k (\tau)}{\big[G'(v_2^L) - G'(v_1^L)\big]_k(\tau)}_{H^1(\Gamma)} \dtau
		\end{aligned}
	\end{align*}
	for all $t\in [t_0,T]$.
	Passing to the limit $k\to\infty$, we obtain 
	\begin{align}
	\label{EST:CLIP:1}
		&\frac 12 \Big((t-t_0) \bignorm{\grad \bar u(t)}_{L^2(\Omega)}^2
		+ (t-t_0) \bignorm{\gradg \bar v(t)}_{L^2(\Gamma)}^2 \Big) 
		\notag\\
		&\;\; + \int_{t_0}^t (\tau-t_0) \bignorm{\big(\delt \bar u, 
			\delt \bar v\big)(\tau)}_{L,\beta,*}^2 \dtau 
		\notag\\[1ex]
		&=  \frac 12 \int_{t_0}^t \bignorm{\grad \bar u(\tau)}_{L^2(\Omega)}^2 \dtau
		+ \frac 12 \int_{t_0}^t \bignorm{\gradg \bar v(\tau)}_{L^2(\Gamma)}^2 \dtau
		\notag\\
		&\;\; - \int_{t_0}^t (\tau-t_0) \biginn{\delt \bar u(\tau)}{F'\big(u_2^L(\tau)\big) - F'(u_1^L\big(\tau)\big)}_{H^1(\Omega)} \dtau 
		\notag\\
		&\;\; - \int_{t_0}^t (\tau-t_0) \biginn{\delt \bar v(\tau)}{G'\big(v_2^L(\tau)\big) - G'\big(v_1^L(\tau)\big)}_{H^1(\Gamma)} \dtau
		\notag\\[1ex]
		&=  \frac 12 \int_{t_0}^t \bignorm{\grad \bar u (\tau)}_{L^2(\Omega)}^2 \dtau
		+ \frac 12 \int_{t_0}^t \bignorm{\gradg \bar v (\tau)}_{L^2(\Gamma)}^2 \dtau
		\notag\\
		&\;\; + \int_{t_0}^t (\tau-t_0) 
		\Big(
		\SS^L \big(
		(\delt \bar u,
		\delt \bar v
		)(\tau)\big),
		\big[
		F'\big(u_2^L(\tau)\big) - F'\big(u_1^L(\tau)\big),
		G'\big(v_2^L(\tau)\big) - G'\big(v_1^L(\tau)\big)
		\big]
		\Big)_{L,\beta} \mathrm d\tau 
		\notag\\[1ex]
		&\le \frac 12 \int_{t_0}^t \bignorm{\grad \bar u (\tau)}_{L^2(\Omega)}^2 \dtau
		+ \frac 12 \int_{t_0}^t \bignorm{\gradg \bar v (\tau)}_{L^2(\Gamma)}^2 \dtau
		\notag\\
		&\;\; + C\int_{t_0}^t 
		(\tau-t_0) \Big(\bignorm{ F'\big(u_2^L(\tau)\big) - F'\big(u_1^L(\tau)\big) }_{H^1(\Omega)}^2
		+  \bignorm{ G'\big(v_2^L(\tau)\big) - G'\big(v_1^L(\tau)\big) }_{H^1(\Gamma)}^2\Big)
		\dtau
		\notag\\
		&\;\; + \frac 12 \int_{t_0}^t		
		(\tau-t_0) \bignorm{\big(
		\delt \bar u ,
		\delt \bar v 
		\big)(\tau)}_{L,\beta,*}^2 \dtau 
	\end{align}
	for all $t\in [t_0,T]$. Since $T>0$ was arbitrary, \eqref{EST:CLIP:1} actually holds true for all $t\ge t_0$.
	Invoking the smoothing property \eqref{SMOOTH:KLLM} and Sobolev's embedding theorem, we infer that
	\begin{align}
	\bignorm{ u_i^L(\tau) }_{W^{1,\infty}(\Omega)} + \bignorm{ v_i^L(\tau) }_{W^{1,\infty}(\Gamma)} 
	\le C \bignorm{ \big(u_i^L(\tau),v_i^L(\tau)\big) }_{\HH^3} 
	\leq C \left( \dfrac{1+t_0}{t_0}\right)^{\frac{1}{2}} 
	\end{align}
	for almost all $\tau\ge t_0$. By a straightforward computation, this entails that
	\begin{align}
	\label{EST:F:DIFF}
		\bignorm{ F'\big(u_2^L(\tau)\big) - F'\big(u_1^L(\tau)\big) }_{H^1(\Omega)}^2
		&\le C_0 \norm{\bar u(\tau)}_{H^1(\Omega)}^2,\\
	\label{EST:G:DIFF}
		\bignorm{ G'\big(v_2^L(\tau)\big) - G'\big(v_1^L(\tau)\big) }_{H^1(\Gamma)}^2
		&\le C_0 \norm{\bar v(\tau)}_{H^1(\Gamma)}^2,
	\end{align}
	for almost all $\tau \ge t_0$ and $i\in\{1,2\}$.	Plugging the estimates \eqref{EST:F:DIFF} and \eqref{EST:G:DIFF} into \eqref{EST:CLIP:1}, and applying Lemma~\ref{LEM:INT}, we deduce that
	\begin{align}
		\label{EST:CLIP:2}
		&(t-t_0) \bignorm{\big(\grad \bar u, \gradg \bar v\big)(t)}_{\LL^2}^2
		\notag\\
		&\quad\le C_0 \int_0^t \bignorm{\big(\grad \bar u, \gradg \bar v\big)(\tau)}_{\LL^2}^2 \dtau
			+ C_0 \int_{t_0}^t (\tau-t_0) \,\bignorm{\big(\grad \bar u, \gradg \bar v\big)(\tau)}_{\LL^2}^2 \dtau
		\notag\\
		&\qquad + C_0 \int_{t_0}^t (\tau-t_0) \, \bignorm{\big(\bar u, \bar v\big) (\tau)}_{L,\beta,*}^2 \dtau
	\end{align}
	for all $t\ge t_0$.
	Recalling \eqref{EST:GW2}, we infer that 
	\begin{align}
	\label{EST:CLIP:3}
	&(t-t_0) \bignorm{\big(\grad \bar u, \gradg \bar v\big)(t)}_{\LL^2}^2
	\notag\\
	&\quad\le C_0 \int_{t_0}^t (\tau-t_0) \,\bignorm{\big(\grad \bar u, \gradg \bar v\big)(\tau)}_{\LL^2}^2 \dtau
		+ C_0\, e^{Ct}\, \big(1+ (t-t_0)^2\big)\, \bignorm{\big(\bar u, \bar v\big) (0)}_{L,\beta,*}^2 
	\end{align}
	for all $t\ge t_0$.
	We now fix $t_0:=\tfrac 12$. Hence, Gronwall's lemma eventually implies that
	\begin{align}
	\label{EST:CLIP:3}
	\bignorm{\big(\grad \bar u, \gradg \bar v\big)(t)}_{\LL^2}^2
	\le C\, e^{Ct}\, \frac{\big(1+ (t-\tfrac 12)^2\big)}{t-\tfrac 12}\, \bignorm{\big(\bar u, \bar v\big) (0)}_{L,\beta,*}^2 
	\end{align}
	for all $t\ge \tfrac 12$. 
	In view of \eqref{CD:KLLM:1} and Lemma~\ref{LEM:INT}, the compact Lipschitz estimate \eqref{CD:KLLM:2} now directly follows.

	Using Lemma~\ref{LEM:INT:2}, the smoothing property \eqref{SMOOTH:KLLM} and the estimate \eqref{CD:KLLM:2}, we further conclude that
	\begin{align*}
	\bignorm{ \big(\bar u(t),\bar v(t)\big) }_{\HH^2}^2 
	\le C\, \bignorm{ \big(\bar u(t),\bar v(t)\big) }_{\HH^3}
	\bignorm{ \big(\bar u(t),\bar v(t)\big) }_{\HH^1} 
	\le C\, \Lambda_{1}^*(t) \bignorm{ \big(\bar u_0,\bar v_0\big) }_{L,\beta,*}
	\end{align*}
	for almost all $t\ge 1$.
	This proves \eqref{CD:KLLM:3} and thus, the proof is complete.
\end{proof}

We can further show that the weak solution is Hölder continuous in time with a Hölder constant being uniform in $L$.

\begin{proposition}[Uniform Hölder estimate]\label{PROP:HLD:KLLM}
	Suppose that \eqref{ass:dom}--\eqref{ass:pot} hold and that $L\in(0,\infty)$. Let $m\in\R$ be arbitrary, let $(u_{0},v_{0}) \in \Wm $ be any initial datum, and let $(u^L,v^L,\mu^L,\theta^L)$ denote the corresponding weak solution of the system \eqref{CH:INT.}. 
	
	Then there exists a positive constant $\Theta^*$ depending only on $\Omega$ and $E(u_0,v_0)$ such that 
	\begin{alignat}{2}
		\bignorm{(u^L,v^L)(t) - (u^L,v^L)(s)}_{\HH^1} 
		\le \Theta^*\, \abs{s-t}^{\frac{3}{16}},
		\quad \text{for all $s,t\ge 1$.}
	\end{alignat}
\end{proposition}

\begin{proof}
	Let $C$ denote a generic positive constant depending only on $\Omega$ and $E(u_0,v_0)$ that may change its value from line to line.
	Let $s,t\ge 0$ be arbitrary, and without loss of generality we assume $t>s$. Moreover, let $w\in H^1_0(\Omega)$ (i.e., $w\vert_\Gamma = 0$)  be an arbitrary test function.
	Integrating \eqref{WF:KLLM:1} with respect to time, we get
	\begin{align*}
		\biginn{u^L(t)-u^L(s)}{w}_{H^1(\Omega)}
		= - \int_s^t \grad\mu^L(\tau) \cdot \grad w \dtau.
	\end{align*} 
	Now, taking the supremum over all test functions $w\in H^1_0(\Omega)$ with $\norm{w}_{H^1(\Omega)} = 1$, and recalling the energy inequality \eqref{ENERGY:KLLM}, we infer that
	\begin{align*}
		&\bignorm{u^L(t)-u^L(s)}_{H^{-1}(\Omega)} 
		\le \int_s^t \bignorm{\grad\mu^L(\tau)}_{L^2(\Omega)} \dtau \\
		&\quad \le \abs{t-s}^{\frac 12} \left(\int_0^t \bignorm{\grad\mu^L(\tau)}_{L^2(\Omega)}^2 \dtau\right)^{\!\frac 12} 
		\le \sqrt{2E(u_0,v_0)}\, \abs{t-s}^{\frac 12},
	\end{align*}
	where $H^{-1}(\Omega)$ stands for the dual space of $H^1_0(\Omega)$. 
	We now suppose that $s,t\ge 1$.
	Interpolating $H^{3/2}(\Omega)$ between $H^{-1}(\Omega)$ and $H^3(\Omega)$ (see, e.g., \cite[Sec.~4.3.1]{Triebel}), and invoking the smoothing property \eqref{SMOOTH:KLLM}, we obtain
	\begin{align*}
		\bignorm{u^L(t)-u^L(s)}_{H^{3/2}(\Omega)} 
		&\le C \bignorm{u^L(t)-u^L(s)}_{H^{-1}(\Omega)}^{\frac 38}\, \bignorm{u^L(t)-u^L(s)}_{H^{3}(\Omega)}^{\frac 58} \\
		&\le C \abs{t-s}^{\frac{3}{16}}.
	\end{align*}
	Eventually, by means of the continuous embedding $H^{3/2}(\Omega)\emb H^1(\Gamma)$, we conclude that
	\begin{align*}
		\bignorm{(u^L,v^L)(t) - (u^L,v^L)(s)}_{\HH^1} 
		\le C \abs{t-s}^{\frac{3}{16}},
	\end{align*}
	which completes the proof.
\end{proof}

As a consequence of Theorem~\ref{THM:WP:KLLM}, Proposition~\ref{PROP:IMP:KLLM} and Proposition~\ref{PROP:CD:KLLM}, the weak solutions to the problem \eqref{CH:INT.} can be described via a strongly continuous semigroup.

\begin{corollary}
	Suppose that \eqref{ass:dom}--\eqref{ass:pot} hold and that $L\in(0,\infty)$. Let $m\in\R$ be arbitrary. For any $(u_0,v_0)\in\Wm$ let $(u^L,v^L,\mu^L,\theta^L)$ denote the corresponding unique weak solution of the system \eqref{CH:INT.}.  
	
	For any $t\ge 0$, we define the operator
	\begin{align}
		S_m^L(t):\Wm\to\Wm,\quad (u_0,v_0) \mapsto \big(u^L(t),v^L(t)\big).
	\end{align}
	
	Then the family $\big\{S_m^L(t)\big\}_{t\ge 0}$ defines a strongly continuous semigroup on $\Wm$.
\end{corollary}

\subsection{The GMS model ($L=0$)}\label{sec:wellposedGMS}

\begin{theorem}[Strong well-posedness and smoothing property for the GMS model] \label{THM:WP:GMS} 
	Suppose that \eqref{ass:dom}--\eqref{ass:pot} hold and that $L=0$. Let $m\in\R$ be arbitrary and let 
	$(u_0,v_0) \in \Wm$ be any initial datum.
	
	Then there exists a unique global weak solution $(u^0,v^0,\mu^0,\theta^0)$ of the system \eqref{CH:INT.} existing on the time interval $\RP$ and having the following properties:
	\begin{enumerate}[label=$(\mathrm{\roman*})$, ref = $\mathrm{\roman*}$]
		\item For every $T>0$, the solution has the regularity
		\begin{subequations}
			\label{REG:GMS}
			\begin{align}
			\label{REG:GMS:1}
			&\begin{aligned}
			\big(u^0,v^0\big) 
			& \in C^{0,\frac{1}{4}}([0,T];\LL^2) \cap L^\infty(0,T;\VV) \\
			& \qquad \cap L^2(0,T;\HH^3) \cap H^1\big(0,T;(\Db)'\big), 
			\end{aligned}
			\\
			\label{REG:GMS:2}
			&\big(\mu^0,\theta^0\big) \in L^2(0,T;\HH^1).
			\end{align}
		\end{subequations}
		
		\item
		The solution satisfies
		\begin{subequations}
			\label{WF:GMS}
			\begin{align}
			\label{WF:GMS:1}
			\biginn{(\delt u^0,\delt v^0)}{(w,z)}_{\Db} 
			&= - \intO \grad\mu^0 \cdot \grad w \dx - \intG \gradg\theta^0 \cdot \gradg z \dG
			\end{align}
			a.e.~on $\RP$ for all test functions $w\in\Db$, and
			\begin{alignat}{2}
			\label{WF:GMS:3}
			&\mu^0 = -\Lap u^0 + F'(u^0) &&\quad\text{a.e.~in Q}, \\
			\label{WF:GMS:4}
			&\theta^0 = -\Lapg v^0 + G'(v^0) + \deln u^0 &&\quad\text{a.e.~in $\Sigma$}, \\
			\label{WF:GMS:5}
			&\mu^0\vert_{\Si} = \beta\theta^0, 
			\quad u^0\vert_{\Si} = v^0
			&&\quad\text{a.e.~in $\Sigma$}, \\
			\label{WF:GMS:6}
			&(u^0,v^0)\vert_{t=0} = (u_0,v_0) &&\quad \text{a.e.~in } \Omega \times \Gamma.
			\end{alignat}
		\end{subequations}
		
		\item The solution satisfies the mass conservation law
		\begin{align}
		\label{MASS:GMS}
		\beta \intO u^0(t) \dx + \intG v^0(t) \dG = \beta \intO u_0 \dx + \intG v_0 \dG = m,\quad
		\end{align}
		for all $t\in\RP$, i.e., $\big(u^0(t),v^0(t)\big)\in\Wm$ for almost all $t\in\RP$.
		
		\item The solution satisfies the energy inequality
		\begin{align}
		\label{ENERGY:GMS}
		E\big(u^0(t),v^0(t)\big)
		+ \frac 12 \int_0^t \bignorm{\big(\mu^0(s),\theta^0(s)\big)}_{0,\beta}^2
		\le E(u_0,v_0)
		\end{align}	
		for almost all $t\ge 0$. 
		
		\item The solution satisfies the smoothing property, i.e. there exists a constant $C_0>0$ depending only on $\Omega$, $\beta$, $F$, $G$ and  $\norm{(u_0,v_0)}_{\HH^1}$, such that
		\begin{align}
			\label{SMOOTH:GMS}
			\bignorm{ \big(u^0(t),v^0(t)\big) }_{\HH^3} 
			\leq C_0\left( \dfrac{1+t}{t}\right)^{\frac{1}{2}}
		\end{align}
		for almost all $t\ge 0$.
	\end{enumerate}
\end{theorem}

\begin{proof}
	For the existence and uniqueness of a weak solution satisfying the assertions (i)--(iv), we refer the reader to \cite[Prop.~3.4]{KLLM}. 
	
	As the proof of the smoothing property (v) is very similar to the approach in the proof of Theorem~\ref{THM:WP:KLLM}, we only sketch the most important steps. To provide a cleaner presentation we omit the superscript $0$. Let $(w,z)\in\Db$ be arbitrary (i.e., $w\vert_\Gamma = \beta z$ a.e.~on $\Gamma$). Testing \eqref{WF:GMS:1} with $(w,z)$ and recalling the definition of the product $\inn{\cdot}{\cdot}_{\Db}$ (see \eqref{pre:D}), we obtain
	\begin{align}
	\begin{aligned}
	\big\langle (\delt u,\delt v), (w,z) \big\rangle_{\Db} 
	&= - \big( \grad\mu, \grad w \big)_{L^2(\Omega)} 
		- \big( \gradg \theta , \gradg z \big)_{L^2(\Gamma)}\\
	&= -\big((\mu,\theta),(w,z) \big)_{0,\beta}. 
	\end{aligned}
	\end{align}
	Hence, using the notation for the difference quotient introduced in \eqref{NOT:FDQ}, we infer that 
	\begin{align}
		\label{sumweak:GMS} 
		\begin{aligned}
			\big\langle (\partial_{t}^{h}\delt u(\tau),\partial_{t}^{h}\delt v(\tau)), (w,z) \big\rangle_{\Db} 
			= -\Big(\big(\partial_{t}^{h}\mu(\tau),\partial_{t}^{h}\theta(\tau)\big),(w,z) \Big)_{0,\beta} 
		\end{aligned}
	\end{align}
	for almost all $\tau\in\RP$, where $\del_t^h \mu(\tau)$ and $\del_t^h \theta(\tau)$ can be expressed by means of
	$\eqref{WF:GMS:3}$ and $\eqref{WF:GMS:4}$.
	Choosing the test function $ (w,z) =\SS^0\big(\partial_{t }^{h}u(\tau),\partial_{t }^{h}v(\tau)\big)$, a straightforward computation reveals that
	\begin{align*}
	&-\dfrac{1}{2}\dfrac{\dd}{\dd \tau}\big\| \partial_{t }^{h}u(\tau)\big\|_{0,\beta,*}^{2} \\
	&\quad =\big\|\nabla\partial_{t }^{h}u(\tau) \big\|^{2} _{L^{2}(\Omega)}
	+\dfrac{1}{h}\int_{\Omega }
	\Big(F^{\prime}\big(u(\tau+h)\big)-F^{\prime}\big(u(\tau)\big)\Big) \partial_{t }^{h}u(\tau) \dx\\
	&\qquad	+\big\|\gradg\partial_{t }^{h}v(\tau) \big\|^{2}_{L^{2}(\Gamma)}
	+\dfrac{1}{h}\int_{\Gamma }
	\Big(G^{\prime}\big(v(\tau+h)\big)-G^{\prime}\big(v(\tau)\big)\Big) \partial_{t }^{h}v(\tau) \dG
	\end{align*}
	for almost all $\tau\in\RP$.
	Based on this identity, the smoothing property can be established by proceeding exactly as in the proof of Theorem~\ref{THM:WP:KLLM}. Note that the regularity estimates in \cite[Thm.~3.3]{knopf-liu} are also applicable for bulk-surface elliptic problems with Dirichlet type coupling conditions.	
\end{proof}

\medskip

As for the KLLM model, we obtain an improved continuity result, Lipschitz/Hölder continuous dependence of the weak solutions on the initial data, and a uniform Hölder estimate for times in $[1,\infty)$.

\begin{proposition}[Improved continuity results]\label{PROP:ICR:GMS}
	Suppose that \eqref{ass:dom}--\eqref{ass:pot} hold and that $L=0$. Let $m\in\R$ be arbitrary,
	let $(u_{0},v_{0}) \in \Wm $ be any initial datum, and let $(u^0,v^0,\mu^0,\theta^0)$ denote the corresponding weak solution of the system \eqref{CH:INT.}.
	Then it holds that%
	\begin{align}
	\label{CONT:UV}
	(u^0,v^0) \in C([0,\infty);\HH^1) \cap C((0,\infty);\HH^2).
	\end{align}
\end{proposition}

\medskip

\begin{proposition}[Continuous dependence on the initial data]\label{PROP:CD:GMS}
	Suppose that \eqref{ass:dom}--\eqref{ass:comp} hold and that $L=0$. Let $m\in\R$ be arbitrary.
	For $i\in\{1,2\}$, let $(u_{0,i},v_{0,i}) \in \Wm $ be any initial datum and let $(u_i^0,v_i^0,\mu_i^0,\theta_i^0)$ denote the corresponding weak solution of the system \eqref{CH:INT.}. 
	
	Then there exist positive, non-decreasing functions $\Lambda_k^0 \in C(\RP)$, $k\in\{-1,1,2\}$ depending only on $\Omega$, $\beta$, $F$ and $G$ such that the following statements hold:
	\begin{alignat}{2}
	\label{CD:GMS:1}
	\bignorm{\big(u^0_{2},v_2^0\big)(t) - \big(u^0_{1},v_1^0\big)(t)}_{0,\beta,*}
	&\le \Lambda_{-1}^0(t) \bignorm{(u_{0,2},v_{0,2})-(u_{0,1},v_{0,1})}_{0,\beta,*},
	&&\;\; t\ge 0, \\
	\label{CD:GMS:2}
	\bignorm{\big(u^0_{2},v_2^0\big)(t) - \big(u^0_{1},v_1^0\big)(t)}_{\HH^1} 
	&\le \Lambda_1^0(t) \bignorm{(u_{0,2},v_{0,2})-(u_{0,1},v_{0,1})}_{0,\beta,*},
	&&\;\; t\ge 1, \\
	\label{CD:GMS:3}
	\bignorm{\big(u^0_{2},v_2^0\big)(t) - \big(u^0_{1},v_1^0\big)(t)}_{\HH^2} 
	&\le \Lambda_2^0(t) \bignorm{(u_{0,2},v_{0,2})-(u_{0,1},v_{0,1})}_{0,\beta,*}^{\frac 12},
	&&\;\; t\ge 1.
	\end{alignat}
\end{proposition}

\medskip

\begin{proposition}[Uniform Hölder estimate]\label{PROP:HLD:GMS}
	Suppose that \eqref{ass:dom}--\eqref{ass:comp} hold and that $L=0$. Let $m\in\R$ be arbitrary, let $(u_{0},v_{0}) \in \Wm $ be any initial datum, and let $(u^0,v^0,\mu^0,\theta^0)$ denote the corresponding weak solution of the system \eqref{CH:INT.}. 
	
	Then there exists a positive constant $\Theta^0$ depending only on $\Omega$ and $E(u_0,v_0)$ such that 
	\begin{alignat}{2}
	\bignorm{(u^0,v^0)(t) - (u^0,v^0)(s)}_{\HH^1} 
	&\le \Theta^0\, \abs{s-t}^{\frac{3}{16}},
	&&\quad t\ge 1.
	\end{alignat}
\end{proposition}

The proofs of Proposition~\ref{PROP:ICR:GMS}, Proposition~\ref{PROP:CD:GMS} and Proposition~\ref{PROP:HLD:GMS} are very similar to those of the corresponding results for the KLLM model. Therefore, they will not be presented. However, we point out that the compatibility condition \eqref{ass:comp} is required in the proof of Proposition~\ref{PROP:CD:GMS} in order to proceed analogously as in the proof of Proposition~\ref{PROP:CD:KLLM}.

\medskip

As a consequence of Theorem~\ref{THM:WP:GMS}, Proposition~\ref{PROP:ICR:GMS} and Proposition~\ref{PROP:CD:GMS}, the weak solutions to the problem \eqref{CH:INT.} can be described via a strongly continuous semigroup.

\begin{corollary}
	Suppose that \eqref{ass:dom}--\eqref{ass:pot} hold and that $L=0$. Let $m\in\R$ be arbitrary. For any $(u_0,v_0)\in\Wm$ let $(u^0,v^0,\mu^0,\theta^0)$ denote the corresponding unique weak solution. 
	
	For any $t\ge 0$, we define the operator
	\begin{align}
	S_m^0:\Wm\to\Wm,\quad (u_0,v_0) \mapsto \big(u^0(t),v^0(t)\big).
	\end{align}
	
	Then the family $\big\{S_m^0(t)\big\}_{t\ge 0}$ defines a strongly continuous semigroup on $\Wm$.
\end{corollary}

\section{Long-time dynamics and existence of a global attractor}\label{SEC:LTD}
\subsection{Long-time dynamics of the KLLM model}\label{sec:long-time}
In this section, we investigate the long-time behavior of the KLLM model \eqref{CH:INT.} for any fixed $L\in(0,\infty)$. We prove the existence of a global attractor for weak solutions of this system, and we establish the convergence of each weak solution to a single stationary point as $ t\rightarrow \infty$.
To this end, we first analyze the set of stationary points of the problem \eqref{CH:INT.} with $L\in(0,\infty)$.

\subsubsection{The set of stationary points} \label{sec:stat}
For the semigroup $ {\left\lbrace S_{m}^{L}(t)\right\rbrace }_{t\geq0} $ on $ \Wm $ corresponding to the KLLM model with fixed $L\in(0,\infty)$, the corresponding set of stationary points of the problem \eqref{CH:INT.} can be expressed as follows:
\begin{align}
\mathcal{N}_{m}^L=\left\lbrace  (u,v) \in \Wm:S_{m}^{L}(t) (u,v) = (u,v) \;\text{for all}\; t\geq 0\right\rbrace.\nonumber  
\end{align}
We aim to show that the set $ \mathcal{N}_{m}^L $ is a nonempty and bounded subset of $ \Wmt $. 

To this end, let $ (u,v) \in \mathcal{N}_{m}^L$ be arbitrary, and let $\mu$ and $\theta$ denote the corresponding chemical potentials given by \eqref{WF:KLLM:3} and \eqref{WF:KLLM:4}.
Then the energy inequality \eqref{ENERGY:KLLM} entails that
\begin{equation*}
\int_{0}^{T}\left( \left\Vert \nabla \mu \right\Vert _{L^{2}\left(
	\Omega \right) }^{2}+\left\Vert \nabla \theta \right\Vert _{L^{2}\left(
	\Gamma \right) }^{2}+\dfrac{1}{L}\left\| \beta\theta-\mu\right\|_{L^{2}\left(
	\Gamma \right) }^{2} \right) dt \leq0. 
\end{equation*}%
Hence, there exists a constant $\lambda\in\R$ such that $\mu=\beta\lambda$ a.e.~in $\Omega$ and $\theta=\lambda$ a.e.~on $\Gamma$.
Integrating \eqref{WF:KLLM:3} over $\Omega$ and \eqref{WF:KLLM:4} over $\Gamma$, we infer that 
\begin{equation}
\lambda
=\dfrac{\int_{\Omega }F ^{\prime }\left( u \right)\dx + \int_{\Gamma }G ^{\prime }\left( v \right) \dG}{\beta\left\vert \Omega\right\vert+\left\vert\Gamma\right\vert} 
\text{.}\label{5.33}
\end{equation}%
Therefore, we obtain that $\mathcal{N}_{m}^L$ consists of solutions to the following stationary problem.%
\begin{subequations}
	\label{5.34}
\begin{alignat}{2}
	\label{5.34:1}
-\Delta u +F ^{\prime }\left( u \right)  &=\beta\lambda 
	&&\quad\text{ in } \Omega, \\ 
	\label{5.34:2}
-\Delta_{\Gamma}v+G ^{\prime }\left( v \right)+\partial _{n }u &=\lambda 
	&&\quad\text{ on } \Gamma,\\
	\label{5.34:3} 
u|_{\Gamma} &=v 
&&\quad \text{ on } \Gamma,\\
	\label{5.34:4}
\textstyle \beta\int_\Omega u \dx + \int_\Gamma v  \dG &= m.
\end{alignat}
\end{subequations}
Hence, in order to prove that $\mathcal{N}_{m}^L$ is nonempty, it suffices to show that \eqref{5.34} has at least one solution.
Since the problem \eqref{5.34} does not depend on $ L $, we notice that the set of stationary points $\mathcal{N}_{m}^L$ is actually independent of $ L $. 
Thus, in the remainder of this section, we will just use the notation  $\mathcal{N}_{m} := \mathcal{N}_{m}^{L}$. 

In the following, we interpret the total free energy as a continuous functional
\begin{align}
\label{DEF:ENF}
	E: \VV^1 \to [0,\infty), \quad (u,v) \mapsto E(u,v)
\end{align}
where for any $(u,v)\in\Wm$, the expression $E(u,v)$ is as defined in \eqref{DEF:EN}.

We first prove that the energy functional $E$ has a global minimizer on the space $\Wm$. Next, we show that a minimizer of $E$ on $\Wm$ is actually a strong solution of the system $(\eqref{5.33},\eqref{5.34})$ meaning that it is a stationary point.	

\begin{lemma}\label{LM:minimizer}
	Suppose that the assumptions \eqref{ass:dom}-\eqref{ass:pot} hold. Then the functional $E\vert_{\Wm}$ has at least one minimizer $ (u_*,v_*) \in  \Wm $, i.e.,
	\begin{equation*}
	E(u_*,v_*)=\inf_{ (u,v)\in \Wm} E\left(u,v\right).
	\end{equation*}
\end{lemma}

\begin{proof}
	To prove the assertion, we employ the direct method of calculus of variations.
	From the assumptions \eqref{ass:pot} on $ F $ and $ G $, we obtain the lower bound 
	\begin{equation}
	E\left( u,v \right) \geq \frac{1}{2}\left\Vert \nabla u
	\right\Vert _{L^{2}\left( \Omega \right) }^{2}+\frac{1}{2}\left\Vert\gradg v
	\right\Vert _{L^{2}\left( \Gamma \right) }^{2}+ a_{F}\left\| u\right\| _{L_{p}\left( \Omega \right)}^{p}+a_{G}\left\| v\right\| _{L_{q}\left( \Gamma \right)}^{q}-b_{F}-b_{G} \label{5.35}
	\end{equation}
	for all $(u,v) \in \Wm$.
	Hence, the infimum 
	$${E}_{*}:=\underset{(u,v) \in \Wm}{\inf }E\left( u,v \right) $$ 
	exists and thus, we can find a minimizing sequence $\left\{(u _{k},v _{k}) \right\}_{k\in\N}\subset $ $\Wm$
	such that%
	\begin{equation}\label{5.36}
	\lim_{k\rightarrow \infty }E\left( u _{k},v _{k}\right)={E}_{*} \qquad\text{and}\qquad E\left( u_{k},v _{k}\right) \leq {E}_{*}+1 \quad\text{for all $k\in\N$.} 
	\end{equation}
	Combining \eqref{5.35} and \eqref{5.36}, it follows that $\left\{ \left( u	_{k},v _{k}\right) \right\}_{k\in\N}$ is bounded in ${\cal H}^{1} $. Then, according to the Banach-Alaoglu theorem,
	it holds that
	\begin{equation}
 	(u_k, v_k)\wto ( {u }^{*},{u }^{*}) \quad \text{in $\HH^1$} \label{5.37}
	\end{equation}
	along a non-relabeled subsequence.
	Invoking Sobolev's embedding theorem we further deduce that
	\begin{align}
	\left(  u _{k}, v _{k}\right)\rightarrow \left( {u }^{*},{v }^{*}\right) 
	\quad\text{in }
	L^{p}\left( \Omega \right) \times L^{q}\left( \Gamma \right), \text{ and a.e.~in } \Omega \times \Gamma
	\label{5.39}
	\end{align}%
	after another subsequence extraction, where $p$ and $q$ are the exponents introduced in \eqref{ass:pot}.
	Since  $F $ and $ G $ are continuous, \eqref{5.39} implies%
	\begin{equation*}
	F \left( u _{k}\right) \rightarrow F \left( {u }^{*}\right) \quad\text{a.e.~in }\Omega,
	\quad\text{and}\quad
	G \left( v _{k}\right) \rightarrow G \left( {v }^{*}\right) \quad\text{a.e.~in }\Gamma.
	\end{equation*}%
	Recalling assumption \eqref{ass:pot}, it follows from Lebesgue's generalized convergence theorem \cite[Sect.\,3.25]{Alt} that%
	\begin{equation*}
	\lim_{k\rightarrow \infty } \int_{\Omega }F \left( u
	_{k}\right)\dx =\int_{\Omega }F \left( {u }^{*}
	\right)\dx ,
	\quad\text{and}\quad
	\lim_{k\rightarrow \infty } \int_{\Gamma}G \left( v
	_{k}\right)\dG =\int_{\Gamma }G \left( {v }^{*}
	\right)\dG .
	\end{equation*}%
	As all the other terms in the energy functional are convex, they are lower semicontinuous, and we thus conclude that
	\begin{equation*}
	E\left( u^{*},v ^{*} \right) \leq 
	\underset{m\rightarrow \infty }{\lim \inf }\; E\left( u
	_{k},v
	_{k}\right) ={E}_{*}.
	\end{equation*}%
	This proves that $ (u^{*},v ^{*}) $ is
	a minimizer of the functional $E$ on $\Wm$.
\end{proof}

\medskip\pagebreak[2]

We further show that the energy functional is of class $C^2$.
\begin{lemma}[Regularity of the energy functional]
	\label{LEM:ENC2}
	The energy functional $ E $ is twice continuously Fr\'echet differentiable. For every $ (u,v), (\zeta,\eta) \in \VV^1 $, we have
	\begin{align}\label{frstFrec.der}
	\left\langle E^{\prime}(u,v) ,(\zeta,\eta)\right\rangle_{\VV^1}
	&=\intO \nabla u \cdot \nabla \zeta + F^{\prime}(u) \zeta\dx 
		+ \intG \gradg v \cdot \gradg \eta + G^{\prime}(v) \eta \dG.
	\end{align}
	For every $ (u,v), (\zeta_{1},\eta_{1}), (\zeta_{2},\eta_{2}) \in \VV^1 $, we have
	\begin{align}\label{scndFrec.der}
	\begin{aligned}
	\left\langle E^{\prime \prime}(u,v)(\zeta_{1},\eta_{1}) ,(\zeta_{2},\eta_{2}) \right\rangle_{\VV^1}
	& =\intO \nabla \zeta_{1} \cdot \nabla \zeta_{2} + F^{\prime \prime}(u) \zeta_{1} \zeta_{2}\dx \\
	&\quad + \intG \gradg \eta_{1} \cdot \gradg \eta_{2} + G^{\prime \prime}(v) \eta_{1} \eta_{2}\dG.
	\end{aligned}
	\end{align}
\end{lemma} 

The proof of this lemma is straightforward and will thus not be presented. For a very similar result including a detailed proof, we refer to \cite[Lem.~6.2]{CFP}.

\begin{lemma}[Critical points of $E$ are steady states]
	\label{LM:crtclpnt}
	Suppose that the assumptions \eqref{ass:dom}-\eqref{ass:pot} are satisfied and suppose that $( u_*,v_*)\in\Wm $. Then, the following statements are equivalent:
	\begin{enumerate}[label=$(\mathrm{\roman*})$, ref = $\mathrm{\roman*}$]
		\item $(u_*,v_*)$ is a critical point of $E\vert_{\Wm}$.
		\item $(u_*,v_*) \in \Wmt$ and $(u_*,v_*)$ is a strong solution of the problem $(\eqref{5.33},\eqref{5.34})$. In particular, this means that $(u_*,v_*)\in\mathcal N_m$.
	\end{enumerate}
	 
\end{lemma}
\begin{proof}
	We first notice that (i) is equivalent to 
	\begin{align}
	\begin{aligned}
	&\left\langle E^{\prime}(u_*,v_*) ,(\zeta,\xi)\right\rangle_{\VV^1} \\
	&\quad= \int_{\Omega }\left( \nabla u_* \cdot \nabla \zeta 
		+F^{\prime }\left( u_*\right) \zeta  \right)\dx
	+\int_{\Gamma }\left( \gradg v_* \cdot \gradg \xi 
		+G^{\prime }\left( v_*\right) \xi  \right)\dG=0 \label{5.41}
	\end{aligned}
	\end{align}
	for all $ (\zeta,\xi) \in \Wo $.
	Now, for any $ (\overline{\zeta},\overline{\xi})\in \VV^1 $, we define 
	$$\gamma=\dfrac{\beta\int_\Omega  \overline{\zeta} \dx+\int_\Gamma \overline{\xi}  \dG}
		{\beta\abs{\Omega} + \abs{\Gamma} } .$$ 
	Hence, $ (\zeta,\xi)=(\overline{\zeta}-\gamma,\overline{\xi}-\gamma) \in \Wo\, $, and thus \eqref{5.41} is equivalent to
	\begin{align}
	\begin{aligned}
	&\int_{\Omega }\left( \nabla u_*\cdot\nabla \overline{\zeta} 
		+ F^{\prime }\left( u_*\right) \overline{\zeta}  \right)\dx
		+\int_{\Gamma }\left( \gradg v_* \cdot \gradg \overline{\xi} 
		+G^{\prime }\left( v_*\right) \overline{\xi}  \right)\dG \\
	&\qquad =\int_{\Omega }F^{\prime }\left( u_*\right)\gamma\dx
		+\int_{\Gamma }G^{\prime }\left( v_*\right)\gamma \dG 
	=\int_{\Omega } \beta\lambda\, \overline{\zeta} \dx
		+\int_{\Gamma } \lambda\, \overline{\xi} \dG  
	\end{aligned}
	\label{5.42}
	\end{align}
	for all $ (\overline{\zeta},\overline{\xi})\in \VV^1 $.
	This means that $(u_*,v_*) \in \Wm$ is
	a weak solution of the problem (\eqref{5.33},\eqref{5.34}). Moreover, recalling the growth assumption \eqref{ass:pot}, we can use elliptic regularity theory for problems with bulk-surface coupling (see \cite[Thm.~3.3]{knopf-liu}) to conclude that $(u_*,v_*)\in \mathcal{H}^{2} $. Therefore, $(u_*,v_*)$ is actually a strong solution of the problem (\eqref{5.33},\eqref{5.34}) which directly yields $(u_*,v_*)\in \mathcal N_m$. We conclude that (i) and (ii) are equivalent and thus, the proof is complete.
\end{proof}

\begin{lemma}\label{LM.boundedstat}
	Suppose that the assumptions \eqref{ass:dom}-\eqref{ass:pot} hold. Then, the set of stationary points $ \mathcal{N}_{m} $ is a nonempty and bounded subset of $\Wmt$.
\end{lemma}

\begin{proof}
	Let $C$ denote a generic constant depending only on $\Omega$, $F$, $G$, $\beta$ and $m$ which may change its value from line to line.
	We first notice that Lemma~\ref{LM:minimizer} and Lemma~\ref{LM:crtclpnt} directly imply that $ \mathcal{N}_{m} $ is nonempty. Let now $(u,v)\in \mathcal{N}_{m}$ be arbitrary and let $\lambda \in\R$ as defined in \eqref{5.33}. Testing \eqref{5.34:1} with $u $ and $ \eqref{5.34:2} $ with $v $, and summing the obtained identities we get
	\begin{align}
		\label{5.47}
		\begin{aligned}
		&\left\| \nabla u\right\|_{L^{2}\left( \Omega\right) }^{2} 
			+\left\| \gradg v\right\|_{L^{2}\left( \Gamma\right) }^{2}
			+ \int_{\Omega }F_1^{\prime }\left( u \right)u\dx
			+ \int_{\Gamma }G_1^{\prime }\left( v \right)v \dG \\
		&\qquad =m\lambda
			- \int_{\Omega }F_2^{\prime }\left( u \right)u\dx
			- \int_{\Gamma }G_2^{\prime }\left( v \right)v \dG.
		\end{aligned}
	\end{align}
	Performing a Taylor expansion, recalling that $F_1''$ and $G_1''$ are non-negative, and invoking the growth conditions \eqref{GR:F} and \eqref{GR:G} in \eqref{ass:pot}, we deduce that
	\begin{alignat}{2}
		\label{EST:F1P}
		F_1'(s)\, s &\ge F_1(s) - F_1(0) &&\ge a_{F'}\abs{s}^p - C,\\
		\label{EST:G1P}
		G_1'(s)\, s &\ge G_1(s) - G_1(0) &&\ge a_{G'}\abs{s}^q - C
	\end{alignat}
	for all $s\in\R$.
	Moreover, recalling \eqref{5.33} and the growth assumptions in \eqref{ass:pot}, by means of Young's inequality we infer that for any $\alpha>0$ there exists a constant $C_\alpha$ depending only on $\Omega$, $F$, $G$, $\beta$, $m$ and $\alpha$ such that
	\begin{align}
	\label{EST:ML}
		\abs{m\lambda}
			+ \int_{\Omega } \abs{F_2^{\prime}(u)} \abs{u} \dx
			+ \int_{\Gamma } \abs{G_2^{\prime}(v)} \abs{v} \dG
		\leq C_\alpha + \alpha \big(\left \|u\right\|_{L^{p}\left( \Omega\right) }^{p} 
			+\left\|v\right\|_{L^{q}\left( \Gamma\right) }^{q}\big) .
	\end{align}
	Fixing $\alpha = \frac 1 2 \min\{a_{F'},a_{G'}\}$ and using the estimates \eqref{EST:F1P}--\eqref{EST:ML}, we conclude from \eqref{5.47} that
	\begin{align}
	\label{5.48}
		\begin{aligned}
		& \left\| \nabla u\right\|_{L^{2}\left( \Omega\right) }^{2} 
		+ \left\| \gradg v\right\|_{L^{2}\left( \Gamma\right) }^{2}
		+ \frac{a_{F'}}{2}  \norm{u}_{L^p(\Omega)}^p
		+ \frac{a_{G'}}{2} \norm{v}_{L^q(\Gamma)}^q \le C .
		\end{aligned}
	\end{align}
	This eventually yields $\left\| ( u,v)\right\|_{\HH^1}\leq C$. From the growth conditions in \eqref{ass:pot} we infer $\norm{F'(u)}_{L^2(\Omega)} \le C$ and $\norm{G'(v)}_{L^2(\Gamma)} \le C$. Hence, we can apply regularity theory for elliptic problems with bulk-surface coupling (see \cite[Thm.~3.3]{knopf-liu}) to conclude that 
	$\norm{(u,v)}_{\HH^2} \le C$,
	which completes the proof.
\end{proof}

\subsubsection{Existence of a global attractor}\label{subsec:att}
We first prove the following asymptotic compactness property by exploiting the smoothing estimate \eqref{SMOOTH:KLLM}.

\begin{lemma}[Asymptotic compactness]\label{lm.asympcompact}
	Suppose that the assumptions \eqref{ass:dom}-\eqref{ass:pot} are satisfied, and let $B$ be a bounded subset
	of $ \Wm$. 
	Then for any sequences $\left\{ (u_k,v_k) \right\}_{k\in\N}\subset B$ and $\{t_k\}_{k\in\N} \subset \RP$ with $t_k\to\infty$ as $k\to\infty$,
	the sequence 
	$\left\{ S_{m}^{L}\left( t_k\right) (u_k,v_k) \right\}_{k\in\N}$ 
	has a strongly convergent subsequence in $ \Wmt$.
\end{lemma}

\begin{proof}
	We first obtain from the smoothing property \eqref{SMOOTH:KLLM} that for all $k\in\N$ with $t_k\geq 1$,
	\begin{align*}
	\left\| S_{m}^{L}\left( t_k\right)(u_k,v_k) \right\| _{\HH^3}\leq C_* \left(\dfrac{t_k+1}{t_k} \right)^{1/2}\leq \sqrt{2}\, C_* . 
	\end{align*}
	Recall that the constant $ C_* $ from \eqref{SMOOTH:KLLM} does not depend on $ t_k $. Hence, there exists $(u_*,v_*) \in \HH^3$ such that $S_{m}^{L}\left( t_k\right)(u_k,v_k)$ converges to $(u_*,v_*)$ weakly in $\HH^3$ as $k\to\infty$, up to extraction of a subsequence. As $\HH^3$ is compactly embedded in $\HH^2$, we further deduce that $S_{m}^{L}\left( t_k\right)(u_k,v_k)$ converges to $(u_*,v_*)$ strongly in $\HH^2$. Moreover, since $S_{m}^{L}\left( t_k\right)(u_k,v_k) \in \Wmt$ for all $k\in\N$, we eventually conclude that $(u_*,v_*)\in \Wmt$ which proves the claim.
\end{proof}


\begin{lemma}[Gradient system]
	\label{lm:gradient}
	Suppose that \eqref{ass:dom}-\eqref{ass:pot} hold.
	Then the energy functional $E$ that was defined in \eqref{DEF:EN} is a strict Lyapunov function for the dynamical system $(\Wm, \{S_{m}^{L}\left(t\right)\}_{t\ge 0})$. This means that the dynamical system is a so-called \emph{gradient system}.
\end{lemma}
\begin{proof}
	Let $(u^L,v^L,\mu^L,\theta^L)$ denote an arbitrary weak solution of the system \eqref{CH:INT.} on $\RP$ to the initial datum $(u_0,v_0)\in\Wm$ in the sense of Theorem~\ref{THM:WP:KLLM}. In particular, this means that $\big(u^L(t),v^L(t)\big) = S_m^L(t)(u_0,v_0)$ for all $t\ge 0$.
	Interpreting $(u^L,v^L,\mu^L,\theta^L)$ as a weak solution starting at any time $s\ge 0$ (instead of zero), the energy inequality \eqref{ENERGY:KLLM}
	yields that for almost all $t\in\RP$ with $t\ge s$, 
	\begin{align}
	\label{5.29}
	\begin{aligned}
	&E\big( u^{L} (t),v^{L} (t) \big) \\
	& \;\; + \frac{1}{2} \int_{s}^{t} \bignorm{\nabla \mu^{L}(\tau)}_{L^2(\Omega)}^{2}
		+ \bignorm{\gradg \theta^{L}(\tau)}_{L^2(\Gamma)}^{2}	 
		+\frac{1}{L} \bignorm{ \beta\theta^{L}(\tau)-\mu^{L}(\tau) }_{L^2(\Gamma)}^{2} \dtau \\
	&\leq E\left( u^{L}(s),v^{L}(s) \right).
	\end{aligned}  
	\end{align}%
	This implies that $E$ 
	is nonincreasing along weak solutions of the system \eqref{CH:INT.}. 	
	Hence, the energy functional $E$ is a Lyapunov function for the dynamical system $\{S_{m}^{L}\left(t\right)\}_{t\ge 0}$. 
	
	To prove that the Lyapunov function $E$ is actually strict, we assume that there exists a weak solution $(u^L,v^L,\mu^L,\theta^L)$ as well as $s,t\in\RP$ with $s<t$ such that $(u^L,v^L) \notin \mathcal N_m$ and 
	\begin{equation*}
	E \left(u ^{L}(t),v^{L}(t) \right) 
	= E \left(u ^{L}(s),v^{L}(s) \right) .
	\end{equation*}%
	Since, $E$ is nonincreasing along weak solutions, this directly implies that $E(u^L(\cdot),v^L(\cdot))$ is constant on the time interval $[s,t]$.
	Thus, the inequality \eqref{5.29} entails that
	\begin{subequations}
	\label{5.30}
	\begin{alignat}{2}
		 \label{5.30:1}
	\nabla \mu ^{L} &=0 &&\quad\text{a.e.~in }\Omega\times[s,t] , \\ 
		 \label{5.30:2}
	\gradg \theta^{L} &=0 &&\quad\text{a.e.~on }\Gamma\times[s,t] , \\ 
		 \label{5.30:3}
	\beta\theta^{L} &=\mu^{L} &&\quad\text{a.e.~on }\Gamma\times[s,t] .
	\end{alignat}
	\end{subequations}
	Plugging these relations into the weak formulations \eqref{WF:KLLM:1} and \eqref{WF:KLLM:2}, we infer that $u^L$ and $v^L$ are constant in time on $[s,t]$. It immediately follows that $\mu^L$ and $\theta^L$ are also constant on $[s,t]$, and from the uniqueness of weak solutions, we eventually conclude that $(u^L,v^L,\mu^L,\theta^L)$ is constant on the whole time interval $\RP$. This means that $(u^L,v^L)$ solves \eqref{5.34} and thus, $(u^L,v^L) \in \mathcal N_m$, which contradicts the assumptions.
	
	We have thus proven that the energy functional $E$ is strictly decreasing along non-stationary weak solutions of the system \eqref{CH:INT.} and thus,
	$E$ is a strict Lyapunov function for the dynamical system $(\Wm , \{S_{m}^{L}\left(t\right)\}_{t\ge 0} ) $.
\end{proof}

To prove the existence of a global attractor which is bounded in $\WW_{\beta,m}^3$, we first recall the following definitions.

\begin{definition} [Global attractor, cf. {\cite[Def.~7.2.1.]{chueshov}}] 
	\label{def:att}
	Let $\left( X,d\right)$ be a metric space and let $\left\{ S(t) \right\} _{t\geq 0}$ be a semigroup on $X$. 
	A set $\mathcal{A}\subset X$ is called a
	global attractor for the dynamical system $(X,\{S(t)\}
	_{t\geq 0})$, if
	\begin{enumerate}[label=$(\mathrm{\roman*})$, ref = $\mathrm{\roman*}$]
	\item $\mathcal{A}$ is a compact subset of $X$,
	\item $\mathcal{A}$ is invariant, i.e. $S(t) \mathcal{A}=%
	\mathcal{A}$ for all $t\ge 0$,
	\item $\underset{t\rightarrow \infty }{\lim }\dist_{X}(S(t) B,\mathcal{A}) =0$ for every bounded set $B\subset X$.
	\end{enumerate}
\end{definition}

In particular, it follows directly from this definition that a global attractor is always unique.

\medskip

\begin{definition}[Unstable manifold, cf. {\cite[Def.~7.5.1.]{chueshov}}]
	\label{DEF:UM}
	Let $\mathcal{N}$ be the set of stationary points of the dynamical system $ \left(X, \{S(t)\}_{t\ge 0}  \right)  $.
	We define the unstable manifold $\mathcal{M}_\mathrm{u}\left( \mathcal{N}\right) $
	emanating from the set $\mathcal{N}$ as the set of all $y\in X$ such that
	there exists a full trajectory $\gamma =\{u(t):t\in \R\}$ where $u$ satisfies the properties%
	\[
	u(0)=y\text{ and }\lim_{t\rightarrow -\infty }\dist_{X}(u(t),\mathcal{N})=0.
	\]
\end{definition}

\medskip

From the results established above, we now deduce the existence of a unique global attractor.

\begin{theorem}[Existence of a unique global attractor]
	\label{thm.glbatt}
	Suppose that the assumptions \eqref{ass:dom}-\eqref{ass:pot} are satisfied. Then, the
	semigroup $\left\{ S_{m}^{L}(t) \right\} _{t\geq 0}$ generated by
	the weak solutions of the problem \eqref{CH:INT.} possesses a unique global attractor ${\mathcal{A}}_{m}^{L} \subset \Wm$, and it holds that 
	\begin{align}
	\label{EQ:UNST}
	\mathcal{A}_{m}^{L}=\mathcal{M}^{L}_\mathrm{u}\left( \mathcal{N}_{m}\right). 
	\end{align}
	Moreover, the global attractor ${\mathcal{A}}_{m}^{L}$ is a bounded subset of $\WW_{\beta,m}^3$.
\end{theorem}

\begin{proof}
	Due to Lemma~\ref{LM.boundedstat}, Lemma~\ref{lm.asympcompact} and Lemma~\ref{lm:gradient}, the existence of a unique global attractor ${\mathcal{A}}_{m}^{L} \subset \Wm$ to the dynamical system $(\Wm,S_{m}^{L}\left(t\right)) $ which can be expressed as $\mathcal{A}_{m}^{L}=\mathcal{M}^{L}_\mathrm{u}\left( \mathcal{N}_{m}\right) $ follows directly from \cite[Cor.~7.5.7]{chueshov}. 
	The fact that $\mathcal{A}_{m}^{L}$ is bounded in $\HH^3$ is a direct consequence of the invariance property in Definition~\ref{def:att}(ii) and the smoothing property \eqref{SMOOTH:KLLM}.
\end{proof}

\subsubsection{Convergence to stationary points}\label{subsec:convstatpoints}

The invariance property of the global attractor ensures that for any $ (u_0,v_0) \in \mathcal A_m^L $, the corresponding weak solution $(u^L,v^L,\mu^L,\theta^L)$ exists on $\R$ and it holds that
\begin{align*}
(u_0,v_0) \in \left\{ (u^{L}(t),v^{L}(t)) \,\big\vert\, t\in\R \right\} \subset \mathcal A_m^L.
\end{align*}
A detailed proof can be found, e.g., in \cite[Lem.~6.1]{chueshov-2}.
In particular, due to Proposition~\ref{PROP:CD:KLLM}(b), this weak solution satisfies $(u^L,v^L)\in C(\R;\HH^2)$.
This allows the following definition:

\begin{definition}[$\omega$-limit set and $\alpha$-limit set]
	\label{DEF:LIMSET}
	For any initial datum $ (u_{0},v_{0}) \in \Wm\,$, let $(u^L,v^L,\mu^L,\theta^L)$ denote the corresponding weak solution of the system \eqref{CH:INT.}.
	We define the set
	\begin{align*}
	&\omega^{L} \left( u_{0} ,v_0\right) 
	:= \left\{ 
	(u_*,v_*) \in \mathcal{A}_{m}^{L} 
	\; \middle| \;
	\begin{aligned}
	& \text{$\exists \{t_k\}_{k\in\N} \subset \RP$ with $t_k\to \infty$ such that} \\
	& \text{$S^{L}_{m}\left( t_{k}\right)(u_{0},v_0) \to (u_*,v_*)$ in $\Wmt$}\\
	\end{aligned}
	\right\},
	\end{align*}
	which is called the \emph{$\omega $-limit set} of $(u_{0},v_0)$.	
	If $ (u_{0},v_{0}) \in \mathcal A^L_m$, we further define 
	\begin{align*}
	&\alpha^{L} \left( u_{0} ,v_0\right) 
	:= \left\{ 
	(u_*,v_*) \in \mathcal{A}_{m}^{L} 
	\; \middle| \;
	\begin{aligned}
	& \text{$\exists \{t_k\}_{k\in\N} \subset \R$ with $t_k\to -\infty$ such that} \\
	& \text{$\big(u^L(t_k),v^L(t_k)\big) \to (u_*,v_*)$ in $\Wmt$}\\
	\end{aligned}
	\right\},
	\end{align*}
	which is referred to as the \emph{$\alpha $-limit set} of $(u_{0},v_0)$.
\end{definition}

\pagebreak[2]

For our purposes, the following properties will be essential.

\begin{lemma}[Properties of the limit sets]
	\label{LEM:LIMSET}
	The limit sets have the following properties:
	\begin{enumerate}[label=$(\mathrm{\alph*})$, ref = $\mathrm{\alph*}$]
		\item For any $ (u_{0},v_{0}) \in \Wm\,$, the set $\omega^{L} \left( u_{0} ,v_0\right)$ is nonempty, compact in $\Wmt$, and invariant in the sense of Definition~\ref{def:att}(ii).
		Moreover, it holds that
		\begin{align}
		\label{LIM:OM}
		\lim_{t \to \infty} \dist_{\HH^2} \big(S^{L}_{m}(t)(u_{0},v_{0}),\omega^L(u_0,v_0)\big)
		=0, \quad
		\omega^{L} (u_{0},v_0) \subset \mathcal N_m \subset \Wmt.
		\end{align}
		\item For any $ (u_{0},v_{0}) \in \mathcal A^L_m\,$, the set $\alpha^{L} \left( u_{0} ,v_0\right)$ is nonempty, compact in $\Wmt$, and invariant in the sense of Definition~\ref{def:att}(ii).
		Moreover, it holds that
		\begin{align}
		\label{LIM:AL}
		\lim_{t \to -\infty} \dist_{\HH^2} \big(S^{L}_{m}(t)(u_{0},v_{0}),\alpha^L(u_0,v_0)\big)
		=0,
		\quad
		\alpha^{L} (u_{0},v_0) \subset \mathcal N_m \subset \Wmt.
		\end{align}
	\end{enumerate}
\end{lemma}

\begin{proof}
The assertions in (a) follow directly from \cite[Thm.~4.3.3]{henry} and \cite[Thm.~9.2.7]{cazenave} thanks to the asymptotic compactness result established in Lemma~\ref{lm.asympcompact}.

It is well-known that the $\alpha$-limit set can be expressed as
\begin{align*}
	\alpha^L(u_0,v_0)
	= \bigcap_{n\in\N} \overline{\phantom{\Big|}\big\{ \big(u^L(t),v^L(t)\big) \;\big\vert\; t \le -n \big\}}^{\HH^2},
\end{align*}
see, e.g., \cite[Eq.~(7.1.4)]{chueshov}. As the intersection of closed sets is also closed, we infer that $\alpha^L(u_0,v_0)$ is a closed subset of $\mathcal A^L_m$. Hence, since $\mathcal A^L_m$ is compact in $\Wmt$ so is $\alpha^L(u_0,v_0)$. Recalling Definition~\ref{DEF:UM}, the assertion \eqref{LIM:AL} follows directly from Theorem~\ref{thm.glbatt} since $\mathcal{A}_{m}^{L}=\mathcal{M}^{L}_\mathrm{u}\left( \mathcal{N}_{m}\right) $.
This proves (b) and thus, the proof is complete.
\end{proof}

We will now show that for every $ (u_{0},v_{0}) \in \Wm\,$, the corresponding state $(u^L(t),v^L(t))=S_m^L(t)(u_0,v_0)$ converges to one single stationary point $(u_\infty,v_\infty) \in \mathcal N_m$ as $ t \rightarrow \infty $. In addtion, we will prove that for every $ (u_{0},v_{0}) \in \mathcal A_m^L\,$, the pair $(u^L(t),v^L(t))$ converges to a single stationary point $(u_{-\infty},v_{-\infty}) \in \mathcal N_m$ as $ t \rightarrow -\infty $.

Arguing as in \cite[Prop.~6.6]{CFP} we next establish a \L ojasiewicz--Simon type inequality.
Here, we restrict ourselves to dealing only with the case $m=0$. To see that this is not really a restriction, we
introduce the constant $M:= m/(\beta\abs{\Omega}+\abs{\Gamma})$. If now $ (u^L,v^L,\mu^L,\theta^L)$ is a weak solution of the system \eqref{CH:INT.} with $(u^L(t),v^L(t)) \in \Wm$ for all $t\ge 0$, then the shifted quadruplet
$$ (u^L_M,v^L_M,\mu^L_M,\theta^L_M) :=  (u^L - M,v^L - M,\mu^L,\theta^L)$$
is in fact a weak solution of the problem \eqref{CH:INT.} written for the modified potentials $F_M := F(\,\cdot\, + M)$ and $G_M := G(\,\cdot\, + M)$ instead of $F$ and $G$. By construction, this solution satisfies $(u_M^L(t),v_M^L(t)) \in \WW_{\beta,0}^1$ for all $t\ge 0$. 
Therefore, the choice $m=0$ does not mean any loss of generality.

\begin{lemma}[\L ojasiewicz--Simon inequality]
	\label{LM.lojasiecwicz}
	Suppose that \eqref{ass:ana} holds, and that $ (u_*,v_*) \in \Wot$ is a critical point of the functional $E\vert_{\Wo} $. Then, there exist constants $ \gamma \in \left( 0, \frac{1}{2}\right]  $ and $ C, \sigma>0 $ depending only on $L$, $\beta$, $\Omega$ and $(u_*,v_*)$ such that for all $ (u,v) \in \Wot$ with $ \norm{ (u,v)- (u_{*},v_{*})}_{\Wo}<\sigma $, the \L ojasiewicz--Simon inequality
	\begin{align} \label{lojasiecwicz}
	\abs{E(u,v)-E(u_{*},v_{*})}^{1-\gamma}\leq C\norm{E^{\prime}(u,v)}_{(\Wo)'}
	\end{align}	
	is satisfied.	
\end{lemma}

\textbf{Comment.} According to Lemma~\ref{LM:crtclpnt}, any critical point $ (u_*,v_*) \in \Wo$ of the energy $E$ is actually a stationary point, i.e., the regularity $(u_*,v_*) \in \mathcal N_0 \subset \Wot$ holds automatically. 

\begin{proof}  
To prove the lemma, we will show that the abstract result in \cite[Cor.~3.11]{Chill} can be applied. In the following, we consider the restriction of the energy functional to the linear subspace $\Wo$ of $\VV^1$, and we write $\bar E:= E\vert_{\Wo}$, $\bar E':= E'\vert_{\Wo}$ and $\bar E'':= E''\vert_{\Wo}$.
	
We first notice that Sobolev's embedding theorem yields $ \Wot  \subset \mathcal{L}^{\infty}$. Hence, the restriction of $  \bar E^{\prime} $ to  $ \Wot $ is analytic with values in $ \mathcal{L}^{2}$. For more a more detailed reasoning in a similar situation we refer to \cite[Cor.~4.6]{Chill}.
For arbitrary $(u,v)\in\Wot$, we regard the second-order Fr\'echet derivative $ \bar E''$ as the linearization of $ \bar E'$ which means
\begin{align*}
	 \bar E''  (u,v): \Wo \to (\Wo)', \quad
	(\zeta_1,\eta_1) \mapsto  \bar E'' (u,v)(\zeta_1,\eta_1).
\end{align*} 
In view of \eqref{scndFrec.der},
$ \bar E''(u,v)$ can be interpreted as a continuous, symmetric, elliptic differential operator associated with the bilinear form 
\begin{align*}
	&\left\langle \bar E^{\prime \prime}(u,v)(\zeta_{1},\eta_{1}) ,(\zeta_{2},\eta_{2}) \right\rangle_{\Wo} \\
	&\quad =\intO \nabla \zeta_{1} \cdot \nabla \zeta_{2} + F^{\prime \prime}(u) \zeta_{1} \zeta_{2}\dx 
	+ \intG \gradg \eta_{1} \cdot \gradg \eta_{2} + G^{\prime \prime}(v) \eta_{1} \eta_{2}\dG,
	\quad (\zeta_1,\eta_1), (\zeta_2,\eta_2) \in \Wo,
\end{align*} 
which results from \eqref{scndFrec.der}. 
We further interpret the restriction of $\bar E'' (u,v)$ to $\Wot$ as
\begin{align*}
	 \bar E'' (u,v)\vert_{\Wot}: \Wot \to \LL^2, \quad
	(\zeta_1,\eta_1) \mapsto  \bar E''(u,v)(\zeta_1,\eta_1).
\end{align*} 
Invoking the Lax--Milgram theorem, we deduce that the operator
$ \bar E^{\prime \prime}(u,v) $ has a nontrivial resolvent set. In particular, there exists at least one $\lambda > 0$ such that the linear map
\begin{align*}
	\big( \bar E^{\prime \prime}(u,v) + \lambda \text{id}\big)^{-1}:
	(\Wo)' \to (\Wo)'
\end{align*}
is well-defined. As the embedding $ \Wo \hookrightarrow  (\Wo)'$ 
is compact, we conclude that this linear operator is compact.

By regularity theory for elliptic problems with bulk-surface coupling (see \cite[Thm.~3.3]{knopf-liu}), we conclude that the kernel
$  \text{Ker}\, \bar E^{\prime \prime}(u,v)$ belongs to $\Wot$. We further recall that the embedding $\Wot\emb \LL^2$ is compact.

Now, proceeding as in \cite[Sect.~6.2,\,Thm.~4]{evans}, the Fredholm alternative implies that the kernel
$  \text{Ker}\, \bar E^{\prime \prime}(u,v) $ is finite dimensional, and that the ranges 
$  \text{Rg}\, \bar E^{\prime \prime}(u,v) $ and 
$  \text{Rg}\, \bar E^{\prime \prime}(u,v)\big|\Index[-1pt]{\Wot} $ 
are closed in $ (\Wo)'  $ and $ \mathcal{L}^{2}  $, respectively. 
Furthermore,  
$ (\Wo)' $ can be expressed as the orthogonal sum of 
$  \text{Ker}\, \bar E^{\prime \prime}(u,v) $ and 
$  \text{Rg}\, \bar E^{\prime \prime}(u,v) $,
whereas 
$ \mathcal{L}^{2} $ can be expressed as the orthogonal sum of
$  \text{Ker}\, \bar E^{\prime \prime}(u,v) $ and
$  \text{Rg}\, \bar E^{\prime \prime}(u,v)\big|\Index[-1pt]{\Wot} $. 
In particular, we can define a continuous orthogonal projection 
$ P:\Wo \rightarrow \Wo $ with $ \text{Rg}\, P=\text{Ker}\, \bar E^{\prime \prime}(u,v) $.

Given these considerations, we can conclude the proof by applying the results in \cite[Cor.~3.11]{Chill} with $ X= \Wot$, $V=\Wo$, $ Y=\mathcal{L}^{2} $, and $ W= (\Wo)'$. 
\end{proof}

\begin{theorem}[Convergence to a single stationary point as $t\to\infty$]
	\label{Theorem:omega}
	Suppose that the assumptions \eqref{ass:dom}-\eqref{ass:pot} and \eqref{ass:ana} are satisfied. Then, for any $ (u_{0},v_{0}) \in \Wm  $ there exists a unique stationary point $ (u_{\infty},v_{\infty}) \in \mathcal{N}_m  $ such that 
	\begin{align}
	\label{CONV:OMEGA}
		\lim_{t \to\infty} 
		\bignorm{S_m^{L}(t)(u_{0},v_{0}) - (u_{\infty},v_{\infty}) }_{\HH^2} 
		=0.
	\end{align}
	In particular, this means that $\omega^L(u_0,v_0) = \{(u_{\infty},v_{\infty})\} \subset \mathcal N_m$.
\end{theorem}
\begin{proof}
	As discussed above, it suffices to handle the case $m=0$.
	Let $(u_0,v_0)\in \Wo$ be arbitrary, and let $(u^L,v^L,\mu^L,\theta^L)$ denote the corresponding weak solution of the system \eqref{CH:INT.}.		
	To prove the existence of a unique limit $(u_\infty,v_\infty) \in \mathcal N_0$, it suffices to show that the $ \omega $-limit set of $ (u_{0},v_{0}) $ consists of one single point. 
	
	Since $E$ is bounded from below and decreasing along weak solutions, we know that the limit
	\begin{align*}
	E_{\infty}:=\lim_{t \rightarrow \infty} E(u^{L}(t), v^{L}(t))
	\end{align*}
	exists. Hence, recalling the definition of $\omega^L(u_0,v_0)$, we infer that 
	\begin{align}
	\label{E_INF}
	E(u_*,v_*) = E_{\infty}\quad \text{for all}\; (u_*,v_*)\in \omega^{L}(u_{0},v_{0}).
	\end{align}
	Since $ E $ is nonincreasing along weak solutions, it holds that $ E(u^{L}(t), v^{L}(t))\geq E_{\infty}$ for all $ t\geq0 $. 
	
	If there exists a time $ t_{*}\ge 0 $ such that $ E(u^{L}(t_{*}),v^{L}(t_{*}))= E_{\infty}$, then the existence of a unique limit $(u_\infty,v_\infty) \in \mathcal N_m$ is a direct consequence of Lemma~\ref{lm:gradient}. In particular, we have
	\begin{align*}
		\big(u_0,v_0\big)
		= \big(u^L(t),v^L(t)\big)
		= \big(u^L(t_*),v^L(t_*)\big)
		= \big(u_\infty,v_\infty\big)
	\end{align*}
	for all $t\ge 0$.
	
	In the following, we will thus assume that $ E(u^{L}(t),v^{L}(t))> E_{\infty}$ for all $ t\geq 0 $.	
	From Lemma~\ref{LM.lojasiecwicz} we infer that for every critical point $ (u_*,v_*)\in \Wot $ of $E$ there exist constants $ \gamma(u_*,v_*) \in \left( 0, \frac{1}{2}\right]  $ and $ C(u_*,v_*)$, $\sigma(u_*,v_*)>0 $ such that the the \L ojasiewicz--Simon inequality \eqref{lojasiecwicz} holds for all 
	\begin{align*}
		(u,v) \in B_{\sigma(u_*,v_*)}(u_*,v_*) 
		= \Big\{ (u,v) \in \Wot \,:\, \norm{(u,v) - (u_*,v_*)}_{\Wot} < \sigma(u_*,v_*) \Big\}.
	\end{align*} 
	Obviously, the set $\omega^L(u_0,v_0)$ is covered by the union of all open balls $B_{\sigma(u_*,v_*)}(u_*,v_*)$ with $(u_*,v_*)\in \omega^L(u_0,v_0)$. 
	Since $\omega^L(u_0,v_0)$ is compact in $\Wot$, we can select finitely many points $(u_n,v_n) \in \omega^L(u_0,v_0)$, $n=1,...,N$ such that
	\begin{align*}
		\omega^L(u_0,v_0) 
		\subset \left(\,\bigcup_{n=1}^{N} B_{\sigma(u_n,v_n)}(u_n,v_n) \right)
		=: \mathcal U.
	\end{align*}
	Recalling \eqref{E_INF} and arguing as in \cite[Theorem 2.3]{CFP},
	we can now find constants 
	$\gamma\in(0,\frac 12]$ and $ C > 0 $ depending only on $L$, $\beta$ and $\Omega$ and $(u_0,v_0)$ but not on $n$ or $N$ such that for all $(u,v) \in \cal U $,
	\begin{align} \label{lojasiecwicz2}
	\abs{E(u,v)-E_{\infty}}^{1-\gamma}
	\leq C\norm{E^{\prime}(u,v)}_{(\Wo)'}\;.
	\end{align}
	We point out that $\gamma < \frac 12$ can be assumed without loss of generality.
	Since $\mathcal U$ is an open subset of $\Wot$, and $\omega^L(u_0,v_0) \subset \mathcal U$, there exists a time $t_0\ge 1$ such that $\big(u^L(t),v^L(t)\big) \subset \mathcal U$ for all $t\ge t_0$. In particular, this means that $\big(u^L(t),v^L(t)\big)$ satisfies the inequality \eqref{lojasiecwicz2} for all $t\ge t_0$.
	
	Let now $s\ge t_0$ be arbitrary, and without loss of generality we assume that $\mu^L$ and $\theta^L$ can be evaluated at the time $s$ with $\big(\mu^L(s),\theta^L(s)\big)\in \HH^1$. Now, applying integration by parts on \eqref{frstFrec.der}, we obtain
	\begin{align}\label{eq1}
	\left\langle E^{\prime}\big(u^L(s),v^L(s)\big) ,(\zeta,\xi) \right\rangle_{\Wo}
	=\intO \mu^L(s)\, \zeta\dx
		+\intG \theta^L(s)\, \xi \dG
	\end{align}
	for all $ (\zeta,\xi) \in \Wo $. We now define
	\begin{align*}
		c := \frac{\beta\int_\Omega \mu^L(s) \dx 
			+ \int_\Gamma \theta^L(s) \dG}
			{\beta^2\abs{\Omega} + \abs{\Gamma}}, \quad
		\bar\mu := \mu^L(s) - \beta c, \quad
		\bar\theta := \theta^L(s) - c,
	\end{align*}
	meaning that $(\bar\mu(s),\bar\theta(s))\in\Wo$. For any $(\zeta,\xi) \in \Wo $, we thus obtain
	\begin{align}\label{eq2}
	\intO \mu^L(s)\, \zeta\dx
		+ \intG \theta^L(s)\, \xi \dG
	= \intO \bar\mu \zeta\dx
		+ \intG  \bar\theta \xi \dG.
	\end{align}
	Applying the bulk-surface Poincar\'e type inequality presented in \cite[Lem.~7.1]{knopf-liu} (with $K:=L$ and $\alpha:=\beta$), we conclude that
	\begin{align}\label{ineq1}
	\bignorm{E^{\prime}\big(u^L(s),v^L(s)\big)}_{(\Wo)'} 
	\le \norm{(\bar\mu,\bar\theta)}_{\LL^2} \le c_P \norm{(\bar\mu,\bar\theta)}_{L,\beta} = c_P \bignorm{\big(\mu^L(s),\theta^L(s)\big)}_{L,\beta}\;,
	\end{align}
	where the constant $c_P>0$ depends only on $L$, $\beta$ and $\Omega$.

	Recalling that $(u,v)\in C([1,\infty);\HH^2)$ and interpreting $(u^L,v^L,\mu^L,\theta^L)$ as a weak solution starting at time $s$ (instead of zero), the energy inequality \eqref{ENERGY:KLLM} yields that \eqref{5.29} holds for all $t\in\RP$ with $t\ge s$. As the right-hand side of this inequality does not depend on $t$, we may pass to the limit $t\to \infty$ on the left-hand side, which gives
	\begin{align}
	\label{EST:EN}
	\dfrac{1}{2}\int_{s}^{\infty}
		\bignorm{\big(\mu^L(\tau),\theta^L(\tau)\big)}_{L,\beta}^2 \dtau
		\leq E(u^{L}(s),v^{L}(s))-E_{\infty}.
	\end{align}
	Hence, combining \eqref{lojasiecwicz2}, \eqref{ineq1} and \eqref{EST:EN}, we infer that
	\begin{align*}
	\int_{s}^{\infty}
		\bignorm{\big(\mu^L(\tau),\theta^L(\tau)\big)}_{L,\beta}^2 \dtau 
	\leq C \bignorm{\big(\mu^L(s),\theta^L(s)\big)}_{L,\beta}^{1/(1-\gamma)}
	\end{align*}
	for some constant $C\ge 0$ depending only on $L$, $\beta$ and $\Omega$, and almost all $s \ge t_0$.
	We recall that $\gamma<\frac 12$ was assumed without loss of generality.
	As the expression on the left hand side can also be bounded by a constant via the energy inequality \eqref{ENERGY:KLLM}, we invoke \cite[Lem.~7.1]{FS}
	to deduce that
	\begin{align*}
	\nabla\mu^{L}\in L^{1}(t_{0},\infty;L^{2}(\Omega)), \text{ \ } \nabla\theta^{L}\in L^{1}(t_{0},\infty;L^{2}(\Gamma)),\text{ \ } \beta\theta^{L}-\mu^{L}\in L^{1}(t_{0},\infty;L^{2}(\Gamma)).
	\end{align*}
	(Alternatively, one could also argue as in \cite[Proof of Thm.~2.3]{CFP} to obtain the same result.)
	Recalling the weak formulations \eqref{WF:KLLM:1} and \eqref{WF:KLLM:2}, we further obtain
	$(\delt u^L, \delt v^L) \in L^{1}\big(t_{0},\infty;\HH_\beta^{-1}\big)$.
	Consequently, for any $t_0<t_1<t_2$, we have
	\begin{align*}
	\bignorm{\big(u^L(t_2),v^L(t_2)\big) - \big(u^L(t_1),v^L(t_1)\big)}_{L,\beta,*} 
	\le \int_{t_1}^{t_2} \bignorm{\big(\delt u^L(\tau),\delt v^L(\tau)\big)}_{L,\beta,*} \dtau\;,
	\end{align*}
	and the right-hand side converges to zero as $t_2>t_1\to \infty$.
	Using the interpolation inequalities from Lemma~\ref{LEM:INT:2} and Corollary~\ref{COR:INT} along with the smoothing property \eqref{SMOOTH:KLLM}, we conclude that
	\begin{align*}
	&\bignorm{\big(u^L(t_2),v^L(t_2)\big) - \big(u^L(t_1),v^L(t_1)\big)}_{\HH^2} \\
	&\quad\le C\left(\frac{t+1}{t}\right)^{\frac 38}\bignorm{\big(u^L(t_2),v^L(t_2)\big) - \big(u^L(t_1),v^L(t_1)\big)}_{L,\beta,*}^{\frac 14} \\
	&\quad\longrightarrow 0, \quad
	\text{as $t_2>t_1\to\infty$}.  
	\end{align*}
	Invoking the Cauchy criterion, this implies the existence of a unique limit $(u_{\infty},v_{\infty}) \in \Wot$ such that the convergence assertion \eqref{CONV:OMEGA} is satisfied. In particular, it holds that $(u_{\infty},v_{\infty}) \in \omega^L(u_0,v_0) \subset \mathcal N_0$ which completes the proof.
\end{proof}

\medskip

We will now see that solutions starting in $\mathcal A^L_m$ converge to a stationary point also for $t\to -\infty$.

\begin{theorem}[Convergence to a single stationary point as $t\to -\infty$]
	\label{Theorem:alpha}
	Suppose that the assumptions \eqref{ass:dom}-\eqref{ass:pot} and \eqref{ass:ana} hold. Then, for any $ (u_{0},v_{0}) \in  \mathcal A_{m}^{L}  $, the corresponding state $(u^{L}(t),v^{L}(t))$ exists for all $t\in\R$, and there exists a unique stationary point $ (u_{-\infty},v_{-\infty}) \in \mathcal{N}_{m} $ such that
	\begin{align}
		\label{CONV:ALPHA}
		\lim_{t \to -\infty} 
		\norm{\big(u^{L}(t),v^{L}(t)\big) - (u_{-\infty},v_{-\infty})}_{\HH^2} = 0.
	\end{align}
	In particular, this means that $\alpha^L(u_0,v_0) = \{(u_{-\infty},v_{-\infty})\} \subset \mathcal N_m$.
\end{theorem}
\begin{proof}
	As in the proof of Theorem~\ref{Theorem:omega}, it suffices to address the case $m=0$. 
	Let $(u_{0},v_{0}) \in \mathcal A_0^L$ be arbitrary, and let $(u^L,v^L,\mu^L,\theta^L)$ denote the corresponding weak solution of the system \eqref{CH:INT.}.
	
	We recall that the state $\big(u^L(t),v^L(t) \big)$ associated with $ (u_{0},v_{0}) $ exists for all times $t\in \R$ and lies in $A^L_0$.
	According to Definition~\ref{def:att} and Theorem~\ref{thm.glbatt}, the global attractor $\mathcal A^L_0$ is a compact subset of $\Wot$. Hence, since $E$ is continuous, we know that $E(\mathcal A^L_0)\subset \R$ is compact. As $E$ is nonincreasing along weak solutions, we infer that the limit 
	\begin{align*}
	E_{-\infty}:=\lim_{t \rightarrow -\infty} E(u^{L}(t),v^{L}(t))
	\end{align*}
	exists. Recalling the definition of $\alpha^L(u_0,v_0)$, we conclude that
	\begin{align*}
		E(u_*,v_*) = E_{-\infty} \quad\text{for all $(u_*,v_*) \in \alpha^L(u_0,v_0)$}.
	\end{align*}
	
	Based on this, the claim of Theorem~\ref{Theorem:alpha} can be established by proceeding analogously to the proof of Theorem~\ref{Theorem:omega}.
\end{proof}

As a consequence, we obtain the following result which provides further information of the geometric structure of the global attractor.
\begin{corollary} \label{gradient-like}
	Under the assumptions \eqref{ass:dom}-\eqref{ass:pot} and \eqref{ass:ana}, the global attractor $ {\cal{A}}_{m}^{L} $ of the dynamical system	$ (\Wm, S_{m}^{L}(t)) $ can be defined as the union of the unstable manifolds of all stationary points, i.e.,
	\begin{align}
	\label{ATT:KLLM}
	{\cal{A}}_{m}^{L}=\underset{(u_*,v_*) \in \mathcal{N}_{m}}{\bigcup}\mathcal{M}^L_\text{u}\big(\{(u_*,v_*)\}\big).
	\end{align}
\end{corollary}
The assertion follows directly from Theorem~\ref{Theorem:alpha} and Definition~\ref{DEF:UM}.

\medskip

\begin{remark} \normalfont
	If the set of stationary points were finite, Corollary~\ref{gradient-like} could also be established without the \L ojasiewicz--Simon inequality. However, as the set of stationary points is usually infinite, we would merely get ${\cal{A}}_{m}^{L}=\mathcal{M}^L_\text{u}( \mathcal{N}_{m})$ (see \eqref{EQ:UNST}) instead of \eqref{ATT:KLLM} if we disregard the \L ojasiewicz--Simon inequality.
\end{remark}

\subsection{Long-time dynamics of the GMS model}\label{LT:GMS}

Based on a result from \cite{PZ} on closed semigroups, the existence of a global attractor in the ``finite-energy phase space'' has already been established for the GMS model ($L=0$) in \cite[Cor.~3.11]{GMS}.

However, it is also possible to obtain the existence of a global attractor for the dynamical system $ \big( \Wm,S_{m}^{0}(t) \big)  $ by using the same arguments as in Section~\ref{sec:long-time}.
It is worth mentioning that the GMS model has the same set of stationary points $ \mathcal{N}_{m} $ as the KLLM model. Hence, we already know from Subsection~\ref{sec:stat} that the stationary point set of the GMS model is nonempty and bounded in $ \Wm $.
The smoothing property \eqref{SMOOTH:GMS} implies the asymptotic compactness of the dynamical system $ \big( \Wm,S_{m}^{0}(t) \big) $, meaning that an analogue of Lemma~\ref{lm.asympcompact} can be established.
Furthermore, it can be shown that the energy functional $E$ is a strict Lyapunov function for the dynamical system $ \big( \Wm,S_{m}^{0}(t) \big) $ and consequently, the existence of a unique global attractor $\mathcal A_m^0$ can be proved analogously as in Subsection~\ref{subsec:att}.

It was further established in \cite[Thm.~3.22]{GMS} that the $\omega$-limits set consists of one single stationary point. We point out that Theorem~\ref{Theorem:alpha} could be adapted to the case $L=0$ and thus, Corollary~\ref{gradient-like} remains true for the GMS model.

\section{Stability of the global attractor for the GMS model} \label{sec:contglbatt}

We intend to study the stability of the global attractor $\mathcal A^0_m$ of the GMS model ($L=0$) with respect to perturbations $\mathcal A^L_m$ with small $L>0$. 
At least formally, this means we have to investigate the asymptotic limit $L\to 0$ of the family $\{ \mathcal A^L_m \}_{L>0}$ of global attractors for the KLLM model. 
To provide a rigorous notion of such a limit, we recall the following definition which is included in \cite[Thm.~7.2.8]{chueshov}.

\begin{definition}[Upper semicontinuity of global attractors]
	\label{DEF:USA}
	Let $X$ be a Banach space, and let $\Lambda$ be a metric space. 
	Suppose that for any $\lambda \in \Lambda$, $ (X,\{S^{\lambda}(t)\}_{t\ge 0}) $ is a dynamical system possessing a global attractor $\mathcal A^\lambda \subset X$.
	Then, the family $ \{\mathcal A^\lambda \}_{\lambda\ge 0} $ is called \emph{upper semicontinuous at the point $ \lambda_* \in \Lambda$} if
	\begin{align*}
	\lim_{\lambda\rightarrow \lambda_* } \dist_X\big(\mathcal A^{\lambda},\mathcal A^{\lambda_*} \big) =0.
	\end{align*}
\end{definition}

To prove that the family $ \{ \mathcal A_{m}^L\}_{L\ge 0} $ is upper semicontinuous at $L=0$, we need to know how a weak solution $(u^L,v^L,\mu^L,\theta^L)$ to the KLLM model behaves in the asymptotic limit $L\to 0$. 

\begin{lemma}[The asymptotic limit $L \to 0$]
	\label{LEM:CONV:GMS}
	Suppose that \eqref{ass:dom}--\eqref{ass:pot} hold, let $L>0$ and $m\in\R$ be arbitrary, and let $(u_0,v_0) \in \Wm$ be any initial datum. 
	Let $(u^L,v^L,\mu^L,\theta^L)$ denote the corresponding unique weak solution to the KLLM model, 
	and let $(u^0,v^0,\mu^0,\theta^0)$ denote the corresponding weak solution of the GMS model. Then, for any $T_*>0$, it holds that
	\begin{align}\label{CONV:GMS}
	(u^L,v^L) \to (u^0,v^0) \quad \text{strongly in $C([0,T_*];\LL^2) $ as $L\to 0$}.
	\end{align}
	Moreover, there exist constants $A,B>0$ depending (monotonically increasingly) on $\norm{(u_0,v_0)}_{\HH^1}$ but not on $L$ such that for all $L>0$,
	\begin{align}\label{cnvrg.rateLW}
	\bignorm{(u^L,v^L) - (u^0,v^0)}_{C([1,T_*];\LL^2)} \leq A e^{BT_*}\, L^{\frac 14}.
	\end{align}
\end{lemma}

\begin{proof}
	The convergence \eqref{CONV:GMS} has already been established in \cite[Thm.~4.1]{KLLM}. Moreover, is was also shown that there exists a positive constant $c(T_*)$ depending (monotonically increasingly) on $\norm{(u_0,v_0)}_{\HH^1}$ and $T_*$ but not on $L$ such that for all $L>0$,
	\begin{align*}
		\bignorm{(u^L,v^L) - (u^0,v^0)}_{C([1,T_*];\LL^2)} \leq c(T_*)\, L^{\frac 14}.
	\end{align*} 
	Studying the proof of \cite[Thm.~4.1]{KLLM} carefully, we find that the constant $c(T_*)$ depends (at most) exponentially on $T_*$. This results from the application of Gronwall's lemma. We can thus find constants $A,B>0$ independent of $L$ such that $c(T_*) = A e^{BT_*}$. Thus, the proof is complete.
\end{proof}

By means of the smoothing properties \eqref{SMOOTH:KLLM} and \eqref{SMOOTH:GMS}, we can further generalize this convergence result to higher order norms.

\begin{corollary}[Convergence rates for the limit $L\to 0$]
	\label{COR:CONV:GMS}
	Suppose that \eqref{ass:dom}--\eqref{ass:comp} hold, let $L>0$ and $m\in\R$ be arbitrary, and let $(u_0,v_0) \in \Wm$ be any initial datum. 
	Let $(u^L,v^L,\mu^L,\theta^L)$ denote the corresponding unique weak solution to the KLLM model, 
	and let $(u^0,v^0,\mu^0,\theta^0)$ denote the corresponding weak solution of the GMS model. Then, for any $T_*>1$, 
	\begin{align}
		\label{CONV}
		(u^L,v^L) \to (u^0,v^0) \quad \text{strongly in $C([1,T_*];\HH^2) $ as $L\to 0$}.
	\end{align}
	Moreover, for any $T_*>1$, there exist constants $C_1(T_*),C_2(T_*)>0$ depending (monotonically increasingly) on $\norm{(u_0,v_0)}_{\HH^1}$ but not on $L$ such that
	\begin{align}
		\label{CRATE:1}
		\bignorm{(u^L,v^L) - (u^0,v^0)}_{C([1,T_*];\HH^1)} \leq C_1(T_*)\, L^{\frac{1}{6}},\\
		\label{CRATE:2}
		\bignorm{(u^L,v^L) - (u^0,v^0)}_{C([1,T_*];\HH^2)} \leq C_2(T_*)\, L^{\frac{1}{12}}.
	\end{align}
\end{corollary}

\begin{proof}
	Let $T_*>1$ be arbitrary.
	We first recall that $(u^L,v^L), (u^0,v^0) \in C([1,T_*];\HH^2)$ according to Proposition~\ref{PROP:CD:KLLM} and Proposition~\ref{PROP:CD:GMS}. Using the interpolation inequality \eqref{INT:H1:ALT} from Lemma~\ref{LEM:INT:2}, and the smoothing properties \eqref{SMOOTH:KLLM} and \eqref{SMOOTH:GMS}, we deduce that
	\begin{align*}
		&\underset{t\in[1,T_*]}{\sup}\; \norm{(u^L,v^L)(t) - (u^0,v^0)(t)}_{\HH^1} \\
		&\quad \le C \underset{t\in[1,T_*]}{\sup}\; \norm{(u^L,v^L)(t) - (u^0,v^0)(t)}_{\LL^2}^{\frac 23} \left[ (C_* + C_0) (1+T_*)^{\frac 12} \right]^{\frac 13}.
	\end{align*}
	Invoking the convergence rate \eqref{cnvrg.rateLW} we directly conclude \eqref{CRATE:1} from the above estimate.
	The estimate \eqref{CRATE:2} can be established similarly. In particular, this proves \eqref{CONV} and thus, the proof is complete.
\end{proof}

We now intend to establish the following stability result which is the main result of this section. Here, the term ``stability'' is to be understood as semicontinuity of the family of perturbed global attractors.

\begin{theorem}[Stability of the global attractor $\mathcal A^0_m$]
	\label{uppercont.GMS}
	Suppose that the assumptions \eqref{ass:dom}-\eqref{ass:ana} hold and let $m\in\R$ be arbitrary. Then, the family of global attractors  $\{ \mathcal A_{m}^{L}\} _{L\geq0} $ is upper semicontinuous at $L=0$ in the sense of Definition~\ref{DEF:USA}.
\end{theorem}

To prove this theorem, we will exploit the following abstract result which can be found in \cite[Thm.~7.2.8]{chueshov}.

\begin{proposition}[A criterion for upper semicontinuity of global attractors]
	\label{PROP:CRIT}
	Let $X$ be a Banach space, and let $\Lambda$ be a metric space. 
	Suppose that for any $\lambda \in \Lambda$, $ (X,\{S^{\lambda}(t)\}_{t\ge 0}) $ is a dynamical system possessing a global attractor $\mathcal A^\lambda \subset X$.
	We further assume that the following conditions hold:%
	\begin{enumerate}[label = $\mathrm{(\roman*)}$]
		\item There exists a compact set $ K \subset X$ such that $\mathcal A^\lambda \subset K $ for all $\lambda\ge 0$.
		\item If $(x_k)_{k\in\N} \subset X$ and $ \lambda_{k} \in \Lambda$ are sequences satisfying
		\begin{itemize}
			\item $x_k \in \mathcal A^{\lambda_k}$ for all $k\in\N$,
			\item $x_k \to x_*$ as $k\to\infty$,
			\item $\lambda_k \to \lambda_*$ as $k\to\infty$,
		\end{itemize}
		then there exists $t_*>0$ such that $ S^{\lambda_{k}}(t)x_{k} \rightarrow S^{\lambda_*}(t)x_* $ in $X$ for all $t>t_*$.
	\end{enumerate}
	Then the family $ \{\mathcal A^\lambda \}_{\lambda\ge 0} $ is upper semicontinuous at the point $ \lambda_* $.
\end{proposition}

\begin{proof}[Proof of Theorem~\ref{uppercont.GMS}]
	In order to apply Proposition~\ref{PROP:CRIT}, we need to verify the conditions (i) and (ii) imposed therein.
	
	\textit{Step 1:} To verify Condition (i), we show that there exists a compact set $\mathcal K_m \subset \Wm$ independent of $L$ such that $\mathcal A^L_m \subset \mathcal K_m$ for all $L\ge 0$.
	
	For the KLLM model ($L>0$), Lemma~\ref{LEM:LIMSET} implies that for every $(u_0,v_0)\in\Wm$, the set $\omega^L(u_0,v_0)$ is included in the set of stationary points. As discussed Subsection~\ref{LT:GMS}, the same holds true for the GMS model ($L=0$). We further recall that the set $\mathcal N_m$ is a bounded subset of $\Wmt$ which does not depend on the parameter $L$.
	Hence, we can choose an open, bounded set $\mathcal U \subset \HH^2$ such that for all $L\ge 0$,
	\begin{align*}
	\omega^L(u_0,v_0) \subseteq \mathcal N_m 
	\subset \mathcal U
	\end{align*}
	for all $(u_0,v_0)\in \Wm$.
	Consequently, for every $L\ge 0$ and every initial datum $(u_0,v_0)\in\Wm$, there exists a time
	$ t^L_{(u_0,v_0)} \ge 1$ such that 
	\begin{align}\label{pointdiss}
	S^L_{m}\left(t\right)(u_0,v_0)\in \mathcal U \quad \text{for all $ t\geq  t^L_{(u_0,v_0)} $.}
	\end{align}
	Now, we define the set 
	\begin{align}\label{absorb}
	\mathcal K_{m}
	:=  \overline{\underset{L \in[0,\infty)}{\bigcup} \; \underset{t\geq 1}{\bigcup} \; S^L_{m}\left(t\right) \mathcal U}^{\HH^1} .
	\end{align}
	Since $ \mathcal U $ is bounded in $\Wmt$ uniformly in $L\ge 0$, the smoothing property \eqref{SMOOTH:KLLM} of the KLLM model implies that $S^L_{m}\left(t\right) \mathcal U$ is bounded in $\HH^3$ uniformly in $L>0$ and $t\ge 1$. Moreover, using the smoothing property \eqref{SMOOTH:GMS} of the GMS model, we conclude that $S^{\infty}_{m}\left(t\right)\mathcal U$ is bounded in $\HH^3$ uniformly in $t\ge 1$. This entails that $ \mathcal K_{m}\subset \Wm $ is bounded in $\HH^{3}$, closed in $\HH^1$, and thus compact in $\Wm$.
	We further infer from \eqref{pointdiss} that $ \mathcal K_{m} $ is an \emph{absorbing set} for every semigroup $S^L_{m}(t)$ with $L\ge 0$ 
	(i.e., for every $L\ge 0$ and every bounded set $B\subset \Wm$, there exists a time $t_B^L\ge 1$ such that $S^L_m(t)B \subset \mathcal K_{m}$ for all $t\ge t_B^L$).
	Recall that for all $L\ge 0$, the global attractor $A_m^L$ is a bounded subset of $\Wm$. We thus conclude that there exists a time $t^L\ge 1$ such that 
	\begin{align*}
		\mathcal A_m^L = S_m^L(t^L) \mathcal A_m^L \subset \mathcal K_m
	\end{align*}
	due to the invariance property of the global attractor. 	
	This directly proves that $\mathcal A^L_m \subset \mathcal K_m$ for all $L\ge 0$.
	
	\textit{Step 2:} To verify Condition (ii) of Proposition~\ref{PROP:CRIT}, let $\{(u_k,v_k)\}_{k\in\N} \subset \Wm$ be any sequence with $(u_k,v_k) \in \mathcal A^{L_k}_m$ and $(u_k,v_k) \to (u_*,v_*)$ in $\HH^1$ as $k\to\infty$, and let $\{L_k\}_{k\in\N} \subset \RP$ be any sequence with $L_k\to 0$ as $k\to\infty$.
	
	Let now $t\ge t_*:=1$ and $\eps>0$ be arbitrary. Using the continuous dependence estimate from Proposition~\ref{PROP:CD:GMS}, we deduce that
	\begin{align}
		&\bignorm{S^{L_k}_m(t)(u_k,v_k) - S^0_m(t)(u_*,v_*)}_{\HH^1} \notag\\
		&\le \bignorm{S^{L_k}_m(t)(u_k,v_k) - S^0_m(t)(u_k,v_k)}_{\HH^1} 
			+ \bignorm{S^0_m(t)(u_k,v_k) - S^0_m(t)(u_*,v_*)}_{\HH^1} \notag \\
		\label{RHS}
		&\le \underset{(u,v)\in \mathcal{K}_{m}}{\sup} \;
		\bignorm{S^{L_k}_m(\cdot)(u,v) - S^0_m(\cdot)(u,v)}_{C([0,t];\HH^1)}  
			+ \Lambda_1^0(t) \bignorm{(u_k,v_k) - (u_*,v_*)}_{L,\beta,*}. 
	\end{align}
	Since $L_k\to 0$ as $k\to\infty$, and since $\mathcal{K}_{m}$ is bounded, Corollary~\ref{COR:CONV:GMS} implies the existence of a number $K_1\in\N$ such that for all $k\ge K_1$, the first summand in \eqref{RHS} is smaller than $\eps/2$. Furthermore, the convergence $(u_k,v_k) \to (u_*,v_*)$ in $\HH^1$ directly implies that 
	\begin{align*}
		\bignorm{(u_k,v_k) - (u_*,v_*)}_{L,\beta,*} \to 0 \quad\text{as $k\to\infty$}.
	\end{align*}
	Hence, there exists a number $K_2\in\N$ such that for all $k\ge K_2$, the second summand in \eqref{RHS} is smaller than $\eps/2$. In summary, we get
	\begin{align*}
		\bignorm{S^{L_k}_m(t)(u_k,v_k) - S^0_m(t)(u_*,v_*)}_{\HH^1} < \eps
		\quad \text{for all $k\ge K:=\max\{K_1,K_2\}$,}
	\end{align*}
	and since $\eps>0$ was arbitrary, this verifies Condition (ii) of Proposition~\ref{PROP:CRIT}.
	
	Eventually, Proposition~\ref{PROP:CRIT} can be applied on the family $\{\mathcal A^L_m\}_{L\ge 0}$ which proves the assertion of Theorem~\ref{uppercont.GMS}.
\end{proof}

\begin{remark} \label{REM:FAIL} \normalfont
	We point out that Proposition~\ref{PROP:CRIT} cannot be used to prove the stability of the global attractor associated with $L=\infty$. This because the semigroups associated with $L<\infty$ and the semigroup associated with $L=\infty$ would have to be defined on the same domain $X$. However, due to the different mass conservation law of the LW model, its solution operator $S^0_{(m_1,m_2)}$ would have to be defined on the linear subspace
	\begin{align*}
		\VV_{(m_1,m_2)} = \big\{ (u,v) \in \VV^1 \suchthat \mean{u}_\Omega = m_1 \;\text{and}\; \mean{v}_\Gamma = m_2\big\}.
	\end{align*}
	This entails that the only reasonable choice for $X$ in Proposition~\ref{PROP:CRIT} would be $X := \Wm \cap \VV_{(m_1,m_2)}$. Thus, $X$ is nonempty only if $m = \beta m_1 + m_2$. However, even in this case, we cannot ensure that $S^L_m X \subseteq X$ for any $L<\infty$, which means that $S^L_m$ does not define a semigroup on the space $X$. Consequently, Proposition~\ref{PROP:CRIT} is not applicable and thus, the results of Section~\ref{sec:contglbatt} cannot be transferred to the scenario $L\to\infty$.
	
	In particular, for the aforementioned reasons, we cannot find a compact set $\mathcal K_m$ which acts as an absorbing set the semigroup corresponding to $L=\infty$ but also for the semigroups associated with $L<\infty$. (Note that if $\mathcal K_m$ were a compact absorbing set, it also would have to absorb itself.) 
	As such a set $\mathcal K_m$ will play a crucial role in the subsequent section, it is also not possible to adapt the results of Section~\ref{SEC:EXP} to the situation $L\to\infty$.
\end{remark}

\section{Existence of a robust family of exponential attractors} \label{SEC:EXP}
In this section, for every $L\in [0,1]$, we intend to construct an exponential attractor for the dynamical system $(\WW^1_{\beta,m} , \{S^{L}_{m}\left(t\right)\}_{t\ge 0})$. We further show that the exponential attractor associated with $L=0$ is robust against perturbations $L>0$ in some certain sense.

We first recall the definition of exponential attractors (see, e.g., \cite[Def.~4.1]{EMZ} in the discrete case and \cite[Def.~7.4.4]{chueshov} in the continuous case).

\begin{definition}[Exponential attractors] \label{DEF:EXP:ATT}
	Let $ M $ be a metric space, let $I=[0,\infty)$ or $I=\N_0$, and let 
	$ \{S(i)\}_{i\ge  0} $ be a semigroup on $M$. Then, a compact set $ {\mathfrak{M}} \subset M $ is called an exponential attractor for the dynamical system $\big(M,\{S(i)\}_{i\in I}\big)$ if the following properties hold:
	\begin{enumerate}[label = $\mathrm{(\roman*)}$]
		\item The set $ {\mathfrak{M}} $ is compact in $M$ and has finite fractal dimension, i.e., $$\dim_{\mathrm{frac},M}(\mathfrak M)<\infty.$$
		\item The set $ {\mathfrak{M}} $ is forward invariant under the semigroup  $ \{S(i)\}_{i\ge  0} $, i.e.,  
		$$ S(i){{\mathfrak{M}}}\subset {{\mathfrak{M}}} 
		\quad\text{for all $i\in I$.}$$
		\item The set $ {\mathfrak{M}} $ is an exponentially attracting set for the dynamical system $\big(M,\{S(i)\}_{i\in I}\big)$, i.e., for every bounded subset $ B \subset M $, there exist constants $C,a>0$ such that
		\begin{align*}
		\dist_{M} \left( S(i)B,{{\mathfrak{M}}}\right) \leq C e^{-a i} \quad \text{for all $i\in I$.}
		\end{align*}
	\end{enumerate}
\end{definition}

\medskip

In the proof of Theorem~\ref{uppercont.GMS}, we obtained the existence of a compact absorbing set $\mathcal K_{m} $ (that does not depend on $L$) such that for all $L\ge 0$ and every bounded set $ B\subset \Wm $ there exists a time $ t_0^L(B) \ge 1$ such that
\begin{align*}
	S_{m}^{L}(t)B \subset \mathcal K_{m} \quad \text{for all $t\geq t_0^L(B)$.}
\end{align*}
As for any $t\ge 0$ the state $S_{m}^{L}(t)$ depends continuously on the parameter $L$, we infer that also the time $t_0^L(B)$ depends continuously on $L$. We thus obtain
\begin{align*}
	\bigcup_{L\in[0,1]} S_{m}^{L}(t)B \subset \mathcal K_{m} \quad \text{for all $t\geq t_0(B) := \underset{L\in[0,1]}{\max}\, t_0^L(B)$.}
\end{align*}
It is worth mentioning that the set $ \mathcal K_{m} $ might not be forward invariant under the semigroup $ \{S_{m}^{L}(t)\}_{t\geq0} $ for any $L\ge 0$. However, as $ \mathcal K_{m} $ is a bounded absorbing set, it also absorbs itself. Thus, defining $ T_{0} := t_0(\mathcal K_m)\ge 1$, we conclude that
\begin{align}
\label{EX:T0}
\bigcup_{L\in[0,1]}S_{m}^{L}(t)\mathcal K_{m} \subset\mathcal K_{m} \quad\text{for all $t\geq T_0$.}
\end{align}

The next step is to construct a family $\{\mathfrak M^L_D\}_{L\in[0,1]}$ of exponential attractors for the discrete dynamical systems  
\begin{align}
	\left( \mathcal K_{m}, \{D_{m}^{L}(n)\}_{n\in\N_0} \right),
	\quad\text{where}\quad
	D^L_m(n) := S^L_m(nT_0) \quad\text{for all $n\in\N_0$},
\end{align}
with $L\in[0,1]$. 
We further show that the exponential attractor associated with $L=0$ is robust against perturbations of the parameter $L$.
Eventually, we will extend these results to the continuous dynamical systems $ \big( \mathcal K_{m},  S_{m}^{L}(t)\big)  $ with $L\in[0,1]$. 

It obviously holds that for all $(u,v)\in\mathcal K_m$,
\begin{align}
	D^L_m(n)(u,v) 
	= \big[ \underbrace{ D^L_m(1) \circ ... \circ D^L_m(1)}_{\text{$n$ times}} \big] (u,v)
	=: \big[ D^L_m(1) \big]^n(u,v).
\end{align}
Therefore, the operator $D^L_m(1) = S^L_m(T_0)$ will play a crucial role in the analysis.

\subsection{Robust exponential attractors for discrete dynamical systems}

To prove the existence of a robust family of exponential attractors for the discrete dynamical system $(\mathcal K_{m}, \{D_{m}^{L}(n)\}_{n\in\N_0})$ we first present a simple generalization of the abstract result \cite[Thm.~4.4.]{EMZ}, which can be established by a rescaling argument.

\begin{lemma}[Robust exponential attractors for general discrete dynamical systems]
	\label{LEM:GEN:DISC}
	Suppose that $X$ and $Y$ are Banach spaces such that $Y$ is compactly embedded in $X$, and let $M$ be a bounded subset of $X$. Let $L_*>0$ be arbitrary. We further assume that there exists a family $\{\mathcal F^L\}_{L\in[0,L_*]}$ of operators $\mathcal F^L:M\to M\cap Y$ which satisfies the following assumptions.
	\begin{enumerate}[label = $\mathrm{(\roman*)}$]
		\item There exist a constant $\Lambda\ge 0$ such that for all $x_1,x_2\in M$ and all $L\in [0,L_*]$,
		\begin{align}\label{Lipcomp.LW}
			\bignorm{\mathcal F^L(x_1) - \mathcal F^L(x_2)}_Y 
			\le \Lambda \norm{x_1-x_2}_X \; .
		\end{align}
		\item There exists constants $\Theta>0$ and $\zeta\in(0,1]$ such that for all $x\in M$, $L\in [0,L_*]$ and $n\in\N_0$,
		\begin{align}\label{asymlmt.LW}
			\bignorm{\big[\mathcal F^L\big]^{n}(x) - \big[\mathcal F^0\big]^{n}(x)}_X
			\leq \Theta^n\, L^\zeta \; ,
		\end{align}
		where $\big[\mathcal F^L\big]^{n}$ denotes the $n$-fold composition $\mathcal F^L \circ ... \circ \mathcal F^L$.
	\end{enumerate}
	Then, for every $ L \in [0,L_*] $, there exists an exponential attractor $ \mathcal{M}^L \subset M$ for the discrete semigroup $(M,\{\mathcal D^L(n)\}_{n\in\N_0})$ with $\mathcal D^L(n):=[\mathcal F^L]^n$, $n\in\N_0$ in the sense of Definition~\ref{DEF:EXP:ATT}.
	 
	Moreover, these exponential attractors can be chosen in such a way that the following properties hold:
	\begin{enumerate}[label = $\mathrm{(\alph*)}$, ref = $\mathrm{(\alph*)}$ ]
		\item There exists constants $C_1>0$ and $0<\gamma<1$ such that for all $L\ge 0$,
		\begin{align}\label{tend.LW}
		\dist_{\mathrm{sym},X}(\mathcal{M}^L,\mathcal M^0)\leq C_1 L^{\zeta\gamma}.
		\end{align}
		\item The fractal dimension of $ \mathcal{M}^L $ is uniformly bounded, i.e., there exists a constant $C_2> 0$ such that for all $ L \in [0,L_{0}]$: 
		\begin{align}\label{expunibd.LW}
		\dim_{\mathrm{frac},X}(\mathcal{M}^L)\leq C_2.
		\end{align}
		\item There exist constants $C_3,A>0$ such that for all $L\in [0,L_*]$ and all $n\in\N_0$,
		\begin{align}
			\dist_X\big(\mathcal D^L(n) M\,,\, \mathcal M^L\big) \le C_3 e^{-A n},
		\end{align}
		meaning that the rate of convergence to these attractors is uniform in $L$.
	\end{enumerate}
\end{lemma}

\medskip

\begin{proof}
	We first assume that (ii) holds with $\zeta=1$. In this case, the assertions of Lemma~\ref{LEM:GEN:DISC} have already been established in \cite[Thm.~4.4.]{EMZ}. We point out that the assertion (c) is not explicitly stated in \cite[Thm.~4.4.]{EMZ} but follows directly from its proof, in which the authors showed that
	\begin{align*}
		\dist_X\big(\mathcal D^L(n) M\,,\, \mathcal M^L\big)
		= \dist_X\big([\mathcal F^L]^n M\,,\, \mathcal M^L\big)
		\le 2^{1-n} R
		\quad\text{for all $n\in\N_0$.}
	\end{align*}
	Here, $R>0$ denotes the radius of a fixed ball $B_R(x_0)$ in $X$ with some center $x_0\in M$ such that $M\subset B_R(x_0)$. Hence, choosing $C_3 = 2R$ and $A=\ln(2)$, we obtain (c).
	
	Let us now assume that (ii) holds with $\zeta\in(0,1)$. In this case, the assertions can be established by slightly modifying the proof of \cite[Thm.~4.4.]{EMZ}. Alternatively, the change of variables $L\mapsto K:=L^\zeta$ could be used to transform back to the already known situation $\zeta=1$.
\end{proof}

We now want to apply this abstract result on the discrete dynamical system $(\mathcal K_m,\{D^L_m(n)\}_{n\in\N_0})$. To this end, we first need a compact Lipschitz estimate that is uniform in $L\in[0,1]$.

\begin{lemma}[Uniform compact Lipschitz estimate]\label{LEM:CLE}
	Suppose that \eqref{ass:dom}--\eqref{ass:comp} hold, and let $m\in\R$ and $L\in[0,1]$ be arbitrary.
	Moreover, for any $i\in\{1,2\}$, let $(u_{0,i},v_{0,i}) \in \Wm $ be an arbitrary initial datum,
	and let $(u_i^L,v_i^L,\mu_i^L,\theta_i^L)$ denote the corresponding weak solution of the system \eqref{CH:INT.}. 
	
	Then there exists a positive, non-decreasing function $\Lambda\in C(\RP)$ depending only on $\Omega$, $\beta$, $F$ and $G$, such that for all $L\in[0,1]$,
	\begin{align*}
		\bignorm{\big(u^L_{2},v_2^L\big)(t) - \big(u^L_{1},v_1^L\big)(t)}_{\HH^1} 
		\le \Lambda(t) \bignorm{(u_{0,2},v_{0,2})-(u_{0,1},v_{0,1})}_{(\HH^1)'}\,,
		\quad\text{for all $t\ge 1.$} 
	\end{align*}
\end{lemma}

\begin{proof}
	Recall that the functions $\Lambda_1^*$ in Proposition~\ref{PROP:CD:KLLM} and $\Lambda_1^0$ in Proposition~\ref{PROP:CD:GMS} are independent of $L$. Along with Lemma~\ref{LEM:LBS}, we obtain
	\begin{align*}
		\bignorm{\big(u^L_{2},v_2^L\big)(t) - \big(u^L_{1},v_1^L\big)(t)}_{\HH^1}
		&\le  \max\{\Lambda^*_1(t),\Lambda^0_1(t)\} \, \bignorm{(u_{0,2},v_{0,2})-(u_{0,1},v_{0,1})}_{L,\beta,*}\\
		&\le  C \max\{\Lambda^*_1(t),\Lambda^0_1(t)\} \, \bignorm{(u_{0,2},v_{0,2})-(u_{0,1},v_{0,1})}_{(\HH^1)'},
	\end{align*}
	for all $L\in[0,1]$ and all $t\ge 1$. This directly proves the claim.
\end{proof}

\medskip

Lemma~\ref{LEM:GEN:DISC} and Lemma~\ref{LEM:CLE} can now be used to establish the following result.

\begin{corollary}[Robust exponential attractors for $(\mathcal K_m,\{D^L_m(n)\}_{n\in\N_0})$]
	\label{COR:REA:DIST}
	Suppose that the assumptions \eqref{ass:dom}-\eqref{ass:comp} are satisfied and let $m\in\R$ be arbitrary. 
	
	Then, for every $ L \in[0,1]$, there exists an exponential attractor $ {\mathfrak M}^{L}_D \subset \mathcal K_m$ for the dynamical system $( \mathcal K_{m}, D_{m}^{L}(n))_{n\in\N_0} $ in the sense of Definition~\ref{DEF:EXP:ATT}, where $M=\mathcal K_m$ is to be understood as a metric subspace of $X=(\HH^1)'$.

	Moreover, these exponential attractors can be chosen in such a way that the following properties hold:
	\begin{enumerate}[label = $\mathrm{(\alph*)}$, ref = $\mathrm{(\alph*)}$ ]
		\item The fractal dimension of $ \mathfrak{M}^L_D $ is uniformly bounded, i.e., there exists a constant $c_1\ge 0$ such that for all $L \in[0,1]$: 
		\begin{align}
		\dim_{\mathrm{frac},{(\HH^1)'}}(\mathfrak{M}^L_D)\leq c_{1}.
		\end{align}
		\item There exist constants $c_2>0$ and $0<\gamma<1$ such that for all $L \in[0,1]$,
		\begin{align}
		\dist_{\mathrm{sym},{(\HH^1)'}}(\mathfrak{M}^L_D,\mathfrak M^0_D)\leq c_{2} L^{\gamma/4}.
		\end{align}
		\item There exist constants $c_3,a >0$ such that for all $L \in[0,1]$ and all $n\in\N_0$,
		\begin{align}
			\dist_{{(\HH^1)'}}\big(D^L(n)\mathcal K_m\,,\,\mathfrak M^L_D\big) \le c_3 e^{-a n},
		\end{align}
		meaning that the rate of convergence to these attractors is uniform in $L$.
	\end{enumerate} 
\end{corollary}

\begin{proof}
	To prove the assertion we intend to apply Lemma~\ref{LEM:GEN:DISC} with $Y=\Wm$, $X=(\HH^1)'$, $M=\mathcal K_m$, $\mathcal F^L = D^L_m(1)$ and $\mathcal D^L = D^L_m$. This means that only the conditions (i) and (ii) of Lemma~\ref{LEM:GEN:DISC} need to be verified.
	
	\textit{Condition (i):}
	We first recall that $D^L_m(1) = S^L_m(T_0)$ with $T_0\ge 1$ as constructed in \eqref{EX:T0}. Hence, Lemma~\ref{LEM:CLE} implies that
	\begin{align*}
		\bignorm{D^L_m(1)(u_2,v_2) - D^L_m(1)(u_1,v_1)}_{\HH^1} 
		&\le \Lambda(T_0) \bignorm{(u_2,v_2) - (u_1,v_1)}_{(\HH^1)'}
	\end{align*}
	for all $L\in[0,1]$.
	This verifies condition (i) of Lemma~\ref{LEM:GEN:DISC}.
	
	\textit{Condition (ii):}
	Furthermore, using the continuous embedding $\LL^2\emb (\HH^1)'$ and applying Lemma~\ref{LEM:CONV:GMS} with $T_*=nT_0$, we infer that for all $L\in [0,1]$ and $(u,v)\in \mathcal K_m$,
	\begin{align*}
	&\bignorm{D^L_m(n)(u,v) - D^0_m(n)(u,v)}_{(\HH^1)'} 
	\le C \bignorm{D^L_m(n)(u,v) - D^0_m(n)(u,v)}_{\LL^2} \\
	&\quad \le C e^{Cn T_0}\, L^{\frac{1}{4}}
	\le \big[ (C+1) e^{C T_0} \big]^n \, L^{\frac{1}{4}}.
	\end{align*}
	Here, $C$ is a positive constant independent of $L$.
	Hence, condition (ii) of Lemma~\ref{LEM:GEN:DISC} with $\zeta=1/4$ and $\Theta=(C+1) e^{C T_0}$ is fulfilled.
	
	This allows us to apply Lemma~\ref{LEM:GEN:DISC} as described above which directly proves the assertions and thus, the proof is complete.
\end{proof}

\subsection{Robust exponential attractors for the continuous dynamical system\mbox{}}

We are now ready to present the main result of this section, which is the existence of a robust family of exponential attractors for the continuous dynamical system $(\mathcal K_m,\{S^L_m(t)\}_{t\ge 0})$.

\begin{theorem}[Robust exponential attractors for $(\mathcal K_m,\{S^L_m(t)\}_{t\ge 0})$]
	Suppose that the assumptions \eqref{ass:dom}-\eqref{ass:comp} are satisfied, and let $m\in\R$ be arbitrary. 
	
	Then, for every $ L \in [0,1]$, there exists an exponential attractor $ {\mathfrak M}^{L} \subset \mathcal K_m$ for the dynamical system $(\mathcal K_m,\{S^L_m(t)\}_{t\ge 0})$ in the sense of Definition~\ref{DEF:EXP:ATT}, where $M=\mathcal K_{m}$ is to be understood as a metric subspace of $\HH^1$.
			
	Moreover, these exponential attractors can be chosen in such a way that the following properties hold:
	\begin{enumerate}[label = $\mathrm{(\alph*)}$, ref = $\mathrm{(\alph*)}$ ]
		\item The fractal dimension of $ \mathfrak{M}^L $ is uniformly bounded, i.e., there exists a constant $C_1\ge 0$ such that for all $ L \in [0,1]$: 
		\begin{align}
		\dim_{\mathrm{frac},\HH^1}(\mathfrak{M}^L)\leq C_{1}.
		\end{align}
		\item There exist constants $C_2>0$ and $0<\gamma<1$ such that for all $L\in [0,1]$,
		\begin{align}
		\dist_{\mathrm{sym},\HH^1}(\mathfrak{M}^L,\mathfrak M^0)
		\leq C_{2} \big( L^{1/6} + L^{\gamma/4}).
		\end{align}
		\item There exists constants $C_3,\alpha >0$ such that for all $L\in [0,1]$ and all $t\ge 0$,
		\begin{align}
		\dist_{\HH^1}\big(S^L_m(t)\mathcal K_m,\mathfrak M^L\big) \le C_3\, e^{-\alpha t}, 
		\end{align}
		meaning that the rate of convergence to these attractors is uniform in $L$.
	\end{enumerate}  
\end{theorem}
\begin{proof}
	To establish the assertion, we proceed similarly as in the proof of \cite[Thm.~5.1]{EMZ} (which in turn is based on the proof of the general result \cite[Thm.~3.1]{eden}).
	Let $T_0\ge 1$ denote the number introduced in \eqref{EX:T0}. 
	Since $\mathcal K_m$ is a compact subset of $\HH^1$, we can find a radius $K>0$ (independent of $L$) such that $\norm{(u,v)}_{\HH^1}\le K$ for all $(u,v) \in \mathcal K_m$.
	Let now $C$ denote a generic positive constant, depending only on $T_0$ and $K$ that may change its value from line to line.
	For any $L\in [0,1]$, we set
	\begin{align}
		\label{CONST:MLW}
		\mathfrak M^L := \bigcup_{t \in [T_{0},2T_{0}]} S_{m}^{L}(t) \mathfrak M^L_D.
	\end{align}
	We now intend to show that $\{\mathfrak M^L\}_{L\in[0,1]}$ is a family of exponential attractors with respect to the $\HH^1$-metric which satisfies the assertions (a), (b) and (c).
	
	Invoking the compact Lipschitz estimate from Lemma~\ref{LEM:CLE}, we infer that $\mathfrak M^L$ is compact in $\HH^1$ for every $L\in [0,1]$.
	We next show that the fractal dimension of $\mathfrak M^L$ is bounded uniformly in $L$.
	To this end, we consider the mapping
	\begin{align*}
		\mathcal F^L: [T_0,2T_0] \times \mathcal K_m \to \mathcal K_m,\quad
		(t,u,v) \mapsto S^L(t)(u,v).
	\end{align*}
	It obviously holds that $\mathfrak M^L = \mathcal F^L([T_0,2T_0]\times \mathfrak M^L_D)$.
	Hence, invoking Proposition~\ref{PROP:CD:KLLM}, Proposition~\ref{PROP:HLD:KLLM}, Proposition~\ref{PROP:CD:GMS}, Proposition~\ref{PROP:HLD:GMS} and Lemma~\ref{LEM:CLE}, we infer that there exists some constant $h \ge 0$ depending only on $T_0$ and $K$, such that 
	\begin{align*}
	&\bignorm{ \mathcal F^L(t_1,u_1,v_1) - \mathcal F^L(t_2,u_2,v_2) }_{\HH^1} 
	= \bignorm{ S_{m}^{L}(t_{1})(u_1,v_1)-S_{m}^{L}(t_{2})(u_2,v_2) }_{\HH^1} \\
	&\quad \le C\big( \bignorm{(u_1,v_1)-(u_1,v_2)}_{\HH^1}
		+ \left|t_{1}-t_{2} \right|^{\frac{3}{16}}\big) 
	\le h \bignorm{(t_1,u_1,v_1)-(t_2,u_1,v_2)}_{\R\times\HH^1}^{\frac{3}{16}} 	
	\end{align*}
	for all $t_1,t_2 \in [T_0,2T_0]$ and all $(u_1,v_1),(u_2,v_2)\in \mathcal K_m$, where $\R$ is endowed with the Euclidean norm. This means that $\mathcal F^L$ is Hölder continuous with exponent $3/16$ and Hölder constant $h$.
	
	\pagebreak[2]
	
	Note that $h r^{3/16} \le r^{1/8}$ if $r>0$ is sufficiently small. Hence, recalling the definition of the fractal dimension (see \eqref{DEF:DIMFRAC}), we obtain 
	\begin{align*}
		&\dim_{\mathrm{frac},\HH^1}\big(\mathfrak M^L\big)
		= \dim_{\mathrm{frac},\HH^1}\big(\mathcal F^L([T_0,2T_0]\times \mathfrak M^L_D)\big) 
		\\[1ex]
		&\quad = \underset{r\to 0}{\lim\sup}\; \frac{\mathcal N_{h r^{3/16}}\big(\mathcal F^L([T_0,2T_0]\times \mathfrak M^L_D)\,;\,\HH^1\big)}{-\ln\big(h\, r^{3/16} \big)}
		\\[1ex]
		&\quad \le  \underset{r\to 0}{\lim\sup}\; \frac{\mathcal N_{r}\big([T_0,2T_0]\times \mathfrak M^L_D\,;\,\R\times\HH^1\big)}{-\ln\big(r^{1/8} \big)}
		\\[1ex]
		&\quad = 8 \dim_{\mathrm{frac},\R\times \HH^1}
			\big([T_0,2T_0]\times \mathfrak M^L_D\big) 
		\\
		&\quad \le 8 \big( \dim_{\mathrm{frac},\HH^1}
		\big(\mathfrak M^L_D\big) + 1 \big)
		\le 8\big(c_1 + 1\big),
	\end{align*}
	where $c_1>0$ is the constant from Corollary~\ref{COR:REA:DIST}(a). This verifies (a).
	
	Proceeding as in the proof of \cite[Thm.~3.1]{eden}, it is straightforward to check that the set $\mathfrak M^L$ is forward invariant, i.e., $S^L(t) \mathfrak M^L \subset \mathfrak M^L$ for all $t\ge 0$. 
	
	Let now $L\in (0,1]$, $t>0$ and $(u,v)\in\mathcal K_m$ be arbitrary. Without loss of generality, we assume that $t\ge T_0$. Then there exist $s\in[T_0,2T_0)$ and $n\in\N_0$ such that $t = nT_0 + s$. 
	Due to Corollary~\ref{COR:REA:DIST}, and since $\mathfrak M^L_D$ is compact in $(\HH^1)'$, we can further find $(\bar u,\bar v) \in \mathfrak M^L_D$ such that 
	\begin{align*}
		\bignorm{ D^L(n)(u,v) - (\bar u,\bar v) }_{(\HH^1)'} \le c_3 e^{-a n}.
	\end{align*}
	Then, recalling $1\le T_0\le s<2T_0$ and $n=(t-s)/T_0$, we can use the compact Lipschitz estimate from Lemma~\ref{LEM:CLE} to conclude that
	\begin{align*}
		&\bignorm{S^L_t (u,v) - S^L_s(\bar u,\bar v)}_{\HH^1}
			= \bignorm{S^L_s D^L(n) (u,v) - S^L_s(\bar u,\bar v)}_{\HH^1} \\
		&\quad \le C \bignorm{D^L(n) (u,v) - (\bar u,\bar v)}_{(\HH^1)'}
			\le C  e^{-an} \\
		&\quad = C e^{-at/T_0} e^{as/T_0}	
			\le C e^{-\alpha t}
	\end{align*}
	with $\alpha := a/T_0$. Since $S_s^L(\bar u,\bar v) \in \mathfrak M^L$, this implies
	\begin{align*}
		\dist_X\big(S^L_m(t)\mathcal K_m,\mathfrak M^L\big) \le C e^{-\alpha t}
	\end{align*}
	which proves (c). In particular, we have thus shown that for all $L\in [0,1]$, $\mathfrak M^L$ is indeed an exponential attractor for the dynamical system $( \mathcal K_{m}, S_{m}^{L}(t)) $ with respect to the $\HH^1$-metric.
	
	It remains to prove (b). Therefore, let $L\in (0,1]$ and $(u_*,v_*)\in\mathfrak M^L$ be arbitrary. Hence, there exist $t\in[T_0,2T_0]$ and $(u_0,v_0) \in \mathfrak M^L_D$ such that $(u_*,v_*) = S^L(t)(u_0,v_0)$. According to Corollary~\ref{COR:REA:DIST}(b), there exists $(\bar u_0,\bar v_0) \in \mathfrak M^0_D$ such that
	\begin{align*}
		\bignorm{(u_0,v_0) - (\bar u_0,\bar v_0)}_{(\HH^1)'} \le c_2 L^{\gamma/4}.
	\end{align*}
	We now set $(\bar u_*,\bar v_*) := S^0(t)(\bar u_0,\bar v_0) \in \mathfrak M^0$.
	Hence, using the Hölder estimate from Proposition~\ref{PROP:CD:GMS}, the convergence estimate from Corollary~\ref{COR:CONV:GMS}, and Lemma~\ref{LEM:LBS}, we conclude that
	\begin{align*}
		&\bignorm{(u_*,v_*) - (\bar u_*,\bar v_*)}_{\HH^1}
			= \bignorm{ S^L(t)(u_0,v_0) - S^0(t)(\bar u_0,\bar v_0)}_{\HH^1} \\
		&\quad \le \bignorm{ S^L(t)(u_0,v_0) - S^0(t)(u_0,v_0) }_{\HH^1}
			+ \bignorm{ S^0(t)(u_0,v_0) - S^0(t)(\bar u_0,\bar v_0) }_{\HH^1} \\
		&\quad \le \bignorm{ S^L(t)(u_0,v_0) - S^0(t)(u_0,v_0) }_{\HH^1}
			+ C \bignorm{ (u_0,v_0) - (\bar u_0,\bar v_0) }_{(\HH^1)'} \\	
		&\quad \le C L^{1/6} + C L^{\gamma/4}.
	\end{align*}
	This directly implies
	\begin{align*}
		\dist_{\HH^1}\big(\mathfrak M^L,\mathfrak M^0\big) 
		\le C \big( L^{1/6} + L^{\gamma/4} \big).
	\end{align*}
	Proceeding analogously, we derive the same estimate for $\dist_{\HH^1}(\mathfrak M^0,\mathfrak M^L) $. Combining both estimates, we obtain (b).
	Thus, the proof is complete. 
\end{proof}

\section*{Acknowledgement} Sema Yayla was supported by the Scientific and Technological Research Council of Turkey (TUBITAK). Harald Garcke and Patrik Knopf were partially supported by the RTG 2339 ``Interfaces, Complex Structures, and Singular Limits''
of the German Science Foundation (DFG). The support is gratefully acknowledged.

%
%

\footnotesize

\bibliographystyle{plain}
\bibliography{GKY}

\end{document}